\documentclass{amsart}

\usepackage{amsmath,amsthm,amssymb,amscd}
\usepackage{tikz}
\usepackage[all,cmtip,color]{xy}
\usepackage[german,english]{babel}
\usepackage{amsmath}
\usepackage{hyperref}
\usepackage{mathrsfs} 
\usepackage{todonotes}
\usepackage[all,cmtip]{xy}
\usepackage[margin = 2.7cm]{geometry}
\usepackage{enumerate}
\usepackage{mathtools}
\usepackage{wasysym}
\usepackage{colonequals}
\usepackage[export]{adjustbox}
\usepackage{graphicx}
\usepackage{wrapfig}
\usepackage{units}

\newcommand{\defeq}{\colonequals}

\DeclareMathOperator{\Hom}{\mathsf{Hom}}

\DeclareMathOperator{\Bun}{\mathsf{Bun}}

\DeclareMathOperator{\codim}{codim}

\DeclareMathOperator{\Hb}{\mathbb{H}}
\DeclareMathOperator{\Nm}{\mathsf{Nm}}
\DeclareMathOperator{\Pic}{\mathsf{Pic}}
\DeclareMathOperator{\Tr}{\mathsf{Tr}}
\DeclareMathOperator{\Fr}{\mathsf{Fr}}

\DeclareMathOperator{\Nb}{\mathbf{N}}

\DeclareMathOperator{\orb}{\text{orb}}

\DeclareMathOperator{\Rb}{\mathbf{R}}
\DeclareMathOperator{\Cb}{\mathbb{C}}
\DeclareMathOperator{\Fb}{\mathbb{F}}
\DeclareMathOperator{\Xc}{\mathcal{X}}

\newcommand{\Cc}{\mathcal{C}}
\newcommand{\Gb}{\mathbb{G}}

\newcommand{\Gg}{\mathcal{G}}

\DeclareMathOperator{\et}{\text{\'et}}
\DeclareMathOperator{\Zar}{\mathsf{Zar}}

\renewcommand{\lim}{\mathsf{lim}}

\DeclareMathOperator{\coker}{coker}

\DeclareMathOperator{\Pc}{\mathcal{P}}

\DeclareMathOperator{\Br}{\mathsf{Br}}
\DeclareMathOperator{\inv}{\mathsf{inv}}
\DeclareMathOperator{\Pb}{\mathbb{P}}

\DeclareMathOperator{\id}{id}

\DeclareMathOperator{\Mm}{\mathcal{M}}

\DeclareMathOperator{\Aa}{\mathcal{A}}

\DeclareMathOperator{\Qb}{\mathbb{Q}}
\DeclareMathOperator{\Zb}{\mathbb{Z}}
\DeclareMathOperator{\st}{\mathsf{st}}

\DeclareMathOperator{\Ab}{\mathbb{A}}

\DeclareMathOperator{\Coh}{\mathsf{Coh}}

\DeclareMathOperator{\F}{\mathcal{F}}

\DeclareMathOperator{\Gal}{Gal}
\DeclareMathOperator{\Aut}{\mathsf{Aut}}

\DeclareMathOperator{\G}{\mathbb{G}}

\DeclareMathOperator{\GL}{GL}
\DeclareMathOperator{\PGL}{PGL}
\DeclareMathOperator{\SL}{SL}

\DeclareMathOperator{\X}{\mathcal{X}}

\DeclareMathOperator{\Z}{\mathcal{Z}}

\DeclareMathOperator{\vol}{\mathsf{vol}}
\DeclareMathOperator{\Spec}{Spec}

\DeclareMathOperator{\End}{\mathsf{End}}
\DeclareMathOperator{\Eend}{\underline{\End}}
\DeclareMathOperator{\Hhom}{\underline{\Hom}}
\DeclareMathOperator{\Oo}{\mathcal{O}}

\renewcommand{\top}{\text{top}}
\DeclareMathOperator{\ord}{ord}

\DeclareMathOperator{\Ext}{Ext}
\DeclareMathOperator{\Split}{Split}
\let\into\hookrightarrow

\DeclareMathOperator{\pr}{pr}

\newcommand{\iso}{\text{iso}}
\newcommand{\I}{I\!}

\AtBeginDocument{%
   \def\MR#1{}
}

\begin{document}

\theoremstyle{definition}
\newtheorem{definition}{Definition}[section]
\newtheorem{rmk}[definition]{Remark}

\newtheorem{example}[definition]{Example} 

\newtheorem{ass}[definition]{Assumption}
\newtheorem{warning}[definition]{Dangerous Bend}
\newtheorem{porism}[definition]{Porism}
\newtheorem{hope}[definition]{Hope}
\newtheorem{situation}[definition]{Situation}
\newtheorem{construction}[definition]{Construction}

\theoremstyle{plain}

\newtheorem{theorem}[definition]{Theorem}
\newtheorem{proposition}[definition]{Proposition}
\newtheorem{corollary}[definition]{Corollary}
\newtheorem{conj}[definition]{Conjecture}
\newtheorem{lemma}[definition]{Lemma}
\newtheorem{claim}[definition]{Claim}
\newtheorem{cl}[definition]{Claim}

\title{Mirror symmetry for moduli spaces of Higgs bundles via p-adic integration}
\author{Michael Groechenig \and Dimitri Wyss \and Paul Ziegler}
\address{Department of Mathematics, University of Toronto}
\email{michael.groechenig@utoronto.ca}
\address{	Mathematical Institute, University of Oxford}
\email{paul.ziegler@maths.ox.ac.uk}
\address{institut de Math\'ematiques de Jussieu - Paris Rive Gauche}
\email{dimitri.wyss@imj-prg.fr}
\begin{abstract}We prove the Topological Mirror Symmetry Conjecture by Hausel--Thaddeus for smooth moduli spaces of Higgs bundles of type $SL_n$ and $PGL_n$. More precisely, we establish an equality of stringy Hodge numbers for certain pairs of algebraic orbifolds generically fibred into dual abelian varieties. Our proof utilises p-adic integration relative to the fibres, and interprets canonical gerbes present on these moduli spaces as characters on the Hitchin fibres using Tate duality. Furthermore, we prove for $d$ prime to $n$, that the number of rank $n$ Higgs bundles of degree $d$ over a fixed curve defined over a finite field, is independent of $d$. This proves a conjecture by Mozgovoy--Schiffmann in the coprime case.\end{abstract}
\thanks{
M. G. was supported by a Marie Sk\l odowska-Curie fellowship: This project has received funding from the European Union's Horizon 2020 research and innovation programme under the Marie Sk\l odowska-Curie Grant Agreement No. 701679.
D. W. was supported by the ERC and the Swiss National Science Foundation: This work was supported by the the European Research Council [No. 320593]; and the Swiss National
Science Foundation [No. 153627]. P.Z. was supported by the Swiss National Science Foundation. This research was supported through the program ``Research in Pairs" by the Mathematisches Forschungsinstitut Oberwolfach in 2016. \\ \includegraphics[height=1cm,right]{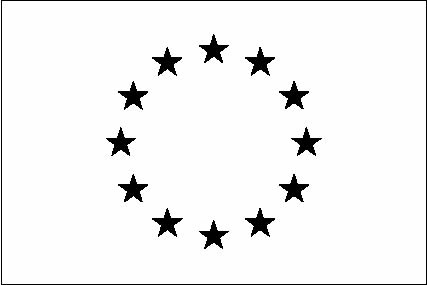}
}

\maketitle
\tableofcontents

\section{Introduction}

Moduli spaces of Higgs bundles are known for their rich and intricate geometry. As manifolds they are distinguished by the presence of a hyperk\"ahler structure; moreover, they admit a completely integrable system. The latter is in fact defined as a morphism of complex algebraic varieties and is referred to as the \emph{Hitchin map}. It yields a fibration of the moduli space whose generic fibres are abelian varieties. Furthermore, even though these complex varieties are not projective, the cohomology of smooth moduli spaces of Higgs bundles is pure. 

In \cite{MR2653248} Ng\^o exploited these properties to prove the Fundamental Lemma in the Langlands Programme. His proof utilises the aforementioned purity of cohomology, and foremost natural symmetries of Hitchin fibres, and connects them to the arithmetic phenomena of \emph{stabilisation} and \emph{endoscopy}. Our article reverses the flow of these ideas. We use the arithmetic of abelian varieties to compare topological and complex-analytic invariants of moduli spaces of Higgs bundles for different structure groups.

Higgs bundles on a smooth complete curve $X$ (or compact Riemann surface) are given by a pair $(E,\theta)$, where $E$ is a principal $G$-bundle and $\theta$ is an additional structure known as Higgs field. The geometric features of the moduli spaces mentioned above are intimately connected with representation theory and arithmetic. For $G$ and $G^{L}$ two Langlands dual reductive groups, the Hitchin fibrations share the same base, and the generic fibres are dual in the sense of abelian varieties. This was observed by Hausel--Thaddeus \cite{MR1990670} in the case of $\SL_n$ and $\PGL_n$, and for general pairs of Langlands dual reductive groups this is a theorem by Donagi--Pantev \cite{MR2957305}. Inspired by the SYZ philosophy, Hausel--Thaddeus conjectured that the moduli spaces of $\SL_n$ and $\PGL_n$-Higgs bundles are mirror partners, and predicted an agreement of appropriately defined Hodge numbers. We prove this conjecture.

Let $n$ be a positive integer, and $d$, $e$ two integers prime to $n$. We choose a line bundle $L \in \mathsf{Pic}(X)$ of degree $d$, and denote by $\Mm_{\SL_n}^L$ the moduli space of Higgs bundles $(E,\theta)$, where $E$ is a vector bundle of rank $n$ together with an isomorphism $\det(E) \simeq L$, and $\theta$ is trace-free. We let $\Mm_{\PGL_n}^e$ be the moduli space of families of $\PGL_n$-Higgs bundles, which admit over geometric points a reduction of structure group to a $\GL_n$-Higgs bundle of degree $e$. Moreover, there exists a natural unitary gerbe on $\Mm_{\PGL_n}^e$, which we denote by $\alpha_L$ \cite[Section 3]{MR1990670}.

\begin{theorem}[Topological Mirror Symmetry Conjecture of \cite{MR1990670}]\label{thm:main}
We have an equality of (stringy) Hodge numbers $h^{p,q}(\Mm_{\SL_n}^L) = h^{p,q}_{\mathsf{st}}(\Mm_{\PGL_n}^e,\alpha_L)$. 
\end{theorem}

The coprimality assumption on $d$ and $e$ with respect to $n$ ensures that the notion of stability and semi-stability coincide. The resulting $\SL_n$-moduli space $\Mm_{\SL_n}^L$ is smooth, while $\Mm_{\PGL_n}^e$ has finite quotient singularities. We use $h^{p,q}_{\mathsf{st}}$ to denote \emph{stringy} Hodge numbers. These are numerical invariants introduced by Batyrev \cite{Ba99}, which include appropriate correction terms to compensate for the presence of singularities. In addition, the gerbe $\alpha$ living on these spaces needs to be taken into account. This is natural from the point of view of duality of the Hitchin fibres.
The proof of this result proceeds by proving an equality for stringy point-counts over finite fields first, by means of $p$-adic integration. We then use $p$-adic Hodge theory to deduce the topological assertion from the arithmetic one. The details will be given in \ref{conclusion}.

Our methods are general enough to be applicable beyond the $\SL_n/\PGL_n$-case. In fact we prove an equality of appropriately defined Hodge numbers for any ``dual pair of abstract Hitchin systems" (see Theorem \ref{Mirror}). We refer the reader to Definition \ref{situation} for a detailed account of what this means. At the heart of the concept lies a pair of maps $(\Mm_1 \xrightarrow{\pi_1} \Aa \xleftarrow{\pi_2} \Mm_2)$ of complex algebraic orbifolds, where the base $\Aa$ is assumed to be smooth, and contains an open dense subset $\Aa^{\lozenge}$, such that over this open subset there exist families of abelian varieties $\Pc_1^{\lozenge} \rightarrow \Aa^{\lozenge} \leftarrow \Pc^{\lozenge}_2$, which act faithfully and transitively on the fibres of $\Mm_i$. Moreover, we assume that $\Pc_1^{\lozenge}$ and $\Pc_2^{\lozenge}$ are dual in the sense of abelian varieties. There are further technical conditions that are omitted for the sake of brevity. They guarantee that $\Mm_1$ and $\Mm_2$ are \emph{minimal}, and hence enable us to compare topological invariants. These conditions are modelled on the same structural properties of the Hitchin map fundamental to \cite{MR2653248}.

Finally, inspired by the set-up of \cite[Section 3]{MR1990670}, we consider gerbes $\alpha_i \in H^2_{\text{\'et}}(\Mm_i,\mu_r)$ satisfying the following condition, which extends the fibrewise duality of the abelian varieties $\Pc_1^{\lozenge,\vee} \simeq \Pc_2^{\lozenge}$ to the torsors $\Mm^{\lozenge}_i$.

\begin{definition}
We say that the pair $(\Mm_1,\alpha_1)$ is dual to $(\Mm_2,\alpha_2)$, if we have canonical equivalences
$\Mm^{\lozenge}_1 \simeq \mathsf{Split}'(\Mm_2^{\lozenge}/\Aa^{\lozenge},\alpha_2),$ and $\Mm_2^{\lozenge} \simeq \mathsf{Split}'(\Mm_1^{\lozenge}/\Aa^{\lozenge},\alpha_2)$,
where $\mathsf{Split}'$ denotes the \emph{principal component} of the stack of fibrewise splittings of a gerbe, as defined in Definition \ref{principalcomponent}.
\end{definition}

It is for systems satisfying these conditions that we prove our main mirror symmetry result. At first we need to recall the definition of the $E$-polynomial (or Serre characteristic). For a smooth complex projective variety $X$ this is defined to be the polynomial
$$E(X;x,y) = \sum_{p,q \in \Nb} (-1)^{p+q}h^{p,q}(X)x^p y^q.$$
There exists a unique extension of the $E$-polynomial to arbitrary complex varieties, such that 
$$E(X;x,y) = E(X \setminus Z;x,y) + E(Z;x,y)$$
for every closed subvariety $Z\subset X$ (see \cite[Definition 2.1.4]{MR2453601}).

Suppose that $\Gamma$ is a finite group acting generically fixed point free on a complex quasi-projective variety $X$, such that for every $\gamma \in \Gamma$ the fixpoint set $X^{\gamma}$ is connected (for simplicity). We define the \emph{stringy $E$-polynomial} as
$$E_{\mathsf{st}}([X/\Gamma];x,y) = \sum_{\gamma \in \Gamma/\text{conj}} E(X^{\gamma}/C(\gamma);x,y) (xy)^{F(\gamma)},$$
where $C(\gamma)$ denotes the centraliser of $\gamma$, and $F(\gamma)$ denotes the so-called \emph{fermionic shift}. We refer the reader to Definition \ref{fermionicshift} for more details.

Furthermore, a $\mu_r$-gerbe $\alpha$ on $[X/\Gamma]$ gives rise to a modified invariant $E_{\mathsf{st}}([X/\Gamma],\alpha;x,y)$. We refer the reader to  Definition \ref{defi:Ealpha}.

\begin{theorem}[Topological Mirror Symmetry, c.f. Theorem \ref{Mirror}]
Let $(\Mm_i \to \Aa, \alpha_i)$ be a dual pair of abstract Hitchin systems in the sense of Definition \ref{situation} over a ring of finite type over $\Zb$. Then we have the equality of stringy $E$-polynomials
$E_{\mathsf{st}}(\Mm_1,\alpha_1;x,y) = E_{\mathsf{st}}(\Mm_2,\alpha_2;x,y)$.
\end{theorem}

By a ``stringy version'' of the Weil conjectures, the dimensions of the stringy cohomology groups appearing here are governed by ``stringy point-counts'' over finite fields (c.f. \ref{stringyinvariants}). As mentioned above we will prove our mirror symmetry result by a comparison of stringy point-counts of similar varieties over finite fields.

\subsubsection*{Strategy}

Our approach to Theorem \ref{thm:main} is strongly inspired by Batyrev's proof of the following result (see \cite{batyrev1999birational}):

\begin{theorem}[Batyrev]\label{batyrev}
Let $X$ and $Y$ be smooth projective birational Calabi-Yau varieties over the field of complex numbers. Then $X$ and $Y$ have equal Betti numbers.
\end{theorem}

Batyrev's proof uses the fact that it suffices to compare the point-counts of two such varieties over finite fields by virtue of the Weil Conjectures. Using standard reduction steps one is led to the following set-up: Let $F/\mathbb{Q}_p$ be a local field with ring of integers $\Oo_F$ and residue field $\mathbb{F}_q$. We have smooth and projective Calabi-Yau schemes $X$ and $Y$ over $\Oo_F$ together with a birational transformation $X \dashrightarrow Y$ over $\Oo_F$. Batyrev then invokes a result of Weil \cite[Theorem 2.2.5]{weil2012adeles}, which asserts that the set of $\Oo_F$-integral points of $X$ has a canonical measure satisfying
\begin{equation}\label{eqn:weil}\vol\left(X(\Oo_F)\right) = \frac{\#X(\mathbb{F}_q)}{q^{\dim X}}.\end{equation}
It is therefore sufficient to prove that $X(\Oo_F)$ and $Y(\Oo_F)$ have the same volume. Since the canonical measure can be described in terms of a non-vanishing top-degree form on $X$, and the two varieties are isomorphic up to codimension $2$, the equality of the volumes is remarkably easy to prove.

The first piece of evidence that a similar strategy can be applied to the Hausel--Thaddeus conjecture is provided by the fact that the $p$-adic volume of an orbifold over $\Oo_F$ is related to the stringy point-count.

\begin{theorem}[Theorem \ref{thm:appendix}]
Let $X$ be a smooth scheme over the ring of integers $\Oo_F$ of a local field $F$ with residue field $\mathbb{F}_q$. Assume that an abstract finite abelian group $\Gamma$ acts generically fixed point-free on $X$ preserving the canonical line bundle and that $F$ contains all roots of unity of order $|\Gamma|$. Then we have
$$\vol\left((X/\Gamma)(\Oo_F)\right) = \frac{\#_{\mathsf{st}}[X/\Gamma](\mathbb{F}_q)}{q^{\dim X}},$$
with respect to the canonical orbifold measure on $X/\Gamma$.
\end{theorem}

This is well-known to experts, and various versions exist in the literature \cite{DL2002, Yasuda:2014aa}. However we were unable to find a reference that can be applied directly in our context, particularly concerning the calculation of the fermionic shift in measure-theoretic terms. For this reason we have included a proof in Section \ref{sectorbm}.

The stringy point-count of $X/\Gamma$ twisted by a gerbe $\alpha$ can also be computed in this manner: As we explain in Definition \ref{fi}, the gerbe $\alpha$ induces a measurable function $f_{\alpha}$ on $(X/ \Gamma)(\Oo_F)$ (defined almost everywhere) whose integral determines the stringy point-count twisted by $\alpha$. Analogously to Batyrev's proof of Theorem \ref{batyrev}, using this one deduces Theorem \ref{thm:main} from the following equality:

\begin{theorem}[TMS for $p$-adic integrals, c.f. Theorem \ref{p-adic-Mirror}]\label{thm:p-adic}

Let $F$ be a local field and $(\Mm_i \to \Aa, \alpha_i)$ a dual pair of abstract Hitchin systems over $F$. Then we have the equality
$$\int_{\Mm_1 (\Oo_F)} f_{\alpha_1} d\mu_{\text{orb}}= \int_{\Mm_2(\Oo_F)} f_{\alpha_2} d\mu_\text{orb}.$$
\end{theorem}

The proof proceeds by evaluating both sides fibre-by-fibre along $\Mm_i\to \Aa$  using relative measures. After discarding a subset of measure zero we may only work with the fibres above those rational points $a \in \Aa(\Oo_F)$, which belong to $\Aa^{\lozenge}(F)$ as well. That is we only have to analyse the fibrewise version of the identity above over torsors of abelian varieties defined over $F$. Using a reinterpretation of the functions $f_{\alpha_i}$ in terms of Tate Duality one can show that on a single fibre these functions are either constant of value $1$ or exhibit character-like behaviour. This leads to the cancellation of the corresponding integral. The latter guarantees that only a contribution of those fibres which can be matched by an equal contribution on the other side survives.

A similar idea also applies to a conjecture of Mozgovoy--Schiffmann on the number of points of $\Mm_{\GL_n}^d$, the moduli space of semi-stable $\GL_n$-Higgs bundles of fixed degree $d$.
\begin{theorem}[Theorem \ref{mozschiff}]
Let $d,e$ be positive integers coprime to $n$ and $X/\mathbb{F}_q$ a smooth proper curve. Then we have 
\[\#\Mm_{\GL_n}^d(X)(\mathbb{F}_q) = \#\Mm_{\GL_n}^{e}(X)(\mathbb{F}_q).\]
\end{theorem}

There are independent proofs of this theorem by Yu \cite{YH18}, and by Mellit \cite{mellit2017poincar} without the coprimality assumption. 

\subsubsection*{Previously known cases of TMS}

Theorem \ref{thm:main} was conjectured by Hausel--Thaddeus \cite{MR1990670}. They provided ample evidence for their prediction, particularly a full proof for the cases $n = 2$ and $n=3$. 

In \cite{MR2967059} Hausel--Pauly analysed the group of connected components of Prym varieties. They show that the finite group $\Gamma = J[n]$ acts trivially on the cohomology up to the degree predicted by Theorem \ref{thm:main}.

In \cite[Section 5.3]{hausel2011global} Hausel observes that Ng\^o's \cite[Theorem 6.4.1]{MR2653248} should imply a fibrewise version of Theorem \ref{thm:main}, over an open dense subset of the Hitchin base.

Finally we remark that there is also an analogue of the Topological Mirror Symmetry Conjecture for parabolic Higgs bundles. The paper \cite{biswas2012syz} by Biswas--Dey verified that also in the parabolic case the moduli spaces for the $\SL_n$ and $\PGL_n$-case are generically fibred into dual abelian varieties. This motivates study of the stringy cohomology of moduli spaces of parabolic Higgs bundles. There is a certain family of examples of such spaces, which can be described in terms of Hilbert schemes of the cotangent bundles of an elliptic curve (see \cite{MR1859601}). Using G\"ottsche's formula for the cohomology of Hilbert schemes one can verify the conjecture for these cases, as has been explained to us by Boccalini--Grandi \cite{bg}. Recently, the rank $2$ and $3$ case of the topological mirror symmetry conjecture for strongly parabolic Higgs bundles was established by Gothen--Oliveira in \cite{gothenoliveira}.

In \cite{hmrv} by Hausel--Mereb--Rodriguez-Villegas the authors will show the analogue of the Topological Mirror Symmetry Conjecture for $E$-polynomials of character varieties. 

\subsubsection*{Conventions}
For $R$ a commutative ring, we will sometimes use the terminology \emph{$R$-variety} to refer to a reduced and separated scheme of finite presentation over $\Spec(R)$.

Some aspects of our work require a choice primitive roots of unity. We therefore assume for the sake of convenience that for every field $F$ appearing in this article, and every positive integer $r$, such that $\mu_r(F)$ has order $r$, we have chosen a primitive root of unity $\zeta_r$ of order $r$. Furthermore, we assume that these choices are compatible, that is, satisfy $\zeta_{rr'}^{r} = \zeta_{r'}$. If there is no risk of confusion, we will drop the subscript and simply write $\zeta = \zeta_r$ for a fixed positive integer $r$.

\subsubsection*{Acknowledgements}
A substantial part of the research described in this paper was conducted during a ``Research in Pairs" stay at the \textit{Mathematisches Forschungsinstitut Oberwolfach}. We thank the MFO for its hospitality and the excellent working conditions. We thank Tamas Hausel for the encouragement we received while working on this article. We thank the participants of the Intercity Geometry Seminar, in particular Peter Bruin, Bas Edixhoven, Chris Lazda, and David Holmes, for alerting us about incorrect statements in sections 2 and 3 of the first version of this paper and pointing out numerous typos. We thank the anonymous referee for helpful comments and suggestions, which improved the exposition substantially. 
 It is also a pleasure to thank Johannes Nicaise, Ng\^o B\`ao Ch\^au, Benjamin Antieau, Jacques Hurtubise, Ezra Getzler and H\'el\`ene Esnault for interesting conversations on the subject of this article.


\section{Stringy cohomology and gerbes}

\subsection{Stringy invariants}


We will mostly consider stringy invariants of varieties, which admit a presentation as a global quotient $Y/\Gamma$, where $Y$ is a smooth variety, and $\Gamma$ a finite abstract group. A priori, the invariants will depend on this presentation, but can be shown to be a well-defined invariant of the quotient stack $[Y/\Gamma]$. For more details on this, as well as a treatment of stringy Hodge numbers for more general classes of singular varieties, we refer the reader to Batyrev's \cite{Batyrev1996}. 

\begin{definition}\label{definition:restrictions}
Let $\Xc$ be a Deligne-Mumford stack. We say that $\Xc$ 
\begin{itemize}
\item[(a)] is a \emph{finite quotient stack}, if there exists an algebraic space $Y$ with a generically fixed-point free action of an abstract finite group $\Gamma$ such that $\Xc \simeq [Y/\Gamma]$.

\item[(b)] is a \emph{finite abelian quotient stack}, if (a) holds, and the group $\Gamma$ can furthermore be assumed to be abelian.
\end{itemize}
\end{definition}

\begin{definition}\label{fermionicshift}
Let $\Xc= [Y/\Gamma]$ be a smooth finite quotient stack over a field $k$ as in Definition \ref{definition:restrictions}, such that $|\Gamma|$ is invertible in $k$. Fix an algebraic closure $\bar k$ of $k$. Let $x \in Y$ be a closed point fixed by an element $\gamma \in \Gamma$. The tangent space $T_xY$ inherits therefore a representation by the finite cyclic group $I = \langle \gamma \rangle$ generated by $\gamma$. Over $\overline{k}$ we choose a basis of eigenvectors $(v_1,\dots,v_k)$ and denote by $(\zeta_1,\dots,\zeta_k)$ the corresponding list of eigenvalues.

Let $\zeta$ be our fixed primitive root of order $r = \mathsf{ord}(\gamma)$ in $\bar k$. For each eigenvalue $\zeta_i$ there exists a unique expression $\zeta_i = \zeta^{c_i}$ with $0\leq c_i < r$. With respect to this choice we define the fermionic shift of $\gamma$ at $x$ to be the sum of fractions 
\[F(\gamma,x) = \sum_{i = 1}^k \frac{c_i}{r}.\]
This number is locally constant on $Y^{\gamma}$, and therefore defines a function on $\pi_0(Y^{\gamma})$. Furthermore $F(\gamma, \cdot)$ is constant on $C(\gamma)$-orbits in $\pi_0(Y^\gamma)$ , where $C(\gamma) \subset \Gamma$ denotes the centraliser of $\gamma$. Hence we obtain a function
\[F(\gamma,\cdot): \pi_0([Y^{\gamma}/C(\gamma)]) \to \Qb. \]

There is also a version of the fermionic shift in the literature, where the numbers $c_1,\dots,c_k$ are chosen to satisfy $0 < c_i \leq r$ (c.f. e.g. \cite{DL2002,lo02,Ya06}). We write 
\[w(\gamma,\cdot):  \pi_0([Y^{\gamma}/C(\gamma)]) \to \Qb, \]
for the corresponding locally constant function. For any connected component $\Z \in \pi_0([Y^\gamma/C(\gamma)])$ the two functions $F,w$ are related by the formulas $F(\gamma,\Z) = w(\gamma,\Z)-\dim \Z$ and $w(\gamma,\Z) = \dim \Xc-F(\gamma^{-1},\Z)$.\\
\end{definition}

We reiterate that the definition of the fermionic shift depends on the choice of a primitive root of unity $\zeta$ of order $r$. Over the field of complex numbers it is standard to choose $\zeta = e^{\frac{2\pi i}{r}}$, but for all other fields this choice is a recurring aspect of our work. 
For an element $\gamma$ of a group $\Gamma$ we will denote by $C(\gamma)$ the centraliser of $\gamma$.

\begin{rmk} For a groupoid $A$ (typically the $k$-points of $\X$), we will denote by $A_\iso$ the set of isomorphism classes of $A$. If $A_\iso$ is finite we write $\# A$ for the mass of $A$, that is
\[ \#A =  \sum_{a \in A_\iso} \frac{1}{|\Aut(x)|}.    \]
\end{rmk}

\begin{definition}\label{stringyinvariants}
Let $\Xc$ be a smooth finite quotient stack over a field $k$. Choose a presentation $\Xc=[Y/\Gamma]$ as in Definition \ref{definition:restrictions} with $V$ smooth as well as a primitive root of unity $\xi \in \bar k$ of order $|\Gamma|$.
\begin{itemize}
\item[(a)] If $k = \Cb$ is the field of complex numbers, we denote by $E_{\mathsf{st}}(\Xc;x,y)$ the polynomial 
  \begin{equation*}
    E_{\mathsf{st}}(\Xc;x,y) = \sum_{\gamma \in \Gamma/\text{conj}} \left( \sum_{\Z \in \pi_0(Y^{\gamma}/C(\gamma))}E(\Z;x,y)(xy)^{F(\gamma,\Z)}\right),
  \end{equation*}
 where we define for $\Z = [W/ C(\gamma)]$
$$E(\Z;x,y)=\sum_{p,q,k}(-1)^k \dim (H_c^{p,q;k}(W)^{C(\gamma)})x^py^q.$$
Here $H_c^{p,q;k}(W)$ denotes the space $\operatorname{gr}^W_{p+q}H_c^k(W)^{p,q}$ given by the mixed Hodge structure on the compactly supported cohomology of $W$.
\item[(b)] If $k = \Fb_q$ is a finite field, we denote by $\#_{\mathsf{st}}(\Xc)$ the sum 
\begin{equation*}
  \#_{\mathsf{st}}(\Xc) = \sum_{\gamma \in \Gamma/\text{conj}} \left( \sum_{\Z \in \pi_0(Y^{\gamma}/C(\gamma))} q^{F(\gamma,\Z)}\#\Z(k)\right) \in \Rb.
\end{equation*}

\end{itemize}
\end{definition}
The fermionic shift $F(\gamma,Z)$ depends on the choice of $\zeta$. These stringy invariants are however independent of it: Let $\zeta'$ be another choice, and $F'(\gamma,Z)$ the resulting shift. There exists an integer $k$, prime to $|\Gamma|$, such that $\zeta' = \zeta^k$. Elementary group theory shows that $\gamma^k$ is also a generator of the finite cyclic group $I$ generated by $\gamma$. Therefore we have $F(\gamma,Z) = F' (\gamma^k,Z)$ and the above sums remain the same.
One can also check that these definitions do not depend on the choice of presentation of $\Xc$.

\begin{rmk}(Inertia stacks) \label{inerst} Recall that for any stack $\Xc$ the inertia stack $\I\Xc$ is defined to be $\Xc \times_{\Xc \times \Xc} \Xc$. Concretely, for any scheme $S$ the groupoid of $S$-points $\I\Xc(S)$ equals the groupoid of pairs $(x,\alpha)$ where $x \in \Xc(S)$ is an $S$-point of $\Xc$ and $\alpha \in \Aut_{\Xc}(x)$. One can check that for a finite quotient stack $\Xc = [Y/\Gamma]$ one has an equivalence
\[\I\Xc \cong \bigsqcup_{[\gamma] \in \Gamma/\text{conj}} [Y^{\gamma}/C(\gamma)].\]
In particular the Fermionic shift and the weight from Definition \ref{stringyinvariants} can be considered as locally constant functions
\[F,w: \I\Xc(k)_\iso \to \Qb. \]
Over a finite field $k$ one then has
\[\#_{\mathsf{st}}(\Xc) = \sum_{x \in \I\Xc(k)_\iso} \frac{q^{F(x)}}{|\Aut_{\I\Xc(k)}(x)|}. \]
\end{rmk}

In the next subsection we will introduce a variant of this definition, which also depends on a gerbe $\alpha \in H^2([Y/\Gamma],\mu_r)$ on the quotient stack.


\subsection{Gerbes and transgression}\label{gerbes}\label{GerbesSection}

We begin this subsection by recalling terminology from the theory of gerbes. Only gerbes banded by $A=\mu_r$ and $A=\G_m$ will appear in this article. 

\begin{definition}
  Let $\mathcal{S}$ be a Deligne-Mumford stack and $A$ a commutative group scheme over $\mathcal{S}$.
  \begin{enumerate}[(i)]
  \item A \emph{gerbe} $\alpha$ over $\mathcal{S}$ is a morphism of algebraic stacks $\alpha \to \mathcal{S}$ satisfying the following two conditions:
  \begin{itemize}
  \item For any scheme $\mathcal{S}'$ over $\mathcal{S}$ and any two objects $x,x'\in \alpha(\mathcal{S}')$ there exists an \'etale covering $\mathcal{S}''$ of $\mathcal{S}'$ such that $x$ and $x'$ become isomorphic in $\alpha(\mathcal{S}'')$.
  \item There exists an \'etale covering $\mathcal{S}'$ of $\mathcal{S}$ such that $\alpha(\mathcal{S}')$ is not empty.
  \end{itemize}
\item A \emph{banding} of a gerbe $\alpha$ over $\mathcal{S}$ by $A$ consists of isomorphisms $A_{\mathcal{S}'}\cong \underline\Aut_{\mathcal{S}'}(x)$ of \'etale group sheaves over $\mathcal{S}'$ for every $\mathcal S$-stack $\mathcal S'$ and every object $x \in \alpha(\mathcal S')$, which are compatible with pullbacks. A gerbe together with a $A$-banding is called a $A$-gerbe.

  \end{enumerate}
\end{definition}

Descent data for $A$-gerbes are given by $2$-cocycles with values in $A$. For this reason, the set of isomorphism classes of $A$-gerbes is equal to $H^2_{\text{\'et}}(S,A)$.

The reason we care about gerbes is that for $\Xc$ a Deligne-Mumford stack and $\alpha$ an $A$-gerbe on $\Xc$, by the so-called transgression construction, the inertia stack $\I\Xc$ inherits an $A$-torsor $P_{\alpha}$. The formal definition of this torsor uses the functoriality of the inertia stack construction. In the remainder of this subsection we describe three equivalent pictures of $P_{\alpha}$, hoping that at least one of them will appeal to the reader. 

\subsubsection*{An explicit picture for quotient stacks}

We give the first construction of $P_{\alpha}$: For $\Xc = [Y/\Gamma]$ a quotient of a variety by an abstract group $\Gamma$ we have 
$$\I\Xc = \bigsqcup_{\gamma \in \Gamma/\text{conj}} [Y^{\gamma}/C(\gamma)].$$
  Let $A$ be a commutative group scheme. An $A$-gerbe $\alpha$ on $\Xc$ corresponds to a $\Gamma$-equivariant $A$-gerbe $\alpha$ on $Y$. We will summarise the discussion of \cite{MR1990670}, where a $C(\gamma)$-equivariant $A$-torsor is defined on every stratum $Y^{\gamma}$. By descending this torsor one obtains an $A$-torsor $P_{\alpha}$ on $\I\Xc$.

The $\Gamma$-equivariant structure of $\alpha$ is given by equivalences of gerbes $\eta_{\gamma}\colon \gamma^*\alpha \xrightarrow{\simeq} \alpha$ for every $\gamma \in \Gamma$, which satisfy various compatibility and coherence conditions (most of which are not relevant to us). Restricting this equivalence for a given $\gamma \in \Gamma$ to the fixed point locus $Y^{\gamma}$ one obtains an automorphism of $\alpha|_{Y^{\gamma}}$:
\begin{equation}\label{autgerb}\eta_{\gamma}|_{Y^{\gamma}}\colon \alpha|_{Y^{\gamma}} = (\id_{Y^{\gamma}})^* \alpha|_{Y^{\gamma}} = (\gamma|_{Y^{\gamma}})^*\alpha|_{Y^{\gamma}} \xrightarrow{\eta_{\gamma}} \alpha|_{Y^{\gamma}}. \end{equation}
The groupoid of automorphisms of an $A$-gerbe on $Y^{\gamma}$ is equivalent to the groupoid of $A$-torsors on $Y^{\gamma}$. This construction therefore yields an $A$-torsor $P'_{\alpha}$ on $Y^\gamma$, which is $C(\gamma)$-equivariant by virtue of the $\Gamma$-equivariance of $\alpha$. Descent theory yields an $A$-torsor $P_{\alpha}$ on the component $[Y^{\gamma}/C(\gamma)]$ of $\I\Xc$.

Below we record a technical lemma relating central extension and equivariant structures on gerbes. The proof can be skipped when reading this article for the first time.

\begin{lemma}\label{thatsalemma}
In the following we assume that $Y$ is a scheme endowed with the action of an abstract finite group, and that $A\to Y$ is a smooth group scheme.
  \begin{enumerate}[(a)]
  \item The set of isomorphism classes of $\Gamma$-equivariant structures on a trivial $A$-gerbe $\alpha$ on $Y$ is isomorphic to $H^2_\text{sm. grp. $Y$-sch.}(\Gamma,A)$, that is, the set of isomorphism classes of central extensions 
\begin{equation}\label{eqn:cent_ext}1 \to A \to \widehat{\Gamma} \to \Gamma \to 1,\end{equation}
where $A$, $\widehat{\Gamma}$, and $\Gamma$ are viewed as smooth group schemes on $Y$. We denote the corresponding Ext-group by $\Ext^1_\text{sm. ab. grp. $Y$-sch.}(\Gamma,A)$.
\item If the group $\Gamma$ is \emph{cyclic}, then a central extension $\widehat{\Gamma}$ as in (a) is automatically abelian. Hence in this case isomorphism classes of central extensions of $\Gamma$ by $A$ correspond to elements of $\Ext^1_\text{sm. ab. grp. $Y$-sch.}(\Gamma, A)$. 
\item In case $Y=\Spec k$ for $k$ a field with $H^2_{\text{\'et}}(\Spec k, A) =0$, and $\Gamma$ a cyclic group, there is a short exact sequence 
\begin{equation}\label{eqn:rmk}
  H^1_{\et}(k,\Hom_\text{sm. ab. grp. $k$-sch.}(\Gamma,A))\hookrightarrow H^2_{\text{\'et}}([Y/\Gamma],A) \twoheadrightarrow \Ext^1_\text{sm. ab. grp. $\overline{k}$-sch.}(\Gamma,A_{\overline{k}}).
\end{equation}
Furthermore, if $A$ is a constant \'etale group scheme, then this sequence splits.
  \end{enumerate}
\end{lemma}

\begin{proof}
Since $\alpha$ is assumed to be trivial on $Y$, an equivariant structure on an $A$-gerbe $\alpha$ corresponds to a choice of equivalences of $A$-gerbes
$$\eta_\gamma\colon \alpha \simeq \alpha (\simeq \gamma^*\alpha)$$
for every $S$-valued point $\gamma \in Y(S)$ for $S$ a $Y$-scheme, such that $\eta_{\gamma}$ is the identity for $\gamma = e$ the neutral element, and $\eta_\gamma$ is compatible with pullbacks, and for $\gamma_1,\gamma_2 \in \Gamma(S)$
we have a commutative diagram
\[
\xymatrix{
\alpha \ar[r]^{\eta_{\gamma_1}} \ar[rd]_{\eta_{\gamma_2\gamma_1}} & \alpha \ar[d]^{\eta_{\gamma_2}} \\
& \alpha.
}
\]
Furthermore, there is a compatibility condition, which is needed to be satisfied for every triple $\gamma_1,\gamma_2,\gamma_3 \in \Gamma(S)$ (see the commutative diagram below). An equivalence $\alpha \simeq \alpha$ of $A$-gerbes corresponds to an $A$-torsor. We therefore see that a $\Gamma$-equivariant structure on $\alpha$ assigns to every $\gamma \in \Gamma(S)$ an $A$-torsor $L_{\gamma}$ on $S$, such that for the neutral element $e \in \Gamma$ we have that $L_e \simeq A$ is the trivial torsor, and furthermore we have isomorphisms indexed by pairs $(\gamma_1,\gamma_2) \in \Gamma^2$
$$\phi_{\gamma_1,\gamma_2}\colon L_{\gamma_1} \otimes L_{\gamma_2} \simeq L_{\gamma_1\gamma_2},$$
such that for the neutral element $e \in \Gamma$ one has that
$$\phi_{\gamma,e}\colon L_{\gamma} \otimes L_e \simeq L_{\gamma}$$
is compatible with the chosen trivialisation of $L_e$ (and similarly for $\phi_{e,\gamma}$),
and for every triple $\gamma_1,\gamma_2,\gamma_3 \in \Gamma(S)$ we have a commutative diagram
\[
\xymatrix{
(L_{\gamma_1} \otimes L_{\gamma_2}) \otimes L_{\gamma_3} \ar[r]^-{\phi_{\gamma_1,\gamma_2}\otimes \id} \ar[d]_{\text{assoc. constraint}} & L_{\gamma_1\gamma_2} \otimes L_{\gamma_3} \ar[rd]^-{\phi_{\gamma_1\gamma_2,\gamma_3}} \\
L_{\gamma_1}\otimes (L_{\gamma_2} \otimes L_{\gamma_3}) \ar[r]^-{\id \otimes \phi_{\gamma_2,\gamma_3}} & L_{\gamma_1} \otimes L_{\gamma_2\gamma_3} \ar[r]^{\phi_{\gamma_1,\gamma_2\gamma_3}} & L_{\gamma_1\gamma_2\gamma_3}.
}
\]
This amounts to a monoidal map of stacks in groups 
$$\phi\colon \Gamma \to B_YA.$$
Given such a map, we associate to it the central extension
$$1 \to A \to \Gamma \times_{\phi,B_YA} Y \to \Gamma \to 1,$$
where $Y \to B_YA$ is the canonical map to the quotient stack. Vice versa, given a central extension \eqref{eqn:cent_ext}, it is clear that $\widehat{\Gamma} \to \Gamma$ is an $A$-torsor, and hence we obtain a map of stacks $\Gamma \to B_YA$. The central extension property yields that this is a monoidal map of stacks in groups.

It suffices to prove the analogue of Claim (b) for abstract groups, since commutativity of smooth group schemes is a local property. The corresponding statement for extensions of abstract groups can be deduced from the short exact sequence (*) in \cite{prasad}.

Claim (c): to see why \eqref{eqn:rmk} is true we argue as follows. By the assumption on $k$, a $A$-gerbe on $[\Spec k/ \Gamma]$ is the same as a $\Gamma$-equivariant structure on the trivial $A$-gerbe on $\Spec k$. By (a) and (b) such data correspond to elements of $\Ext^1_{\text{sm. ab. grp. sch.}}(\Gamma,A)$. This Ext-group maps to $\Ext^1(\Gamma,A_{\bar k})$ by base change. The kernel of this homomorphism can be identified with the set of isomorphism classes of extensions $\widehat{\Gamma}$, which split when pulled back to the algebraic closure. Since the set of splitting is a torsor under $\Hom(\Gamma,A)$ we obtain a short exact sequence
$$H^1(k,\Hom(\Gamma,A)) \hookrightarrow \mathsf{Ext}^1_{\text{sm. ab. grp. sch.}}(\Gamma,A) \twoheadrightarrow \mathsf{Ext}^1(\Gamma,A_{\bar k}).$$

If $A$ is a constant \'etale group scheme, then the sequence splits by sending an extension of abstract abelian groups to the constant extension of abelian group schemes.
\end{proof}

\subsubsection*{A purist's approach to transgression: central extensions of inertia groups}

Recall that the inertia stack of a Deligne-Mumford stack $\Xc$ is defined to be $\I\Xc = \Xc \times_{\Xc \times \Xc} \Xc$. A more direct definition can be given as follows:
Let us denote by $S$ a \emph{test scheme}. An $S$-point of $\Xc$ is (by definition) the same as a morphism $S \to \Xc$. There exists a group $\Aut_x(\Xc)$, which is equal to the automorphism group of $x$ in the groupoid $\Xc(S)$. The assignment $S \mapsto \{(x,\alpha)|x \in \Xc(S)\text{ and }\alpha \in \Aut_x(\Xc)\}$ is represented by $\I\Xc$. 

The second viewpoint shows that there exists a morphism $\I\Xc \to \Xc$, such that the fibre over a given $S$-point equals the group $S$-scheme denoted by $\underline{\Aut}_x(\Xc)$. Vice versa, one can use the abstract definition of $\I\Xc$ as the self-intersection of the diagonal of $\Xc$ to deduce the existence of a relative group scheme structure on the morphism $\I\Xc \to \Xc$.

We give the second construction of $P_{\alpha}$: Let $A$ be a commutative group scheme and fix an $A$-gerbe $\alpha$ over a Deligne-Mumford stack $\Xc$. An $S$-point $y \in \alpha(S)$ induces an $S$-point $x \in \Xc(S)$. Moreover we obtain a surjective morphism of automorphism groups $\underline{\Aut}_y(\alpha) \twoheadrightarrow \underline{\Aut}_x(\Xc)$. The kernel of this morphism is equal to the group scheme $A$ and is central in $\underline{\Aut}_y(\alpha)$. Therefore, we obtain a central extension 
$$A \hookrightarrow \underline{\Aut}_y(\alpha) \twoheadrightarrow \underline{\Aut}_x(\Xc).$$
The inertia stack of $\Xc$ is fibred in these group schemes over $\Xc$, and the total spaces of these central extensions can be assembled into an $A$-torsor on $\I\Xc$. The technical details are summarised below.

\begin{construction} \label{TransgressionPurist}
A morphism of stacks $\alpha\to \Xc$ induces a morphism of inertia stacks $\I\alpha \to \I\Xc$. For $\alpha$ an $A$-gerbe (where $A$ is assumed to be a flat commutative group scheme), we obtain a canonical morphism $\I\alpha \to \I\Xc$. For example, the trivial $A$-gerbe $B_{\Xc}A$ yields the stack $\I(B_{\Xc}A) = \alpha \times_{\Xc} A$. Since every $A$-gerbe is \'etale locally equivalent to the trivial $A$-gerbe we conclude that $\I\alpha$ is \'etale locally equivalent to $\alpha \times_{\Xc} A$. This shows that $\I\alpha$ is a stack. As above there exists for every $S$-point $Z \in \I\alpha(S)$ a central embedding $A(S) \hookrightarrow \Aut_z(\I\alpha)$. We can therefore apply the \emph{rigidification} process of \cite[Theorem 5.1.5]{abramovich2003twisted} to obtain a stack $\widehat{\I\Xc} = \I\alpha//A$, which gives a central extension of $\Xc$-group schemes 
$$1 \to A \to \widehat{\I\Xc} \to \I\Xc \to 1.$$
Then, $P_{\alpha} = \widehat{\I\Xc}$ is an $A$-torsor on the Deligne-Mumford stack $\I\Xc$. 
\end{construction}

\subsubsection*{A modern viewpoint on transgression}

We give a construction of $P_{\alpha}$ (which we learned from B. Antieau). An $A$-gerbe on $\Xc$ corresponds to a morphism of $2$-stacks $\Xc \to B^2A = B(BA)$. Since the inertia stack $\I(B^2A)$ is equivalent to $BA \times B^2A$ we obtain the morphism
$\I\Xc \to \I B^2A \to BA$, which is the classifying morphism of the $A$-torsor $P_{\alpha}$ on $\I\Xc$.

\subsection{Transgression for torsion $\G_m$-gerbes}

As before we let $\Xc$ be a Deligne-Mumford stack, and let $r$ be a positive integer invertible on $\Xc$. The Kummer sequence
\begin{equation}\label{eqn:kummer_seq}\cdots \to H^1_{\et}(\Xc,\G_m) \xrightarrow{[r]} H^1_{\et}(\Xc,\G_m) \to H^2_{\et}(\Xc,\mu_r) \to H^2_{\et}(\Xc,\G_m) \xrightarrow{[r]} H^2_{\et}(\Xc,\G_m) \to \cdots\end{equation}
implies that $H^2_{\et}(\Xc,\mu_r)$ surjects onto the $r$-torsion subgroup $H^2_{\et}(\Xc,\G_m)[r]$. That is, for every $\G_m$-gerbe $\beta$, such that $\beta^r$ is a neutral gerbe, there exists a $\mu_r$-gerbe $\alpha$, which induces $\beta$ via the embedding $\mu_r \hookrightarrow \G_m$. 

\begin{lemma}\label{lemma:muGm}
Let $Y$ be an irreducible Noetherian scheme endowed with the action of an abstract finite group $\Gamma$, such that $Y$ admits a Zariski-open covering by $\Gamma$-equivariant affine subsets. Let $r$ be a positive integer, such that $\mu_{r\cdot{}|\Gamma|}$ is a constant \'etale group scheme on $Y$. Then, there exist a map 
$$\tau \colon H^2_{\et}([Y/\Gamma],\G_m)[r] \to H^1_{\et}(\I[Y/\Gamma],\mu_r),$$
such that the diagram
\[
\xymatrix{
H^2_{\et}([Y/\Gamma],\mu_r) \ar[rr]^{\alpha \mapsto P_{\alpha}} \ar[d] & & H^1_{\et}(\I[Y/\Gamma],\mu_r) \ar[d] \\
H^2_{\et}([Y/\Gamma],\G_m)[r] \ar[rr]^{\beta \mapsto P_{\beta}} \ar@{-->}[rru] & & H^1_{\et}(\I[Y/\Gamma],\G_m)
}
\]
commutes.
\end{lemma}

\begin{proof}
As we have seen above, it follows from the Kummer sequence that $H^2_{\et}([Y/\Gamma],\mu_r)$ surjects onto $H^2_{\et}([Y/\Gamma],\G_m)[r]$. It is therefore sufficient to prove that for 
$$\alpha \in \ker(H^2_{\et}([Y/\Gamma],\mu_r) \to H^2_{\et}([Y/\Gamma],\G_m)[r])$$
the associated $\mu_r$-torsor on the inertia stack $$P_{\alpha} \in H^1_{\et}(\I [Y/\Gamma],\mu_r)$$ is trivial.

Let $Y = \bigcup_{i \in I} U_i$ be a covering by $\Gamma$-equivariant Zariski-open affine subsets. Since $Y$ is Noetherian, we may assume that $I$ is finite. Furthermore, by further refining this covering, we achieve that the open subsets $U_i$ do not carry non-trivial line bundles. Below, we will prove triviality of $P_{\alpha}|_{\I U_i}$. The general case follows from the following claim.

\begin{claim} 
A Zariski locally trivial $\mu_r$-torsor on $Y$ is trivial.
\end{claim}
\begin{proof}
At first we observe that $\mu_r$ is by assumption a constant \'etale group scheme. In his Tohoku paper, Grothendieck proved a vanishing result for higher cohomology groups $H^i_{\Zar}(Y,A)$ of a constant abelian group valued sheaf on an irreducible Noetherian space $Y$ (see \cite[p. 168]{tohoku} and also \cite[Tag 02UW]{stacks-project}). In particular we have $H^1_{\Zar}(Y,\mu_r) = 0$, which implies triviality of Zariski locally trivial $\mu_r$-torsors.
\end{proof}

Henceforth we may assume without loss of generality that $Y$ is affine and irreducible.

By virtue of assumption, $\alpha$ induces the trivial $\G_m$-gerbe on $[Y/\Gamma]$. It follows from the Kummer sequence that $\alpha = \delta(L)$, where $L \in \Pic([Y/\Gamma]) = H^1_{\et}([Y/\Gamma],\G_m)$ and 
$$\delta\colon H^1_{\et}([Y/\Gamma],\G_m) \to H^2_{\et}([Y/\Gamma],\mu_r)$$
denotes the boundary map of the Kummer sequence \eqref{eqn:kummer_seq}. The $\mu_r$-gerbe $\delta(L)$ measures the obstruction to the existence of an $r$-th root of $L$. As a groupoid-valued functor, it assigns to a test scheme $S \to Y$ the groupoid of pairs $(M,\phi)$, where $M$ is a line bundle on $S$ and $\phi\colon M^{\otimes r} \simeq L$. 

After replacing $Y$ by a Zariski open subset on which $L$ is trivialisable, we see that the gerbe $\alpha=\delta(L)$ splits when pulled back along $Y \to [Y/\Gamma]$, since $L|_Y\simeq\Oo_Y \simeq \Oo_Y^{\otimes r}$. It follows from Lemma \ref{thatsalemma} that $\alpha$ corresponds to a central extension of $Y$-group schemes
\begin{equation}\label{centext}1 \to \mu_{r,Y} \to \widehat{\Gamma}_{\alpha} \to \Gamma \to 1.\end{equation}
\begin{claim}
The central extension \eqref{centext} is Zariski locally induced by a central extension of abstract groups.
\end{claim}
\begin{proof}
By assumption, $\mu_r$ is a constant \'etale group scheme on $Y$, and $\Gamma$ is the constant group scheme. These facts together with connectivity of $Y$, imply that it suffices to prove that $\widehat{\Gamma}_{\alpha}$ is constant.

Recall that $\alpha = \delta(L)$ where $L \in \Pic([Y/\Gamma])$. Since we replaced $Y$ by an affine open on which (the pullback of) $L$ is trivial, we see that $L$ corresponds to the choice of a $\Gamma$-equivariant structure on $\Oo_Y$. That is, $L$ corresponds to a character $\lambda\colon \Gamma \to \G_m$.

We claim that the central extension \eqref{centext} can be constructed explicitly as a fibre product of $\lambda\colon \Gamma \to \G_{m,Y}$ by the $r$-th power map $[r]\colon \G_m \to \G_m$:
$$1 \to \mu_{r,Y} \to \widehat{\Gamma}_{\alpha} = \Gamma \times_{\lambda,\G_{m,Y},[r]} \G_{m,Y} \xrightarrow{\pr_1} \Gamma \to 1.$$
Indeed, projection to the second component gives rise to a character 
$$\mu\colon \widehat{\Gamma}_{\alpha} \to \G_{m,Y},$$
such that $\mu^r = \lambda \circ \pr_1$. That is, the central extension $\widehat{\Gamma}_{\alpha}$ agrees with the obstruction for $\lambda$ to have an $r$-th root.

The character $\lambda\colon \Gamma \to \G_{m,Y}$ factors through $\mu_{|\Gamma|,Y}$, which we assumed to be a constant group scheme. Therefore, $\widehat{\Gamma}_{\alpha}$ is obtained by pulling back the central extension 
$$1 \to \mu_{r,Y} \to \mu_{r\cdot{}|\Gamma|,Y} \xrightarrow{[r]} \mu_{|\Gamma|,Y} \to 1$$
along $\lambda\colon \Gamma \to \mu_{r,Y}$. The assumption that $\mu_{r\cdot{}|\Gamma|}$ is a constant \'etale group scheme on $Y$ yields that the latter central extension is induced by an extension of abstract groups.
\end{proof}
By definition of the transgression map, the $\mu_r$-torsor $P_{\alpha}$ is given by pulling back the $\mu_r$-torsor $\widehat{\Gamma}_{\alpha}$ along the map $\I[Y/\Gamma] \to \Gamma \times Y$. Since $\widehat{\Gamma}_{\alpha} \to \Gamma$ is a trivial $\mu_r$-torsor by the claim above, we deduce that $P_{\alpha}$ is trivial. This proves what we wanted.
\end{proof}


\subsection{Twisted stringy invariants}
The $\mu_r$-torsor $P_\alpha$ associated to a $\mu_r$-gerbe $\alpha$ by the transgression construction from the previous subsection enables us to define a variant of stringy cohomology of a quotient stack $\Xc = [Y/\Gamma]$, which takes a given $\mu_r$-gerbe on $\Xc$ into account:
\begin{definition}\label{defi:Ealpha}
Let $\Xc = [Y/\Gamma]$ be a complex quotient stack of a smooth complex variety $Y$ by a finite group $\Gamma$. For a positive integer $r$ and a gerbe $f\in H^2_{\text{\'et}}(\Xc,\mu_r)$ we define the \emph{$\alpha$-twisted stringy $E$-polynomial} of $\Xc$ as
\[E_{\mathsf{st}}(\Xc,\alpha;x,y) = \sum_{\gamma \in \Gamma/\text{conj}} \left( \sum_{\Z \in \pi_0([Y^{\gamma}/C(\gamma)])}E(\Z,L_{\gamma};x,y) (xy)^{F(\gamma,\Z)}\right) ,\]
where $L_{\gamma}$ denotes the $\mu_r$-torsor $P_{\alpha}|_{[Y^{\gamma}/C(\gamma)]}$ on $[Y^{\gamma}/C(\gamma)]$ given by transgression and
\begin{equation*}
  E(\Z,L_{\gamma};x,y) = E^{\chi}(L_{\gamma};x,y),
\end{equation*}
where $\chi \colon \mu_r(\Cb) \hookrightarrow \Cb^{\times}$ denotes the standard character, and $E^{\chi}$ denotes the part of the $E$-polynomial corresponding to the $\chi$-isotypic component of the cohomology of the total space $H^*_c(L_{\gamma})$.
\end{definition}

Similarly we can define $\alpha$-twisted versions of stringy points counts for quotient stacks over a finite field $k$. The definition passes via $\ell$-adic cohomology, and hence requires us to choose an embedding $\mu_r(\overline{k}) \subset \overline{\Qb}_{\ell}$, in order to extract an $\ell$-adic local system $L_{\gamma}$ from the $\mu_r$-torsor $P_{\alpha}$. 
\begin{definition}\label{defi:l-torsor}
Let $Y$ be a variety over a finite field $k=\Fb_q$ with an action of a finite abstract group $\Gamma$. Let $\Xc = [Y/\Gamma]$ be the associated quotient stack. For a positive integer $r$ prime to the characteristic of $k$, and $\alpha \in H^2_{\text{\'et}}(\Xc,\mu_r)$ we define 
$$\#^{\alpha}_{\mathsf{st}}(\Xc)= \sum_{\gamma \in \Gamma/\text{conj}} \left( \sum_{\Z \in \pi_0([Y^{\gamma}/C(\gamma)])} q^{F(\gamma,\Z)}\#^{L_\gamma}\Z(k)\right),$$
where $L_{\gamma}$ denotes the induced $\ell$-adic local system on $Y^{\gamma}$ obtained from the $\mu_r$-torsor $P_{\alpha}|_{[Y^{\gamma}/C(\gamma)]}$ and
$$\#^{L_\gamma}\Z(k)=\sum_{x\in \Z(k)_\iso}\frac{\Tr(\Fr_x,L_{\gamma,x})}{|\Aut(x)|},$$
where $\Fr_x$ denotes the Frobenius at $x$.
\end{definition}


\subsection{From point-counts to $E$-polynomials}

Let $X_{\Cb}$ be a complex variety. For a subring $R \subset \Cb$ we refer to an $R$-scheme $X_R$ together with an isomorphism $X_R \times_R \Cb \cong X_{\Cb}$ as an $R$-model of $X_{\Cb}$. Using a finite presentation argument it is easy to show that for every complex variety $X_{\Cb}$ there exists a ring $R \subset \Cb$ of finite type over $\Zb$ such that $X_{\Cb}$ has an $R$-model.

In this section we recall an argument, which allows one to deduce from an agreement of all possible stringy point counts for given $R$-models of two complex varieties $X_{\Cb}$, $Y_{\Cb}$ an agreement of stringy $E$-polynomials $E_{\mathsf{st}}(X_{\Cb};x,y) = E_{\mathsf{st}}(Y_{\Cb};x,y)$.

We begin by proving an equivariant analogue of Katz's Theorem 6.1.2 in the appendix of \cite{MR2453601} (which is a generalisation of Ito's \cite{MR2098399}). This will be obvious to experts, but we are including the details for the sake of completeness. We recommend to readers who are new to this application of $p$-adic Hodge theory to take a look at Katz's explanations in \emph{loc. cit.}

In the following we will always work with compactly supported cohomology when dealing with non-projective varieties.

\begin{definition} Let $G$ be a finite abstract group. We understand $G$-schemes over a base scheme $S$ to be schemes over $S$ with a $G$-action over $S$, and $G$-varieties to be separated $G$-schemes of finite type over a base field, such that every $G$-orbit is contained in an affine open subscheme. For a $G$-representation $V$ over a field $k$ and a character $\chi\colon G\to k^{\times}$, we denote by $V^{\chi}$ the $\chi$-isotypic component of $V$.
\begin{enumerate}
\item[(a)] For a complex $G$-variety $X$ and a complex-valued character $\chi$ of $G$ we let $E_G^{\chi}(X;x,y)\in \Zb[x,y]$ be the polynomial 
$$E_G^{\chi}(X;x,y)=\sum_{p,q \in \Zb} \sum_{i \in \Zb} (-1)^i \dim[\mathsf{gr}_W^{p,q}H_c^{i}(X)^{\chi}]x^p y^q.$$
\item[(b)] For a $G$-variety $X$ over a finite field $\Fb_q$, and a $\overline{\Qb}_{\ell}$-valued character $\chi$ of $G$, we define the $\chi$-twisted point count $\#^{\chi}_G(X) \in \overline{\Qb}_{\ell}$ to be the alternating sum of traces
$$\#_{G}^{\chi}(X) = \sum_{i \in \Zb} (-1)^i \mathsf{Tr}[\Fr,H^i_{c,\text{\'et}}(X,\overline{\Qb}_{\ell})^{\chi}] = \sum_{x \in [X/G](\Fb_q)_\iso}  \mathsf{Tr}[\Fr_X,(L_{\chi})_x],$$
where $L_{\chi}$ denotes the $\ell$-adic local system on the quotient $[X/G]$ induced by $\chi$.
\end{enumerate}
\end{definition}

We can now state an equivariant analogue of Katz's \cite[Theorem 6.1.2]{MR2453601}. We repeat once more that our proof follows closely the one given in \emph{loc. cit.}, and refer the reader to the original source for a less terse account. 

\begin{theorem}\label{equivariant}
Let $G$ be a finite group and $R \subset \Cb$ a subalgebra of finite type over $\Zb$. We fix an abstract isomorphism of $\Cb$ and $\overline{\Qb}_{\ell}$ and let $\chi$ be a complex-valued character of $G$. Assume that $X$ and $Y$ are separated $G$-schemes of finite type over $R$, such that for every ring homomorphism $R \to \Fb_q$ to a finite field $\Fb_q$ we have 
$\#_G^{\chi}(X \times_R \Fb_q) = \#_G^{\chi}(Y \times_R \Fb_q).$ Then, we also have 
\[E_G^{\chi}(X \times_R \Cb;x,y) = E_G^{\chi}(Y \times_R \Cb;x,y).\]
\end{theorem}

\begin{proof}
As in \emph{loc. cit.} we may assume that $X$ and $Y$ are smooth and projective over $R$. This is possible by virtue of Bittner's \cite[Lemma 7.1]{MR2059227}, which we use to replace \cite[Lemma 6.1.1]{MR2453601} in Katz's proof. 

Furthermore, we can achieve regularity of $\Spec R$ by inverting a single element in $R$ (since $R \subset \Cb$ is integral, it is generically smooth). We let $f\colon X \to \Spec R$ and $g\colon Y \to \Spec R$ denote the structural morphisms. As in $\emph{loc. cit.}$ we choose a prime $\ell$, such that $\ell$ is larger than $\dim X$ and $\dim Y$ and such that we have a finite extension $E$ of $\Qb_{\ell}$ together with an embedding $R$ into the valuation ring $\Oo$ of $E$ (in fact, by virtue of Cassels's Embedding Theorem \cite{cassels1976embedding} there are infinitely many prime numbers $\ell$, such that $E$ can be chosen to be $\mathbb{Q}_{\ell}$). For every integer $i$ we then consider the lisse sheaves $(R^if_*\Qb_{\ell})^{\chi}$ and $(R^ig_*\Qb_{\ell})^{\chi}$ on $\Spec R[\frac{1}{\ell}]$. Smoothness and projectivity of $X$ and $Y$ guarantee purity of these sheaves. Hence the equality $\#_G^{\chi}(X_{\Fb_{q^r}}) = \#_G^{\chi}(Y_{\Fb_{q^r}})$ for all finite overfields $\mathbb{F}_{q^r}$ of $\mathbb{F}_q$ implies the identity
$$\det(1-t\Fr_q,(R^if_*\Qb_{\ell})^{\chi}) = \det(1-t\Fr_q,(R^ig_*\Qb_{\ell})^{\chi})$$ 
for all integers $i$. Chebotarev's Density Theorem yields an isomorphism of semi-simplifications $$((R^if_*\Qb_{\ell})^{\chi})^{ss} \cong ((R^ig_*\Qb_{\ell})^{\chi})^{ss}$$ of lisse sheaves of $\Spec R[\frac{1}{\ell}]$. 

Using the embedding $R \hookrightarrow \Oo$ we can pull back our constructions and insights obtained so far to this finite extension of $\mathbf{Z}_{\ell}$. We obtain two smooth projective $E$-varieties $X_E$ and $Y_E$ of good reduction for which we have an equivalence of lisse sheaves over $\Spec E$ 
$$((R^i(f_E)_*\Qb_{\ell})^{\chi})^{ss} \simeq ((R^i(g_E)_*\Qb_{\ell})^{\chi})^{ss}.$$
Fontaine--Messing's \cite{MR902593} or Faltings's \cite{MR924705} shows that $R^i((f_E)_*\Qb_{\ell}) = H^i(X_{\bar E},\Qb_{\ell})$ is a Hodge-Tate representation of $\Gal(E)$ and that there is a natural isomorphism $$\bigoplus_{p+q = i}H^q(X_E,\Omega^p) \otimes \Cb_{\ell}(-p) \simeq H^i(X_{\bar E},\Qb_\ell) \otimes \Cb_{\ell}.$$
The naturality of this isomorphism implies that this isomorphism respects the $G$-action on both sides. We infer that the Hodge Tate numbers of $H^i(X_{\bar E},\Qb_{\ell})^{\chi}$ recover the dimension of the $\chi$-isotypic component of $H^q(X_{E},\Omega^q)$. We conclude the proof by applying the same reasoning to $Y_E$ and using the equivalence of the semi-simplifications of the $\Gal(E)$-representations $H^i(X_{\bar E},\Qb_{\ell})^\chi$ and $H^i(Y_{\bar E},\Qb_{\ell})^{\chi}$.
\end{proof}

We will use a slightly more general result. In order to apply Theorem \ref{equivariant} to compare stringy $E$-polynomials, we have to adjoin formal $r$-th roots of the Lefschetz motive $\mathbb{L} = [\Ab^1]$ to the Grothendieck ring of $G$-varieties $K_0(\mathsf{Var}_G/R)[(\mathbb{L}^{\frac{1}{r}})_{r \geq 2}]$. 

\begin{definition}\label{importantrmk}
\begin{enumerate}[(a)]
\item For $q\in \Zb$ a prime power, a compatible system of roots of $q$ is a sequence $s=(s_r)_{r \geq 2}$ in $\Cb$ with $s_1=q$ and $s_{rr'}^r = s_{r'}$  for all $r,r' \geq 1$.
\item Given a ring homomorphism $R\to \Fb_q$, every compatible system $s$ of roots of $q$ induces a well-defined equivariant point-count homomorphism
$$\#_{\Fb_q}^{s,\chi}\colon K_0(\mathsf{Var}_G/R)[(\mathbb{L}^{\frac{1}{r}})_{r \geq 2}] \to \Cb$$
extending the usual point count homomorphism $$\#^\chi_{\Fb_q} \colon K_0(\mathsf{Var}_G/R) \to \Cb$$ by stipulating $\#_{\Fb_q}^{s,\chi}(\mathbb{L}^{\frac{1}{r}}) = s_r$ for the trivial character $\chi$ and $0$ otherwise. 

\item Since two choices $s$ and $s'$ of roots of $p$ differ by an element $\sigma$ of $\Gal(\bar{\Qb}/\Qb)$, we see that for $X, Y \in K_0(\mathsf{Var}_G/R)[(\mathbb{L}^{\frac{1}{r}})_{r \geq 2}]$ we have $\#_{\Fb_q}^{s,\chi}(X) = \#_{\Fb_q}^{s,\chi}(Y)$ if and only if $\#_{\Fb_q}^{s',\chi}(X) = \#_{\Fb_q}^{s',\chi}(Y)$. Hence we obtain a well-defined relation $\#_{\Fb_q}^{\chi}(X) = \#_{\Fb_q}^{\chi}(Y)$ on $K_0(\mathsf{Var}_G/R)[(\mathbb{L}^{\frac{1}{r}})_{r \geq 2}]$.
\item Similarly, the equivariant $E$-polynomial extends to a function
$$E^{\chi}\colon K_0(\mathsf{Var}_G/\Cb)[(\mathbb{L}^{\frac{1}{r}})_{r \geq 2}] \to \Zb[(x^{\frac{1}{r}},y^{\frac{1}{r}})_{r \geq 2}]$$
by stipulating $E^{\chi}(\mathbb{L}^{\frac{1}{r}}) = (xy)^{\frac{1}{r}}$ for the trivial character, and $0$ otherwise.
\end{enumerate}
\end{definition}

The most natural choice of a compatible system of roots of $q$ is the sequence of positive real roots. However we will also consider the sequence $(\Tr(\Fr_{\Fb_q},\Qb_{\ell}(\frac{1}{r})))_{r\geq 1}$ given by a homomorphism $R \to \Fb_q$ and a compatible system of roots of the Tate twist $\Qb_\ell(1)$ on $\Spec R$ in the following sense:

\begin{lemma} \label{TateRoots}
After replacing $R$ by a finite \'etale extension $R'$ there exist $r$-th roots $\Qb_{\ell}(\frac{1}{r})$ of the Tate twist $\Qb_{\ell}(1)$ as lisse $\ell$-adic sheaves on $\Spec R$ for all $r\geq 1$ satisfying $\Qb_\ell(\frac{1}{rr'})^{\otimes r'} \cong \Qb_\ell(\frac{1}{r})$ for all $r,r'\geq 1$.
\end{lemma}

\begin{proof}
This is a mild generalisation of Ito's \cite[5.3]{MR2098399}. We consider the Tate twist $\Qb_{\ell}(1)$ as a representation of the \'etale fundamental group $\rho\colon \pi_1^{\text{\'et}}(\Spec R,\Cb) \to \Qb_{\ell}^{\times}$. Since it is a continuous $\ell$-adic presentation of a profinite group it factors through $\Zb_{\ell}^{\times}$. We choose an open subgroup $U \subset \Zb_{\ell}^{\times}$ and $V \subset \Zb_{\ell}$, such that we have a logarithm $\log\colon U \to V$ and $\exp\colon V \to U$. There exists a pointed finite \'etale covering $\Spec R' \to \Spec R$, such that $\rho|_{\pi_1^{\text{\'et}}(\Spec R',\Cb)}$ factors through $U$. 

For $r \geq 1$ we define another continuous $\ell$-adic representation of ${\pi_1^{\text{\'et}}(\Spec R',\Cb)}$ by the formula
$$\exp(\frac{1}{r} \log(\rho|_{\pi_1^{\text{\'et}}(\Spec R',\Cb)})).$$
By construction the corresponding lisse \'etale sheaves are $r$-th roots of the Tate twist satisfying the required compatibility condition.
\end{proof}


\begin{theorem}\label{equivariant2}
Let $G$ be a finite group and $R \subset \Cb$ a subalgebra of finite type over $\Zb$. We fix an abstract isomorphism of $\Cb$ and $\overline{\Qb}_{\ell}$ and let $\chi$ be a complex-valued character of $G$. We assume that $X,Y \in K_0(\mathsf{Var}_G/R)[(\mathbb{L}^{\frac{1}{r}})_{r \geq 2}]$, such that for every ring homomorphism $R \to \Fb_q$ to a finite field $\Fb_q$ we have 
$\#_{\Fb_q}^{\chi}(X) = \#_{\Fb_q}^{\chi}(Y).$ Then, we also have 
\[E^{\chi}(X \times_R \Cb;x,y) = E^{\chi}(Y \times_R \Cb;x,y).\]
\end{theorem}

\begin{proof}
After replacing $R$ by a finite \'etale extension we may choose a compatible system $\Qb_\ell(\frac{1}{r})$ as in Lemma \ref{TateRoots}. 

We fix a character $\chi$ of $G$. The assertion is reduced to the following situation: Let $X_0,\dots,X_m$ and $Y_0,\dots,Y_m$ be smooth projective $G$-varieties over $R \subset \Cb$ (which is smooth over $\Zb$), and $$\alpha_0,\dots,\alpha_m, \beta_0,\dots,\beta_{m'} \in \Qb \cap [0,1),$$ such that we have for every $R \to \Fb_q$ an equality
$$\sum_{i=0}^m q^{\alpha_i}\#^{s,\chi}_{\Fb_q}(X_i) = \sum_{i=0}^{m'} q^{\beta_i} \#^{s,\chi}_{\Fb_q}(Y_i),$$
where $s=(\Tr(\Fr_{\Fb_q},\Qb_{\ell}(\frac{1}{r})))_{r\geq 1}$ is the induced system of roots of $q$.
We denote by $X = \sum_{i=0}^m [X_i] \cdot{} \mathbb{L}^{\alpha_i}$ and $Y = \sum_{j=0}^{m'} [Y_i] \mathbb{L}^{\beta_j}$ the corresponding elements of $K_0(\mathsf{\mathsf{Var}_G}/R)[(\mathbb{L}^{\frac{1}{r}})_{r \geq 2}]$. 

We choose a prime $\ell$, such that $\ell > \dim X_i, \dim Y_j$ for all $i,j$, and, such that we have a finite extension $E$ of $\Qb_{\ell}$ together with an embedding $R$ into the valuation ring $\Oo$ of $E$.


We denote by $f_i\colon X_i \to \Spec R$ and $g_i \colon Y_i \to \Spec R$ the structural morphisms. For a rational number $c \in \Qb \setminus \Nb$ we define $(R^cf_{i,*}\Qb_{\ell})^{\chi} = 0$. We can now consider for every rational number $c > 0$ the lisse $\ell$-adic sheaves
$$(R^c f_*\Qb_{\ell})^{\chi} = \bigoplus_{i=0}^m (R^{c-2\alpha_i}f_{i,*}\Qb_{\ell}(\alpha_i))^{\chi},$$
and similarly,
$$(R^c g_*\Qb_{\ell})^{\chi} = \bigoplus_{i=0}^{m'} (R^{c-2\beta_i}g_{i,*}\Qb_{\ell}(\beta_i))^{\chi}.$$
By Chebotarev Density we have an isomorphism of semi-simplifications $(R^c f_*\Qb_{\ell})^{\chi,ss} \simeq (R^c g_*\Qb_{\ell})^{\chi,ss}$ of lisse sheaves over $\Spec R[\ell^{-1}]$. We now use the morphism $R \to E$, and consider the Galois representations of $\Gal(\bar{E}/E)$ induced by these lisse sheaves by pullback to $E$.

By applying $p$-adic Hodge theory (that is, Fontaine--Messing's \cite{MR902593} or Faltings's \cite{MR924705}, as in the proof of Theorem \ref{equivariant}) to the smooth projective varieties $X_i$ (and $Y_j$), we see that the Hodge-Tate weights of these Galois representation are given by the formal expressions:
$$h^{p,q}_{\chi}(X_{\Cb}) = \sum_{i=0}^m h^{p-\alpha_i,q-\alpha_i}_{\chi}(X_i),$$ respectively $h^{p,q}_{\chi}(Y_{\Cb}) = \sum_{i=0}^{m'} h^{p-\alpha_i,q-\alpha_i}_{\chi}(Y_j)$
where $p,q \in \Qb$, and we use the convention that for non-integral rational number $c,d$ we have $h^{c,d}_{\chi}(X_i) = 0$.
We therefore conclude that $E^{\chi}(X_{\Cb}) = E^{\chi}(Y_{\Cb})$.

\end{proof}

\begin{theorem}\label{thm:stringy-katz}
Let $R \subset \Cb$ subalgebra of finite type over $\Zb$. We fix an abstract isomorphism of $\Cb$ and $\overline{\Qb}_{\ell}$. Let $X_i$ be smooth $\Gamma_i$-varieties for two abstract finite abelian groups $\Gamma_1$ and $\Gamma_2$. Let $\Xc_i = [X_i /\Gamma_i]$ be the resulting quotient $R$-stacks and $\alpha_i$ be a $\mu_r$-gerbe on $\Xc_i$ for $i=1,2$. We suppose that for every ring homomorphism $R \to \Fb_q$ to a finite field $\Fb_q$ we have 
$\#_{\mathsf{st}}^{\alpha_1}(\Xc_1 \times_R \Fb_q) = \#_{\mathsf{st}}^{\alpha_2}(\Xc_2 \times_R \Fb_q).$ Then, we also have \[E_{\mathsf{st}}(\Xc_1 \times_R \Cb,\alpha_1;x,y) = E_{\mathsf{st}}(\Xc_2 \times_R \Cb,\alpha_2;x,y).\]
\end{theorem}

\begin{proof}
We may assume $\Gamma_1 = \Gamma_2$, since we can replace $\Gamma_1$ and $\Gamma_2$ by $\Gamma=\Gamma_1 \times \Gamma_2$ in the following way: 
$$\Xc_1 = [(X_1 \times \Gamma_2)/\Gamma],$$
and similarly for $\Xc_2$. Moreover after enlarging $R \subset \Cb$ we may assume that these presentations as quotient stacks are also defined over $R$. 
We shall assume that $R$ contains $\mu_r(\mathbf{C})$; this can always be achieved by suitably modifying $R$.
 
For each prime power $q$ let $s_q=(q^\frac{1}{r})_{r\geq 1}$ be the compatible system of positive real roots. Then for every homomorphism $R \to \Fb_q$, the stringy point-count of $\Xc_i$ twisted by $\alpha_i$ is given by $\#^{s_q,\chi}_{\Fb_q}$ applied to
$\sum_{[\gamma] \in \Gamma_i} \mathbb{L}^{F(\gamma)}L_{i,\gamma},$
where $L_{i,\gamma}$ is the $\Gamma$-equivariant local system on $X_i^{\gamma}$ induced by the gerbe $\alpha_i$.

 The local system $L_{i,\gamma}$ is induced by a $\mu_r$-torsor (with respect to a chosen embedding $\chi\colon \mu_{r} \hookrightarrow \overline{\Qb}_{\ell}$). Let $Y_{i,\gamma}$ be the total space of this torsor. It is acted on by $\Gamma \times \mu_r$, and $E_{\Gamma}(X_i^\gamma,L_\gamma;x,y)$ is equal to $E_{\Gamma \times \mu_r}^{\chi}(Y_{i,\gamma};x,y)$, where $\chi$ denotes the character induced by the chosen embedding $\mu_r \hookrightarrow \Qb_{\ell}$ and the projection $\Gamma \times \mu_r \twoheadrightarrow \mu_r$. We can therefore apply Theorem \ref{equivariant2} to $G=\Gamma \times \mu_r$ and $\sum_{\gamma \in \Gamma}[Y_{i,\gamma}]\cdot{}\mathbb{L}^{\alpha_i}$ to deduce the assertion.
\end{proof}


\section{Arithmetic of local fields}

\subsection{Galois theory of local fields} \label{galois}


We fix a prime $p>0$. First we recall some general facts about local fields of residue characteristic $p$, that is, finite extensions of $\Qb_p$ or $\mathbb{F}_p((T))$.

For a field $F$ equipped with a valuation $v\colon F^{\times}\to \Zb$ we denote by $\Oo_F$ the ring of integers of $F$, by $\mathfrak{m}_F$ the maximal ideal of $\Oo_F$ and by $k_F = \Oo_F/\mathfrak{m}_F$ its residue field. We will be interested in local fields $F$ as well as their algebraic extensions, which we equip with the unique prolongation of the valuation on $F$.

Now we fix a local field $F$ and a separable closure $F^\text{s}$ of $F$. We also fix a uniformiser $\pi\in F$.

\begin{definition}\label{locgal}
For an algebraic overfield $L$ of $F$ we let the inertia group $I_L$ be the kernel of the canonical surjective homomorphism $\Gal(L/F)\to \Gal(k_L/k_F)$ and, in case $I_L$ is finite, we let the ramification index $e_{L/F}$ of $L$ over $F$ be the order of $I_L$.

\begin{enumerate}[(i)]

\item  An algebraic field extension $F\subset L$ is called \emph{totally ramified} if the induced extension of residue fields $k_F\subset k_L$ is the trivial one.
\item An algebraic field extension $F\subset L$ is called \emph{unramified} if for every intermediary field $F\subset L' \subset L$, which is finitely generated over $F$ we have $e_{L'/F}=1$.
\item An algebraic field extension $F\subset L$ is called \emph{tamely ramified} if for every intermediary field $F\subset L' \subset L$, which is finitely generated over $F$ the ramification index $e_{L'/F}$ is prime to $p$.
\end{enumerate}
\end{definition}

We will be mainly interested in abelian algebraic extensions of $F$. A fixed separable closure $F^s$ of $F$ contains the following tower of extensions:
\begin{equation*}
  F \subset F^\text{ur} \subset F^\text{tr} \subset F^\text{ab} \subset F^\text{s}
\end{equation*}
Here $F^\text{ab}$ is the maximal abelian extension of $F$, $F^\text{ur}$ is the maximal abelian unramified extension and $F^\text{tr}$ is the maximal tamely ramified abelian extension of $F$.

 By local class field theory there is a canonical homomorphism $r\colon F^{\times}\to \Gal( F^\text{ab} /F)$, which is injective and has dense image. It fits into a diagram of split short exact sequences
\begin{equation} \label{RecDiag}
  \xymatrix{ 
0 \ar[r] & \Oo_F^{\times} \ar[r] \ar[d] & F^{\times} \ar[r]^v \ar[d]^r & \Zb \ar[r]  \ar[d] & 0 \\
             0 \ar[r] & \Gal(F^\text{ab}/F^\text{ur}) \ar[r] & \Gal( F^\text{ab}/F) \ar[r] & \Gal(\overline{k}_F/ k_F) \ar[r] & 0
}
\end{equation}
where each vertical homomorphism is injective with dense image and the right vertical homomorphism sends $1\in \Zb$ to the Frobenius automorphism $x\mapsto x^{|k_F|}$ of $\overline{k}_F$. 

Brauer groups play an important role in local class field theory. We refer the reader to \cite[Proposition XIII.6]{MR0354618} for a proof of the following result, which summarises the main properties of Brauer groups of local fields.

\begin{theorem}\label{BrauerLoc}
The Brauer group $\Br(F)$ of a local field is isomorphic to $\Qb/\Zb$ by means of the Hasse invariant $\inv\colon \Br(F) \to \Qb/\Zb$. For a finite field extension of local fields $L/F$ we have a commutative diagram
\[
\xymatrix{
\Br(F) \ar[r]^{\simeq} \ar[d] & \Qb/\Zb \ar[d]^{\cdot{}[L:F]} \\
\Br(L) \ar[r]^{\simeq} & \Qb/\Zb.
}
\]
\end{theorem}

\par
For $M$ a finite \'etale abelian group scheme over $F$, we denote by $H^i_\text{\'et}(F,M)$ the degree $i$ Galois cohomology group of $M$, that is, the group $H^i(\Gal(F^{\text{sep}}/F),M)$. Alternatively, we can view it as the $i$-th \'etale cohomology group of $M$ over $\Spec F$. For $i=0$ we obtain the finite group of $F$-rational points of $M$ and for $i = 1$ the group of $M$-torsors defined over $\Spec F$. In higher degrees one can give similar geometric interpretations, but we will not need this. The cohomology groups are known to vanish in degrees $i \geq 3$ (see \cite[Section I.2]{MilneADT}).

We denote by $H^i_{\text{ur}}(F,M)$ the Galois cohomology group $H^i_\text{\'et}(\Gal(F^{\text{ur}}/F),M)$. Since $\Gal(F^{\text{ur}}/F) = \Gal(k_F^{\text{sep}}/k_F)$, we see that $H^i_{\text{ur}}(F,M) \simeq H^i_\text{et}(\Oo_F,M)\simeq H^i_\text{\'et}(k_F,M_{k_F})$.

For a finite abelian group $G$, we denote by $G^*$ the group of characters $G \to \Qb/\Zb$. This construction is a special case of the Pontryagin dual defined below. For a finite commutative group scheme $M$ over $F$, we denote by $M^{\vee}\defeq \underline{Hom}(M,\Gb_m)$ its Cartier dual.

\begin{theorem}[{\cite[Corollary I.2.3]{MilneADT}}]\label{finiteduality}
Let $M$ be a commutative finite group scheme over $F$ of order prime to $p$. For every $i \in \mathbb{Z}$ there exists a canonical perfect pairing $H^i_\text{\'et}(F,M) \times H^{2-i}_\text{\'et}(F,M^{\vee}) \to \Qb/\mathbb{Z}$. Furthermore, the annihilator of $H^i_{\text{ur}}(F,M)$ is equal to $H^{2-i}_{\text{ur}}(F,M^{\vee})$.
\end{theorem}

As we remarked above the cohomology groups $H^i_\text{\'et}(F,M)$ are finite. Their cardinalities are subject to the following constraint.

\begin{theorem} \label{euler}
  Let $M$ be a commutative finite \'etale group scheme over $F$ of order prime to $p$. Then we have
  \begin{equation*}
    |H^1_\text{\'et}(F,M)|=|M(F)| |M^\vee(F)|.
  \end{equation*}
\end{theorem}
\begin{proof}
  This is a combination of \cite[I.2.9]{MilneADT} and the identity $|H^2_\text{\'et}(F,M)|=|M^\vee(F)|$ implied by Theorem \ref{finiteduality}.
\end{proof}

Now let $\Gamma$ be a finite abelian group of order $n$ prime to $p$. We denote by $\underline{\Gamma}$ the constant group schemes over $\Spec(F)$ or $\Spec(\Oo_F)$ with value group $\Gamma$ and by $\mu(F)$ the finite group of roots of unity in $F$.
\begin{construction} \label{H1Cons}
  Consider the homomorphisms
  \begin{equation*}
   \Zb \times \mu(F) \xrightarrow{i} F^* \xrightarrow{r} \Gal(F^{\text{ab}}/F),
  \end{equation*}
where the first $i$ sends $1 \in \Zb$ to the chosen uniformiser $\pi \in F^*$ and is given by the inclusion $\mu(F) \into F^*$ on the second factor and $r$ is the reciprocity homomorphism from \eqref{RecDiag}.

  Since $\underline\Gamma$ is a constant abelian group scheme, there are canonical isomorphisms $$H^1_\text{\'et}(F,\underline\Gamma)\cong \Hom(\Gal(F^\text{s}/F),\Gamma)\cong \Hom(\Gal(F^\text{ab}/F),\Gamma),$$ where $\Hom$ denotes continuous group homomorphisms. 
Hence composition with $i$ gives a homomorphism
\begin{equation*}
   H^1_\text{\'et}(F,\underline\Gamma)\to \Gamma\oplus \Hom(\mu(F),\Gamma).
\end{equation*}
\end{construction}

\begin{proposition} 
  The homomorphism
  \begin{equation} \label{H1Iso}
    H^1_\text{\'et}(F,\underline\Gamma)\to \Gamma\oplus \Hom(\mu(F),\Gamma)
  \end{equation}
from Construction \ref{H1Cons} is an isomorphism. 
With respect to this isomorphism the inclusion of $\Gamma$ corresponds to $H^i_{\text{ur}}(F,\underline \Gamma) \subset H^1_\text{\'et}(F,\underline\Gamma)$. 
\end{proposition}
\begin{proof}
As noted above the reciprocity homomorphism $r\colon F^{\times}\to \Gal(F^\text{ab}/F)$ is injective with dense image. Hence it induces an isomorphism $$\Hom(\Gal(F^\text{ab}/F),\Gamma) \cong \Hom(F^{\times},\Gamma).$$

Now the choice of uniformiser $\pi\in F$ gives an isomorphism $F^{\times}\cong \Zb\times \Oo_F^{\times}$. The $\Zb$-module structure on $\Oo_F^{\times}$ extends uniquely to a continuous $\Zb_p$-module structure. Because of the existence of the logarithm the $\Zb_p$-module $\Oo_F^{\times}$ is isomorphic to the direct sum of a free $\Zb_p$-module (which has finite rank if $F$ has characteristic zero and countably infinite rank otherwise) and its torsion subgroup $\mu(F)$ (see \cite[Satz II.5.7]{MR1085974}). Since the order of $\Gamma$ is prime to $p$ we finally obtain an isomorphism
  \begin{equation*}
    \Hom(F^{\times},\Gamma) \cong \Hom(\Zb \times \mu(F),\Gamma) \cong \Gamma \oplus \Hom(\mu(F),\Gamma).
  \end{equation*}
The second claim follows from the above and the diagram \eqref{RecDiag}.
\end{proof}

\begin{lemma} \label{ConstantTateDuality}
  Assume that $F$ contains all roots of unity of order $|\Gamma|$. Then $(\underline{\Gamma})^\vee$ is the constant group scheme with value $\Hom(\Gamma,\mu(F))$. Under the isomorphisms \eqref{H1Iso}
  \begin{equation*}
    H^1(F,\underline{\Gamma}) \cong \Gamma \oplus \Hom(\mu(F),\Gamma)
  \end{equation*}
and 
\begin{equation*}
  H^1(F,\underline{\Gamma}^\vee) \cong \Hom(\Gamma,\mu(F)) \oplus \Hom(\mu(F),\Hom(\Gamma,\mu(F)))
\end{equation*}
the Tate duality pairing $H^1(F,\underline{\Gamma}) \times H^1(F,(\underline{\Gamma})^\vee) \to \Qb/\Zb$ from Theorem \ref{finiteduality} can be described as follows: The first factors of $H^1(F,\Gamma)$ and $H^1(F,(\underline{\Gamma})^\vee)$ jointly pair to zero, and analogously for the second factors. The remaining part of the pairing is given by the canonical evaluation pairings
\begin{equation*}
  \Gamma \times \Hom(\mu(F),\Hom(\Gamma,\mu(F))) \to \Hom(\mu(F),\mu(F)) \cong \Zb/ |\mu(F)| \Zb 
\end{equation*}
and
\begin{equation*}
  \Hom(\mu(F),\Gamma) \times \Hom(\Gamma,\mu(F)) \to \Hom(\mu(F),\mu(F)) \cong \Zb/ |\mu(F)| \Zb .
\end{equation*}
\end{lemma}
\begin{proof}
  The claim about the first (respectively second) factors pairing to zero follows from the last part of Theorem \ref{finiteduality}. For the remaining part of the claim, using functoriality in $\Gamma$ one can reduce to the case $\Gamma=\Zb/n\Zb$. In this case the claim can be verified using \cite[Remark I.2.5(b)]{MilneADT}.
\end{proof}



\subsection{Tate duality}\label{tate}


Following \cite{MilneADT}, we call an abelian torsion group $M$ \emph{of cofinite type} if for each $n\in \Zb$ the $n$-torsion subgroup of $M$ is finite. In the following we will deal with various abelian groups $M$, which are either profinite or torsion of cofinite type. We will always equip the profinite groups with the profinite topology and the torsion groups with the discrete topology. On the intersection of these two classes, namely finite abelian groups, these topologies agree. For such an $M$ we denote by $M^*\coloneqq \Hom_{\text{cts}}(M,\Qb/\Zb)$ its Pontryagin dual. The functor $M\mapsto M^*$ is a contravariant equivalence from the category of profinite groups to the category of torsion groups of cofinite type and vice versa. For a finite group $M$ we have $|M|=|M^*|$. For a profinite group $M$ and a torsion group $N$ of cofinite type (or vice versa) a continuous bilinear pairing $M\times N\to \Qb/\Zb$ is called \emph{non-degenerate} if the induced homomorphism $M\to N^*$ is an isomorphism.

Let $F$ be a local field as above. For an abelian variety $A$ over $F$ we recall the Tate duality pairings on the \'etale cohomology groups $H^i_\text{\'et}(F,A)$ of $A$: 

\begin{lemma}[{\cite[I.3.1]{MilneADT}}]\label{lemma:shift}
For any abelian variety $A$ over $F$ and any $r \geq 0$ there is a canonical isomorphism $$ H^r_\text{\'et}(F,A^{\vee})\xrightarrow{\simeq} \Ext^{r+1}_F(A,\Gb_m).$$ These isomorphisms are functorial in $A$.
\end{lemma}

\begin{construction}
For each $r\geq 0$ there is a natural pairing
\begin{equation} \label{pairing1}
  H^r_\text{\'et}(F,A)\times \Ext^{2-r}(A,\Gb_m) \to H^2_\text{\'et}(F,\Gb_m).
\end{equation}
From the construction of these pairings one sees that they are functorial in $A$ (see \cite[I.0.16]{MilneADT}).

By combining these pairings with the Hasse invariant
\begin{equation*}
  H^2_\text{\'et}(F,\Gb_m)\cong \Qb/\Zb
\end{equation*}
(c.f. Theorem \ref{BrauerLoc}) we obtain functorial pairings
\begin{equation*}
  H^r_\text{\'et}(F,A) \times H^{1-r}_\text{\'et}(F,A^{\vee}) \to \Qb/\Zb.
\end{equation*}
\end{construction}

\begin{theorem}[{Tate, see \cite[I.3.4]{MilneADT}}]\label{tateduality}
Let $A$ be an abelian variety over $F$. The cohomology groups $H^r(F,A)$ are zero for $r\geq 2$. The group $A(F)$ is profinite and the group $H^1_\text{\'et}(F,A)$ is torsion of cofinite type.
For $r=0,1$ the pairing
\begin{equation*}
  H^r_\text{\'et}(F,A) \times H^{1-r}_\text{\'et}(F,A^{\vee}) \to \Qb/\Zb
\end{equation*}
defined above is continuous and non-degenerate.
\end{theorem}

\begin{rmk} \label{TDCons}
The Tate duality pairing $A(F)\times H^1_\text{\'et}(A,F)\to \Qb/\Zb$ from Theorem \ref{tateduality} can be described as follows: Let $a\in A(F)$ and $T\in H^1_\text{\'et}(A^{\vee},F)$. Under the isomorphism $H^1_\text{\'et}(A,F)\cong \Ext^2(A,\Gb_m)\cong \Hom_D(A,S^2(\Gb_m))$, the torsor $T$ corresponds to a $\Gb_m$-gerbe on $A$. Pulling this gerbe back along $$a\colon \Spec(F)\to A$$ gives a $\Gb_m$-gerbe on $\Spec(F)$, which corresponds to an element of the Brauer group $\Br(F)=H^2_\text{\'et}(F,\Gb_m)$. This element is the image of $a$ and $t$ under the pairing.

This follows from \cite[Section I.0]{MilneADT}, where the pairings \eqref{pairing1} used above are described in the following language: Let $D$ be derived category of the abelian category of \'etale sheaves of abelian groups over $F$. We denote by $S\colon D \to D$ the shift functor $X \mapsto X[1]$. 

Let $\underline{\Zb}$ be the constant \'etale sheaf over $F$ with value group $\Zb$. Then for some $r\geq 0$ there are canonical isomorphisms 
\begin{align*}
  H^r_\text{\'et}(F,A) &\cong \Hom_D(\underline{\Zb},S^r(A)), \\
  \Ext^{2-r}(A,\Gb_m) &\cong \Hom_D(A,S^{2-r}(\Gb_m)),  \\
 \text{ and } H^2_\text{\'et}(F,\Gb_m) &\cong \Hom_D(\underline{\Zb},S^2(\Gb_m))
\end{align*}
 under which the pairing \eqref{pairing1} coincides with the pairing  
\begin{equation*}
  \Hom_D(\underline{\Zb},S^r(A)) \times \Hom_D(A,S^{2-r}(\Gb_m)) \cong \Hom_D(\underline{\Zb},S^r(A)) \times \Hom_D(S^r(A),S^2(\Gb_m)) \to \Hom_D(\underline{\Zb},S^2(\Gb_m))
\end{equation*}
given by composition of morphisms in $D$.
\end{rmk}

The forgetful morphism $\Ext^2(A,\G_m) \to H^2_{\text{\'et}}(A,\G_m)$ is of central importance. The elements of the abelian group $\Ext^2(A,\G_m)$ are isomorphism classes of $\G_m$-gerbes on $A$ endowed with a group structure (see for instance \cite[3.1]{travkin}, or \cite[5.5]{MR2373230} for an exposition using the language of Azumaya algebras). Informally they can be thought of as central extensions
$$1 \to B\G_m \to \widehat{A} \to A \to 1$$
of abelian group stacks.  The canonical map $\Ext^2(A,\G_m) \to H^2_{\text{\'et}}(A,\G_m)$ retains only the isomorphism class of the $\G_m$-gerbe and forgets the group structure. The next lemma shows that after replacing $\G_m$ by $\mu_r$, the analogous morphism is actually an injection, and its image can be described explicitly.

\begin{lemma}\label{lemma:group_structure}
Let $F$ be a local field and $r$ a positive integer invertible in $F$. We denote by $A/F$ an abelian variety, and by $P/F$ an $A$-torsor. Then the natural map $\Ext^2(A,\mu_r) \to H^2_{\text{\'et}}(A,\mu_r)$ is injective. Its image corresponds to those $\mu_r$-gerbes $\alpha$ on $A$, which become trivial when pulled back along the morphisms $A_{F^{s}}=A \times_{F} \Spec F^{s} \to A$ and $\Spec F \xrightarrow{e} A$.

Similarly, we have a canonical equivalence   
\begin{equation*}
 \frac{\ker(\Br(P) \to \Br(P_{F^{s}}))[r]}{\Br(F)[r]} \simeq \Ext^2(A,\mu_r). 
\end{equation*}
\end{lemma}

\begin{proof}
According Lemma \ref{lemma:shift} and Kummer theory we can identify $\Ext^2(A,\mu_r)$ with $H^1_{\text{\'et}}(F,A^{\vee}[r])$. Similarly, the subset of $H^2_{\text{\'et}}(A,\mu_r)$ corresponding to gerbes which are trivial on $A_{F^{s}}$ can be naturally identified with $H^1_{\text{\'et}}(F,A^{\vee}[r])$. To see this one observes that a descent datum for $\mu_r$-gerbes on the trivial $\mu_r$-gerbe on $A_{F^{s}}$ yields a $\Gal(F)$-cocycle taking values in the group of isomorphism classes of $\mu_r$-torsors on $A$, which can be identified with $H^1(F,A^{\vee}[r])$. Vice versa, one can recover $\alpha$ from the corresponding element $H^1(F,A^{\vee}[r])$ up to an element of $\Br(F)$. Since we assume in addition that $e^*\alpha$ is trivial, this establishes the correspondence.

The second assertion is established with the same argument. Since $P$ is an $A$-torsor, there exists an equivalence $P_{F^{s}} \simeq A_{F^{s}}$. As before we can therefore describe descent data on the trivial $\mu_r$-gerbe as a $1$-cocycle in $\mu_r$-torsors. It allows one to recover a gerbe in $\ker(\Br(P) \to \Br(P_{F^{s}}))[r]$ up to an element of $\Br(F)[r]$.
\end{proof}

\begin{corollary}\label{cor:group_torsor}
Let $F$ be a local field and $r$ a positive integer, such that $r$ is invertible in $F$. Let $A/F$ be an abelian variety, and $P/F$ an $A$-torsor. Then there exists a canonical isomorphism
\[
  \frac{\ker(\Br(A) \to \Br(A_{F^{s}}))[r]}{\Br(F)[r]} \simeq \frac{\ker(\Br(P) \to \Br(P_{F^{s}}))[r]}{\Br(F)[r]}.
\]

\end{corollary}

\begin{construction} \label{longseq}
 Using Theorem \ref{tateduality} one sees that for an isogeny $\phi\colon A\to B$ of abelian varieties over $F$, the exact sequence $0\to \ker(\phi) \to A \to B \to 0$ induces the following exact sequence of cohomology groups:
\begin{multline}
  \label{isogeny}
    0\to  \ker(\phi)(F)\to A(F) \xrightarrow\phi B(F) \to H^1(F,\ker(\phi))  
    \to H^1(F,A) \to H^1(F,B) \to H^2(F,\ker(\phi))\to 0
\end{multline}
\end{construction}

\begin{lemma} \label{longseqduality}
  Consider an isogeny $\phi\colon A\to B$ of abelian varieties over $F$ as well as its dual isogeny $\phi^\vee\colon B^\vee \to A^\vee$ with kernel $\ker(\phi^\vee)\cong \ker(\phi)^\vee$. The associated long exact sequences \eqref{isogeny} are dual to each other under Tate duality in the sense that they fit into a commutative diagram
  \begin{equation*}
    \xymatrix@C-=0.25cm{
         \ker(\phi)(F)\ar[r]  \ar[d]^\cong& A(F) \ar[r]^\phi \ar[d]^\cong & B(F) \ar[r] \ar[d]^\cong & H^1(F,\ker(\phi)) \ar[r] \ar[d]^\cong & H^1(F,A) \ar[r] \ar[d]^\cong & H^1(F,B) \ar[r] \ar[d]^\cong & H^2(F,\ker(\phi)) \ar[d]^\cong\\
         H^1(F,\ker(\phi)^\vee)^* \ar[r] & H^1(F,B^\vee)^* \ar[r] & H^1(F,A^\vee)^* \ar[r] & H^1(F,\ker(\phi)^\vee)^*  \ar[r] & B^\vee(F)^* \ar[r] & A^\vee(F)^* \ar[r] & \ker(\phi)^\vee(F)^*,
}
  \end{equation*}
in which the vertical isomorphisms are given by Tate duality.
\end{lemma}

\begin{proof}
  The commutativity of the squares not involving boundary maps is an instance of the naturality of the Tate duality pairings. That these pairings are also compatible with boundary maps follows from their construction. This can be seen as follows: All the pairings appearing arise from the natural pairing 
  \begin{equation}
    \label{ExtPairing}
    \Ext^i(C,\Gb_m) \times H^{2-i}(F,C) \to H^2(F,\Gb_m) \cong \Qb/\Zb,
  \end{equation}
which exists for an \'etale sheaf $C$ of abelian groups on $F$, and is given by composition in the derived category of such sheaves as in Remark \ref{TDCons}. These induce the Tate duality pairings via natural isomorphisms $H^i(F,A^\vee) \cong \Ext^{i+1}(A,\Gb_m)$ for an abelian variety $A$ (c.f. \cite[Lemma I.3.1]{MilneADT}) and $H^i(F,M^\vee)\cong \Ext^i(M,\Gb_m)$ for an \'etale commutative finite group scheme $M$ over $F$ (c.f. the proof of \cite[Cor. I.2.3]{MilneADT}). Thus is suffices to check the compatibility of \eqref{ExtPairing} with boundary maps. This follows from a direct verification using the description as composition in the derived category.
\end{proof}

An isogeny $\phi\colon A \to B$ of abelian varieties $A$ and $B$ is said to be \emph{self-dual}, if there exists an isomorphism $\psi\colon A \cong B^\vee$ such that the diagram
\[
\xymatrix{
A \ar[r]^{\phi} \ar[d]_{\psi}^{\simeq} & B \ar[d]^{\psi^{\vee}}_{\simeq} \\
B^{\vee} \ar[r]^{\phi^{\vee}} & A^{\vee}
}
\]
commutes.
For such an isogeny there are canonical isomorphisms
\begin{equation} \label{SelfDualKernel}
  \ker(\phi)^\vee \cong \ker(\phi^\vee) \cong \ker(\phi).
\end{equation}

The proposition below plays a key role in the proof of our main result.

\begin{proposition}\label{prop:self-dual}
  Let $A\xrightarrow\phi B$ be a self-dual isogeny of abelian varieties over $F$ whose kernel has order prime to $p$. Then we have $|B(F)/\phi(A(F))|=|\ker(\phi)(F)|$.
\end{proposition}
\begin{proof}
Let $K\coloneqq A(F)/\phi(B(F))\subset H^1(F,\ker(\phi))$ and $Q\coloneqq H^1(F,\ker(\phi))/K$. By Theorem \ref{tateduality} the Pontryagin dual of $H^1(F,A)\to H^1(F,B)$ is isomorphic to $B^\vee(F)\xrightarrow{\phi^\vee} A^\vee(F)$, which in turn by our assumption is isomorphic to $A(F)\xrightarrow{\phi} B(F)$. By the sequence \eqref{isogeny} we have an isomorphism $$Q\cong \ker(H^1(F,A)\to H^1(F,B)).$$ Thus taking Pontryagin duals gives an isomorphism $$Q^*\cong \coker(A(F) \xrightarrow\phi B(F))\cong K.$$ Since $Q$ is finite we thus get $|Q|=|Q^*|=|K|$. Hence the exact sequence
\begin{equation*}
  0\to K \to H^1(F,\ker(\phi)) \to Q \to 0
\end{equation*}
implies $|H^1(F,\ker(\phi))|=|K|^2$. 

On the other hand, since by assumption $\ker(\phi)\cong \ker(\phi)^\vee$, by Theorem \ref{euler} we have $|H^1(F,\ker(\phi))|=|\ker(\phi)(F)|^2$. Thus $|K|=|\ker(\phi)(F)|$, which is what we wanted.
\end{proof}


\section{$p$-adic integration}

\subsection{Basic $p$-adic integration}\label{stringy}


As before we fix a local field $F$. We write $|\cdot|$ for its non-archimedean norm and $\mu_F^n$ for the Haar measure on $F^n$ with the usual normalisation $\mu_F^n(\Oo_F^n)=1$. An $F$-analytic manifolds (or $F$-manifold) is essentially defined the same way as over the real numbers, as explained in \cite{MR1743467}: it is a second countable Hausdorff space with an atlas consisting of charts homeomorphic to open subsets of $F^n$, such that the change of coordinate functions  are locally expressible by convergent power series. Similarly one defines $F$-analytic differential forms on $F$-analytic manifolds.

The main example we are interested in comes from algebraic geometry: for a smooth algebraic variety $X$ over $F$ the set of $F$-rational points $X(F)$ admits the structure of an $F$-manifold and any algebraic $n$-form on $X$ induces an  $F$-analytic differential form on $X(F)$.

Given an $n$-dimensional $F$-manifold $X$ and a global section $\omega$ of $(\Omega^n_X)^{\otimes r}$ we can define a measure $d\mu_\omega$ as follows: Given a compact open chart $U \hookrightarrow F^n$ of $X$ and an analytic function $f: U \rightarrow F$ such that $\omega_{|U} = f(x)(d_{x_1} \wedge d_{x_2}\wedge\dots \wedge d_{x_n})^{\otimes r}$ we set
\[ \mu_{\omega}(U)  = \int_U |f|^{1/r} d\mu^n_{F}.\]
This extends to a measure on $X$ as in \cite[3.2]{Yasuda:2014aa}.

For an $F$-manifold $X$ we will denote $\Omega^{\dim(X)}_X$ by $\Omega^\text{top}_X$. We call a nowhere vanishing section of $(\Omega^{\text{top}}_X)^{\otimes r}$ for some $r$ an $r$-gauge form. As in the real case they can be divided in a relative setting. Indeed, for a submersion $f\colon X \to Y$ there is a sheaf of relative top degree forms $\Omega^\text{top}_f = \bigwedge^{\text{top}}\Omega^1_f$. The short exact sequence
$$f^*\Omega_Y^1 \hookrightarrow \Omega_X^1 \twoheadrightarrow \Omega_f^1$$
yields a canonical isomorphism of line bundles
\begin{equation}\label{divided}f^*\Omega_Y^{\text{top}} \otimes \Omega_f^{\text{top}} \simeq \Omega^{\text{top}}_X.\end{equation}
It allows us to assign to a section $\theta$ of $\Omega_f^{\top}$ and $\omega$ of $\Omega_Y^{\top}$ a top degree form $\theta \wedge f^*\omega$.

\begin{proposition} \label{Igusa} Let $f: X \rightarrow Y$ be a submersion of $F$-manifolds and $\omega_X, \omega_Y$ two $r$-gauge forms on $X$ and $Y$. Then there exists a unique analytic section $\theta$ of the sheaf $(\Omega_{f}^{\top})^{\otimes r}$ on $Y$, such that 
$$\theta \wedge f^*\omega_Y = \omega_X$$
with respect to \eqref{divided},
and for an integrable function $\alpha:X \rightarrow \Cb$ we have   
 \[\int_X \alpha \  d\mu_{\omega_X} = \int_Y \left(\int_{f^{-1}(y)} \alpha \ d\mu_{\theta_y}\right) \ d\mu_{\omega_Y},\]
 where $\theta_y$ denotes the restriction of $\theta$ to the fibre $f^{-1}(y)$.
Furthermore, if $X,Y$ are the $F$-rational points of smooth varieties, the submersion $f$ is induced by a smooth morphism and $\omega_X, \omega_Y$ algebraic $r$-gauge forms, then $\theta$ stems from an algebraic section of $(\Omega_{f}^{\top})^{\otimes r}$.
\end{proposition}
\begin{proof} 

In the algebraic case we abuse notation and write also $X$ and $Y$ for the underlying varieties over $\Spec{F}$. We consider the isomorphism of line bundles
\[ (\Omega^{\text{top}}_{X/F})^{\otimes r} \cong (f^*\Omega^{\text{top}}_{Y/F})^{\otimes r} \otimes (\Omega^{\text{top}}_{X/Y})^{\otimes r},\]
which are understood to be sheaves of either algebraic or analytic forms according to the situation. Let $\theta$ be the unique section of $(\Omega^{\text{top}}_{X/Y})^{\otimes r}$ such that $\theta_X = \theta_Y \otimes \theta$. For each $y\in Y$ the section $\theta$ restricts to an $r$-gauge form $\theta_y$ on $f^{-1}(y)$ and we claim that these sections have the required property. Since the above isomorphism of line bundles is compatible with analytification it suffices to show this in the analytic case.

The analytic case for $r=1$ follows from \cite[Theorem 7.6.1]{MR1743467} by inspecting the proof of \textit{loc.cit.} and verifying that the gauge forms $\theta_y$ constructed there coincide with the ones constructed above. For $r>1$ we note that it suffices to prove that the $r$-gauge forms $\theta_y$ have the required property locally on $X$ and $Y$ for the analytic topology. The following claim shows that once we restrict to suitable analytically open subsets we may assume that $\omega_X$ and $\omega_Y$ have $r$-th roots.
\begin{claim}\label{localroot}
Let $\omega$ be an $r$-gauge form on an $F$-analytic manifold $X$. For every $x \in X$ there exists an $F$-analytic $1$-gauge form $\eta$ defined on an open neighbourhood of $x$, such that $|\eta^{\otimes r}| = q^{m}\cdot{}|\omega|$ with $m \in \Qb$.
\end{claim}
\begin{proof}
By choosing a chart containing $x$ we may assume that $X$ is an open subset $U$ of $F^n$ and represent $\omega$ as $g\cdot (dx_1 \wedge \cdots \wedge dx_n)^{\otimes r}$, where $g$ is locally expressible by an analytic power series. We define $m = {\nu_F(g(x))}$. Since $g\colon U \to F$ is continuous, there exists an open neighbourhood $V$ of $x$, such that for all $y \in V$ one has
$$\nu_F(g(y)) = \nu_F(g(x)) \Leftrightarrow |g(y)|_F = |g(x)|_F = q^{-m}.$$
We conclude that for all $y \in V$ one has $1=|(dx_1\wedge \cdots \wedge dx_n)^{\otimes r}| = q^m\cdot{} |g (dx_1\wedge \cdots \wedge dx_n)^{\otimes r}|$, so we can take $\eta$ to be $(dx_1\wedge \cdots \wedge dx_n)^{\otimes r}$.
\end{proof}

This reduces the assertion to the case $r=1$ proven in \emph{loc. cit.}
\end{proof}

\begin{lemma} \label{VarChange}
  Let $f\colon X\to Y$ be an isomorphism of $F$-manifolds and $\omega_Y$ an $r$-gauge form on $Y$. For any  integrable function $\alpha\colon Y\to \Cb$ we have
  \begin{equation*}
    \int_{X}\alpha\circ f\; d\mu_{f^*\omega_Y} = \int_{Y} \alpha\; d\mu_{\omega_Y}.
  \end{equation*}
\end{lemma}
\begin{proof}
  For $r=1$ this is proven in \cite[Section 7.4]{MR1743467}. For $r>1$ by working locally on $X$ we may assume that $\omega_Y$ has an $r$-th root on $Y$, see Claim \ref{localroot}. This allows us to reduce the lemma to the case $r=1$.
\end{proof}

We will briefly recall $p$-adic integration on (not necessarily smooth) $\Oo_F$-varieties, where we essentially follow \cite[Section 4]{Yasuda:2014aa}. For a $\Oo_F$-variety $X$ we write $X_F=X\times_{\Spec(\Oo_F)} \Spec(F)$ and $X_{k_F}=X\times_{\Spec(\Oo_F)} \Spec(k_F)$. Let $X_F^{\text{sm}}$ be the smooth locus of $X_F$ and set
\[ X^\circ= X(\Oo_F) \cap X_F^{\text{sm}}(F),\]
where we think of $X(\Oo_F)$ as a subset of $X(F) = X_F(F)$. Then $X^\circ$ has naturally the structure of an $F$-manifold. Thus we can integrate any section $\omega \in H^0(X_F^{\text{sm}},(\Omega^{\text{top}}_{X/F})^{\otimes r})$ on $X^\circ$. This way we obtain a measure $\mu_\omega$ on $X^\circ$, which we extend by zero to all of $X(\Oo_F)$.

The following facts will be essential for manipulating $p$-adic integrals:

\begin{proposition}\cite[Lemma 4.3, Theorem 4.8]{Yasuda:2014aa}\label{intools} \begin{enumerate}
\item\label{measurezero} For any subscheme $Y \subset X$ of positive codimension $\mu_{\omega}(Y(\Oo_F)) = 0$.
\item\label{chovg} Let $f:Y \rightarrow X$ be a morphism of $\Oo_F$-varieties. Assume $Y$ admits a generically stabiliser-free action by a finite group $\Gamma$, the morphism $f$ is $\Gamma$-invariant and $Y/\Gamma \rightarrow X$ is birational. Then for any open $\Gamma$-invariant subset $A \subset Y(\Oo_F)$ and any $r$-gauge form $\omega$ on $X^\circ$ we have
\[ \frac{1}{|\Gamma|} \int_{A} |f^*\omega|^{1/r} = \int_{f(A)} |\omega|^{1/r}.\]
\end{enumerate}
\end{proposition}


\subsection{Twisting by torsors}\label{twisting}


This subsection is independent of $p$-adic integration, but will be used in the proof of Theorem \ref{thm:appendix} below and also later in Section \ref{cohoact}.

We fix a scheme $S$, and a commutative \'etale group $S$-scheme $\Gamma$. In our applications, the scheme $S$ will be either $\Spec F$ or $\Spec \Oo_F$, and $\Gamma$ a constant group scheme over $S$. 

\begin{definition}\label{defi:twisting}
Let $N$ be an $S$-scheme endowed with a $\Gamma$-action. For a torsor $T \in H^1_{\text{\'et}}(S,\Gamma)$ we define the $T$-twist of $N$ to be the $S$-space $N_T = [(N \times_S T)/\Gamma]$, where $\Gamma$ acts on the fibre product $N \times_S T$ anti-diagonally.  The group scheme $\Gamma$ acts on $N_T$ through its action on $T$.
\end{definition}

We emphasise that $N_T$ is an $S$-space, since it is stabiliser-free. It is an algebraic $S$-space, since $\Gamma$ is assumed to be \'etale (see \cite[Tag 06DC]{stacks-project}).

In general, several $S$-schemes $N'$ with $\Gamma$-action may yield isomorphic quotient stacks $[N'/\Gamma]$. Twisting is a way to produce such examples, as the next lemma shows. 

\begin{lemma}\label{twist_quotient}
There exists a natural equivalence of $S$-stacks $[N/\Gamma] \simeq [N_T/\Gamma]$.
\end{lemma}

\begin{proof}
We have $[N_T/\Gamma] \simeq [[(N \times_S T)/_1\Gamma]/_2\Gamma]$, where the first $\Gamma$ acts anti-diagonally, and the second $\Gamma$ acts only on the first component. Lemma \ref{exchange} below allows us to exchange the two quotients and we obtain
$$[[(N \times_S T)/_2\Gamma]/_1\Gamma] \simeq [[N/\Gamma] \times_S T] / \Gamma \simeq [N/ \Gamma],$$
where we have used that $\Gamma$ acts trivially on $[N/\Gamma]$, and through the standard action on $T$. Since $T$ is a $\Gamma$-torsor, we have $[T/\Gamma] \simeq S$. 
\end{proof}

\begin{lemma}\label{exchange}
Let $N$ be an $S$-scheme endowed with commuting actions of fppf group $S$-schemes $\Gamma_1$ and $\Gamma_2$. Then there exists an equivalence 
$[[N/\Gamma_1]/\Gamma_2] \simeq [[N/\Gamma_2]/\Gamma_1].$
\end{lemma}

\begin{proof}
Let $T$ be an $S$-scheme. By definition, a morphism $T \to [[N/\Gamma_1]/\Gamma_2]$ relative to the base $S$ is given by a $\Gamma_2$-equivariant morphism
$$P \to [N/\Gamma_1],$$
where $P \to T$ is a $\Gamma_2$-torsor. Unravelling this further, we see that this corresponds to a $\Gamma_1$-equivariant morphism $Q \to N$, where $Q \to P$ is a $\Gamma_1$-torsor endowed with a $\Gamma_2$-equivariant structure on $P$. Faithfully flat descent theory implies that $Q$ is the pullback of a $\Gamma_1$-torsor $P' \to S$, and that we have a $\Gamma_1$-equivariant morphism $P' \to N$.

This shows that morphisms $T \to [[N/\Gamma_1]/\Gamma_2]$ are equivalent to the groupoid of triples $(P,P',\phi)$, where $P$ is a $\Gamma_2$-torsor on $S$, $P'$ a $\Gamma_1$-torsor on $S$, and $\phi$ a $(\Gamma_2 \times \Gamma_1)$-equivariant morphism $P \times P' \to N$.

The same argument as above relates this to the groupoid of morphisms $T \to [[N/\Gamma_2]/\Gamma_1]$. This shows that $[[N/\Gamma_2]/\Gamma_1]$ and $[[N/\Gamma_1]/\Gamma_2]$ are equivalent. 
\end{proof}

We record the following assertion for later use.

\begin{proposition}\label{prop:twists}
Let $U$ be a variety over a finite field $k$ with an action of a finite abelian group $\Gamma$. Let $T \in H^1(k,\underline \Gamma) = \Gamma$ be a $\underline \Gamma$-torsor corresponding to an element $\gamma \in \Gamma$. Then, there is an isomorphism of $\overline{\Qb}_{\ell}$-vector spaces
$$H^i_{\text{\'et}}(U,\overline{\Qb}_{\ell}) \simeq H^i_{\text{\'et}}(U_T,\overline{\Qb}_{\ell})$$
between the \'etale cohomology groups of $U$ and $U_T$ with respect to which the Frobenius operators are related by the formula
$$\mathsf{Fr}_{U_T} = (\gamma^*)^{-1} \cdot \mathsf{Fr}_U.$$
\end{proposition}

\begin{proof}
This follows directly from the definition of twists: Over the algebraic closure $\overline{k}$ a section of $T$ induces an isomorphism between the schemes $U_{\overline{k}}$ and $(U_T)_{\overline{k}}$, and with respect to this isomorphism the Frobenius morphisms are related by multiplication with $\gamma$.
\end{proof}

\begin{corollary}\label{todes} For a smooth variety $M$ over a finite field $k$ with an action of a finite abelian group $\Gamma$ we have
\[ |[M/ \Gamma](k)| = \frac{1}{|\Gamma|} \sum_{T \in H^1(k,\Gamma)} |M_T(k)|.\]
\end{corollary}


\subsection{The orbifold measure} \label{sectorbm}


We start with some basic terminology for finite global quotient stacks.  For a ring $R$ and an integer $m$, we say that $R$ \emph{contains all roots of unity of order $m$} if for every homomorphism from $R$ to an algebraically closed field $\bar k$, the group $\mu_m(R)$ surjects onto $\mu_m(\bar k)$.

\begin{definition} \label{AdmStack}
  Let $R$ be a ring. A finite abelian quotient stack $\Mm$ over $R$ is \emph{admissible} if it admits a presentation $\Mm\cong [Y/\Gamma]$ by a smooth quasi-projective $R$-variety $Y$ with a generically fixed-point free action by a finite abelian group $\Gamma$ whose order is invertible in $R$ and such that $R$ contains all roots of unity of order $|\Gamma|$. In this case we write $M$ for the geometric quotient $Y/\Gamma$. 
\end{definition}

\begin{lemma} \label{gorenstein}
  Let $\Mm$ be an admissible finite abelian quotient stack over a ring $R$.  Then $M$ is $\Qb$-Gorenstein.
\end{lemma}
\begin{proof}
We have to show that $M$ is normal and that its canonical divisor $K_M$ is $\Qb$-Cartier. By \cite[0.2 (2)]{mumford1994geometric} the scheme $M$ is normal. Lemma 5.16 in \cite{kollar2008birational} yields that every Weil divisor on $M$ is a $\Qb$-Cartier divisor. This concludes the proof.
\end{proof}

Let $\Mm$ be an admissible finite abelian quotient stack over $\Spec(\Oo_F)$ and fix a presentation $\Mm=[Y/\Gamma]$ as in Definition \ref{AdmStack}. We write $\pr \colon Y \rightarrow M=Y/\Gamma$ for the quotient morphism and $\Delta \subset Y$ the locus on which $\Gamma$ does not act freely and $U = Y \setminus \Delta$ for the complement. The $F$-manifold
\[ M(\Oo_F)^\natural = M(\Oo_F)\cap \pr(U)(F)\]
admits a measure through the following construction.

\begin{construction}[Orbifold measure] \label{muorb}

By Lemma \ref{gorenstein} the quotient $M$ is normal and hence we have canonical $\Qb$-Weil divisors $K_Y$ and $K_{M}$. By \cite[Lemma 7.2]{Yasuda:2014aa} there exists a unique $\Qb$-Weil divisor $D$ on $M$ such that $K_{M}+D$ is $\Qb$-Cartier and such that the pullback $\pr^*(K_{M}+D)$ is equal to $K_Y$. In fact the proof of \textit{loc.cit.} shows that $D$ can be expressed in terms of the pushforward of the ramification divisor in $Y$. We let $\mu_\text{orb}$ be the measure on $M(\Oo_F)$ associated to the pair $(M,D)$ by \cite[Definition 4.7]{Yasuda:2014aa}. It is defined as follows: Pick $r\geq 1$ such that $r(K_{M}+D)$ is a Cartier divisor. Then $r(K_{M}+D)$ gives rise to a line bundle $I$ on $M$.\\
Over $\pr(U)$ the line bundle $I$ is a subsheaf of $(\Omega^\text{top}_{\pr(U)})^{\otimes r}$ and locally on $M(\Oo_F)^\natural$ the measure $\mu_{\text{orb}}$ is given by integrating a non-zero section of $I$. As in \cite[4.1]{Yasuda:2014aa} one sees that this gives rise to a well-defined measure on $M(\Oo_F)^\natural$, which can then by extended by zero to all of $M(\Oo_F)$. 
One can check that the divisor $D$, and hence the measure $\mu_\text{orb}$, are independent of the choice of the presentation of $\Mm$.
\end{construction}

We now discuss two special cases, in which $\mu_{\text{orb}}$ can be described more explicitly.

\begin{rmk}\label{orbmes} If the canonical bundle $\Omega_Y^{\text{top}}$ is trivial one has a global form $\omega_{\text{orb}}$ on $M(\Oo_F)^\natural$ that computes $\mu_{\text{orb}}$. Namely if $\omega \in H^0(Y,\Omega_Y^{\text{top}})$ is a nowhere vanishing global section, its norm
\[ \Nm(\omega) = \bigotimes_{\gamma \in \Gamma} \gamma^*\omega \in H^0\left(Y,(\Omega_Y^{\text{top}})^{\otimes |\Gamma|}\right),  \]
is $\Gamma$-invariant. Thus $\Nm(\omega)$ descends to a global section $\omega_{\orb}$ of $(\Omega_{\pr(U)}^{\text{top}})^{\otimes |\Gamma|}$, which extends to a section of the line bundle $I$ appearing in Construction \ref{muorb}. In particular one has $\pr^*\omega_{\orb} =\Nm( \omega)$.
\end{rmk}

\begin{rmk}\label{cod2} Assume that $\codim_Y \Delta \geq 2$ and that $\Omega^{\text{top}}_{\pr(U)}$ is trivialised by some volume form $\omega$. Then the orbifold measure $\mu_{\text{orb}}$ is given by integrating $|\omega|$ over $M(\Oo_F)^\natural \subset U(F)$.

Indeed $\codim_Y \Delta \geq 2$ implies that $D=0$ (since $\Delta$ cannot contain the support of a non-zero divisor) in Construction \ref{muorb} and thus $\omega^r$ is a local non-zero section of the line bundle $I$ which restricts to $\Omega^{\text{top}}_{\pr(U)}$ in the complement of $\Delta$. Hence we have for every measurable set $A \subset M(\Oo_F)^\natural$ 
\[\mu_{\text{orb}}(A) = \int_A |\omega^r|^{1/r} = \int_A |\omega|.\]
\end{rmk}

We now study the volume of $M(\Oo_F)$ with respect to $\mu_{\text{orb}}$, where the inertia stack $\I\Mm_{k_F}$ of the special fibre $\Mm_{k_F} = \Mm \times_{\Spec(\Oo_F)} \Spec(k_F)$ naturally appears. Namely one has a specialisation map
\begin{equation*}
  e: M(\Oo_F)^\natural \rightarrow  \I\Mm(k_F)_\iso,
\end{equation*}
where as before $\I\Mm(k_F)_\iso$ denotes the set of isomorphism classes of the groupoid $\I\Mm(k_F)$.

\begin{construction}[Specialisation map] \label{SpecMap}
 To any $\phi \in X(\Oo_F)^\natural$ we associate the $\Gamma$-torsor $T_\phi$ over $\Spec(F)$ given by the fibre of $Y \rightarrow M=Y/\Gamma$ over $\phi_{|F}$. Consider the commutative diagram
\[ \xymatrix{T_\phi \ar[d] \ar[r] & \Spec(\Oo_{T_\phi}) \ar[d] \ar[r]^-{\tilde{\phi}} & Y \ar[d]^{\pr}\\
							\Spec(F) \ar[r] & \Spec(\Oo_F)  \ar[r]^-{\phi} & M,
}\]
constructed as follows: the ring $\Oo_{T_{\phi}}$ is the normalisation of $\Oo_F$ inside the ring of global sections $\Gamma(T_\phi)$ of $T_\phi$. More concretely, expressing the \'etale $F$-algebra $\Gamma(T_{\phi})$ as a product of fields $L^{\times i}$, where $L/F$ is a finite Galois extension, and therefore itself a local field, we define $\Oo_{T_{\phi}}$ to be the product $\Oo_{L}^{\times i}$. The $\Gamma$-action on $T_\phi$ extends uniquely to a $\Gamma$-action on $\Spec(\Oo_{T_\phi})$.

 The morphism $\tilde{\phi}$ is the unique one extending the inclusion $T_\phi \to Y$. Here the fact that $\pr$ is finite and hence in particular proper allows one to deduce the existence of $\tilde{\phi}$ by applying the evaluative criterion of properness to each discrete valuation ring factor $\Oo_L$ of $\Oo_{T_{\phi}}$. Uniqueness follows from separatedness of $Y$. The uniqueness of $\tilde \phi$ implies that it is $\Gamma$-equivariant, which gives a morphism of quotient stacks
\[ [\Spec(\Oo_{T_{\phi}})/\Gamma] \rightarrow \Mm.\]
The stabiliser group of the unique closed point $x$ of this stack satisfies 
\begin{equation}\label{noncind} \Aut_x([\Spec(\Oo_{T_{\phi}})/\Gamma]) = I_{L/F},\end{equation}
where $I_{L/F}$ denotes the inertia group of the extension $L/F$. We refer the reader to Lemma \ref{lemma:BI} where we give a detailed proof of a closely related statement.


By local class field theory (see \eqref{RecDiag}), the inertia group $I_{L/F} \subset \Gamma$ is cyclic and receives a canonical surjective map $\mu(F) \twoheadrightarrow I$. Hence our choice of a primitive $|\Gamma|$-th root of unity $\zeta$ yields a generator $\gamma$ of $I$. The morphism $\Aut_x \to \I\Mm_{k_F}$ together with the generator $\gamma$ of $I_{L/F}$ define a point in $\I\Mm(k_F)$. 
This yields a map
$$e\colon M(\Oo_F)^{\natural} \to \I\Mm(k_F)_\iso.$$
\end{construction}

The following theorem explains how the orbifold volume is related to the stringy point count.


\begin{theorem}\label{thm:appendix} Let $\Mm=[Y/\Gamma]$ be an admissible finite abelian quotient stack over $\Oo_F$. For any $\gamma \in \Gamma$ and any $x \in [Y^\gamma / \Gamma](k_F) \subset \I\Mm(k_F)$ we have
\[ \mu_{\orb}(e^{-1}(x)) = \frac{q^{-w(\gamma,x)}}{|\Aut(x)|}.\]

\end{theorem}

The proof of Theorem \ref{thm:appendix} will take up the rest of this section. We will first consider two special cases:

\subsubsection*{The affine case} \label{AffineCase}
Assume first that $M=\Ab^n$, that $\Gamma = \left\langle \gamma \right\rangle$ is cyclic of order $d$ and that $\gamma$ acts on $\Ab^n$ diagonally and non-trivial by $\gamma \cdot (x_1,x_2,\dots,x_n) = (\zeta^{c_1}x_1,\zeta^{c_2}x_2,\dots,\zeta^{c_n}x_n)$, where we choose ${1\leq c_i \leq d}$ (see Remark \ref{fermionicshift}). 

\begin{proposition}\label{localcase} Let $0$ denote the image of the origin in $[(\Ab^n)^\gamma/\Gamma](k_F)$. Then we have 
\[\mu_{\orb}(e^{-1}(0)) =\frac{q^{-w(\gamma,0)}}{d}.\]
\end{proposition}
\begin{proof} The proof is similar to \cite[Section 2]{DL2002}.
Let $\pi \in F$ be a uniformiser and define $\lambda: \Ab^n \to \Ab^n/\Gamma$ on the level of coordinate rings as
\[f(x_1,\dots,x_n) \mapsto f(\pi^{\frac{c_1}{d}}x_1,\dots,\pi^{\frac{c_n}{d}}x_n). \]
Then by \cite[(2.3.4)]{DL2002} one has $\lambda(\Ab^n(\Oo_F)) \cap (\Ab^n/\Gamma)(\Oo_F)^\natural = e^{-1}(0)$ and hence by Proposition \ref{intools}
\[ \mu_{\orb}(e^{-1}(0)) = \int_{e^{-1}(0)} |\omega_{\orb}|^{\frac{1}{d}} = \frac{1}{d}\int_{\Ab^n(\Oo_F)} |\lambda^*\omega_{\orb}|^{\frac{1}{d}}, \]
with $\omega_{\orb}$ satisfying $\pr^*\omega_{\orb} = (dx_1 \wedge \dots \wedge	 dx_n)^{\otimes d}$, see Remark \ref{orbmes}. Now a direct verification shows $\lambda^*\omega_{\text{orb}} = \pi^{\sum_i c_i} (dx_1 \wedge \dots \wedge dx_n)^{\otimes d}$ and hence
\[\frac{1}{d}\int_{\Ab^n(\Oo_F)} |\lambda^*\omega_{\orb}|^{\frac{1}{d}} = \frac{1}{d}\int_{\Oo_F^n} |\pi^{\sum_i c_i}|^{\frac{1}{d}} dx_1\dots dx_n = \frac{q^{-w(\gamma,0)}}{d}. \]
\end{proof}

\subsubsection*{The cyclic case}
Now we assume that as before $\Gamma = \left\langle \gamma \right\rangle$ is cyclic of order $d$, but $Y$ is now any smooth $\Oo_F$-variety as in Theorem \ref{thm:appendix}.

\begin{proposition}\label{cycliccase} For any $x \in [Y^\gamma/\Gamma](k_F) \subset \I\Mm(k_F)$ we have
\[ \mu_{\orb}(e^{-1}(x)) = \frac{q^{-w(\gamma,x)}}{d}.\]
\end{proposition}

This follows from Proposition \ref{localcase} by linearising the $\Gamma$-action around $x$. More precisely we have the following well-known lemma (even without assuming that $\Gamma$ is cyclic).
\begin{lemma}\label{locetal} Let $x \in Y(k_F)$ be a closed point fixed by $\Gamma$. There exists a $\Gamma$-invariant open neighbourhood $U \subset Y$ of $x$ and an \'{e}tale morphism $f:U \rightarrow \mathbb{A}^n_{\Oo_F}$ such that the following holds:
\begin{enumerate}
\item[(i)] There exists a diagonal action of $\Gamma$ on $\mathbb{A}^n_{\Oo_F}$ with respect to which the morphism $f$ is $\Gamma$-equivariant.

\item[(ii)] The morphism $f$ induces a $F$-analytic diffeomorphism  
\[\bar{f}: \{\phi\in Y/\Gamma(\Oo_F)\ | \ \phi_{|k_F} = x\}  \rightarrow \{ \phi \in \mathbb{A}^n/\Gamma(\Oo_F) \ | \ \phi_{|k_F} = 0\}.\] 
\end{enumerate}  
\end{lemma}
\begin{proof}
Proposition 3.24 in \cite{MR559531} implies the existence of an open neighbourhood $U' \subset Y$ of $x$ together with an \'etale morphism $f'\colon U' \to \Ab^n_{\Oo_F}$ (we implicitly use that the only non-empty open subsets of $\Spec \Oo_F$ are given by the singleton $\Spec F$ or $\Spec \Oo_F$ itself). Without loss of generality we may assume that $f(x) = 0$.

We define $U'' = \bigcap_{\gamma \in \Gamma} \gamma\cdot{} U'$, which is a $\Gamma$-invariant open neighbourhood of $x$. The \'etale morphism $f$ induces an isomorphism of tangent spaces $T_xY \simeq T_0(\Ab^n)$. In particular we obtain a basis for $T_xY$. For $\gamma \in \Gamma$ we denote by $A_{\gamma}$ the matrix of the linear map $d\gamma\colon T_xY \to T_xY$ computed with respect to the aforementioned basis of $T_xY$. We define 
$$f= \sum_{\gamma \in \Gamma} A_{\gamma}^{-1}\circ f' \circ \gamma.$$
The assumption $(p,|\Gamma|) = 1$ implies that the Jacobi matrix $df$ of $f$ at $x$ is still invertible (it is equal to $|\Gamma| \cdot{} df'$). We conclude that there exists an $\Gamma$-invariant open neighbourhood $U \subset M$ of $x$, such that the morphism $f$ is \'etale on $U$. We let $\Gamma$ act on $\Ab^n$ via $\gamma\mapsto A_\gamma$. Then $f$ is $\Gamma$-equivariant, since we have for $\gamma' \in \Gamma$ and $y \in U$ the following computation:
$$f(\gamma' y) = \sum_{\gamma \in \Gamma} A_{\gamma}^{-1} \circ f \circ \gamma (\gamma' y) = A_{\gamma'} \left( \sum_{\gamma \in \Gamma} A_{\gamma\gamma'}^{-1} \circ f \circ (\gamma\gamma') (y) \right)$$
This concludes the proof of \emph{(i)}. To see $(ii)$, note that as $U$ contains $x$, the quotient $U/\Gamma$ contains $\{\phi\in Y/\Gamma(\Oo_F)\ | \ \phi_{|k_F} = x\} $. Using this, $(ii)$ follows from  standard properties of Henselian rings (see \cite[Theorem 4.2]{MR559531}) and \'etale maps.
\end{proof}

\begin{proof}[Proof of Proposition \ref{cycliccase}] If $x \in [Y^\gamma/\Gamma](k_F)$ lies in the image of the quotient map $Y^\gamma(k_F) \to [Y^\gamma/\Gamma](k_F)$ we can directly apply Lemma \ref{locetal} to a lift of $x$. In this case the proposition follows from Lemma \ref{VarChange} and Proposition \ref{localcase}, since the orbifold measures in question are compatible with $f$.

In general $x$ will be in the image of $Y^\gamma_T(k_F) \to [Y^\gamma_T/\Gamma](k_F) = [Y^\gamma/\Gamma](k_F)$ (see Lemma \ref{twist_quotient}) for some $\Gamma$-torsor $T$ on $\Oo_F$. The same argument as before now applies with $Y$ replaced by $Y_T$.
\end{proof}

After this preparation we are now ready for the

\begin{proof}[Proof of Theorem \ref{thm:appendix}] Let $I \subset \Gamma$ denote the cyclic group generated by $\gamma$. As in the proof of Proposition \ref{cycliccase} we can assume that $x$ is in the image of the quotient morphism $Y^\gamma(k_F) \to [Y^\gamma/\Gamma](k_F)$ (otherwise we twist by an unramified $\Gamma$-torsor). This implies that every $\phi \in e^{-1}(x)$ fits into a commutative diagram 
\[ \xymatrix{  \Spec(\Oo_L) \ar[r]^{\tilde{\phi}} \ar[d]& Y  \ar[d] \\
								\Spec(\Oo_F) \ar[r]^\phi & Y/\Gamma ,}\]
where $L/F$ is a totally tamely ramified Galois extension with Galois group $I$ and $\tilde{\phi}$ is $I$-equivariant. In particular we see that $\phi$ lies in the image of the map $Y/I(\Oo_F) \to Y/\Gamma(\Oo_F)$ induced by the quotient morphism $s: Y/I \rightarrow Y/\Gamma$.

Now let $\mu_{\orb,I}$ and $e_I$ denote the orbifold measure and specialisation morphism for $[Y/I]$. Then Proposition \ref{intools} implies
\[ \mu_{\orb}(e^{-1}(x)) = \frac{1}{|\Gamma/I|} \mu_{\orb,I}(s^{-1}e^{-1}(x)).\]
Now let $s_\gamma$ be the quotient morphism $Y^\gamma/I \to Y^\gamma/\Gamma$. Then one can check that there is a decomposition 
\[s^{-1}e^{-1}(x) = \bigsqcup_{\tilde{x} \in s_\gamma^{-1}(x)} e_I^{-1}(\tilde{x}),\] 
and thus Proposition \ref{cycliccase} implies
\[\mu_{\orb}(e^{-1}(x)) = 	\frac{|s_\gamma^{-1}(x)| q^{-w(\gamma,x)}}{|\Gamma|}.\]	
Essentially from the definition of the quotient stack $[Y^\gamma/\Gamma]$ is follows that $|\Aut(x)|=\frac{|\Gamma|}{|s_\gamma^{-1}(x)|}$, which finishes the proof.
\end{proof}



\section{A stack-theoretic approach to Brauer groups of local fields}


The Brauer group of a non-archimedean local field $F$ is isomorphic to $\mathbb{Q}/\mathbb{Z}$ (see Theorem \ref{BrauerLoc} and the references given there). We refer to the map $\inv\colon \Br(F) \to \mathbb{Q}/\mathbb{Z}$ specified in \emph{loc. cit.} as the \emph{Hasse invariant}. In Subsection \ref{arithmeticduality}, devoted to the proof of Main Theorem \ref{p-Mirror}, we will consider integrals
$$\int_{Y(F)} \exp(-2\pi i \inv(\alpha)) |\omega|,$$
where is $Y$ a proper $F$-scheme (a Hitchin fibre) endowed with a top-degree form $\omega$ and $\alpha \in \Br(Y)$ is a gerbe. The integrand above is to be understood to be the function which associates to an $F$-point $y \in Y(F)$ the complex number $\exp(-2\pi i \inv(y^*\alpha))$, where $y^*\alpha \in \Br(F)$ denotes the pullback of $\alpha$ along the morphism $y \colon \Spec F \to Y$.

The $F$-scheme $Y$ itself embeds into an ambient Deligne-Mumford $\Oo_F$-stack $\Mm$, which is a quotient stack $[U/\Gamma]$ of a smooth scheme $U$ by a finite abelian group $\Gamma$ (the total space of the Hitchin fibration), such that the gerbe $\alpha$ on $Y$ extends to $\Mm$. It will be crucial to understand the aforementioned function $y \to \inv(y^*\alpha)$ in terms of the inertia stack $\I\Mm$ and the transgression construction discussed in Subsection \ref{GerbesSection}.

This section provides the backbone for such a comparison and develops a stack-theoretic interpretation of the Hasse invariant of an element $\alpha \in \Br(F)[e]$ where $e$ is positive integer coprime to the residue characteristic $p$ of $F$, such that $F$ contains all $e$-th roots of unity. Our approach can be summarised as follows.

\begin{enumerate}[(a)]
\item For every tame field extension $L/F$ of ramification degree dividing $e$, we construct a Deligne-Mumford $\Oo_F$-stack $\Xc_{L/F}$, such that one has an open immersion $i_F\colon \Spec F \hookrightarrow \Xc_{L/F}$ and a closed immersion $j_{\sigma}\colon B_{k_F}I \hookrightarrow \Xc_{L/F}$, where $I \subset \Gal(L/F)$ denotes the inertia group (Subsection \ref{sub:X_LF}).
\begin{equation}
\xymatrix@!=1cm{
\Spec F\;\; \ar@{^{(}->}[r]^{i_F} & \Xc_{L/F} & \;\; B_{k_F}I \ar@{_{(}->}[l]_{j_{\sigma}}
}
\end{equation}
\item\label{b} Pullback along the open immersion $\Spec F \hookrightarrow \Xc_{L/F}$ induces an isomorphism $\Br(\Xc_{L/F})[e] \simeq \Br(F)[e]$ (Subsection \ref{sub:comp_Br}).
\item\label{d} The transgression construction (\ref{GerbesSection}) gives rise to a map $\Br(\Xc_{L/F}) \to \Zb/e\Zb = \Qb/\Zb$ (Subsection \ref{sub:comp_Hasse}).
\item\label{e} We prove that the composition of the maps $\Br(F)[e] \xrightarrow{\eqref{b}} \Br(\Xc_{L/F}) \xrightarrow{\eqref{d}} \Qb/\Zb[e]$ agrees with the Hasse invariant (Theorem \ref{thm:comp_Hasse}).
\end{enumerate}

Furthermore, for $y \in Y(F)$ as above, we will see in Subsection \ref{sub:application} that there exists $L/F$, such that $y \colon \Spec F \to Y$ factors as
\[
\xymatrix{
\Spec F \ar[r] \ar[d] & Y \ar[d] \\
\Xc_{L/F} \ar[r] & \Mm.
}
\] 
This factorisation together with \eqref{e} above, accomplishes our goal to compare $\inv(\alpha_y)$ to the transgression construction with respect to the stack $\Mm$.

\subsection{The construction of the stack $\Xc_{L/F}$}\label{sub:X_LF}

Henceforth we fix a non-archimedean local field $F$ of residue characteristic $p$ and finite Galois extension $L/F$ of degree $d$ and inertia degree $e$. The latter divides the former: $e|d$. Furthermore, we assume that $(p,e) = 1$, that is, the extension $L/F$ is tame and that $F$ contains all $e$-th roots of unity. We use the notation $q$ for the cardinality of the residue field $k_F$ and refer to the inertia group by $I=I_{L/F}$. Subsection \ref{galois} contains a brief overview of the terminology of local field theory.

At first we observe that the assumptions above impose strong restrictions on the structure of the Galois group $\Gamma=\Gal(L/F)$. 

\begin{lemma}\label{lemma:abelian}
Let $L/F$ be as above. 
\begin{enumerate}[(a)]
\item For a uniformiser $\pi_L$ of $L$, the subfield $F(\pi_L)$ of $L$ is totally ramified, Galois and of degree $e$ over $F$. The Galois group $\Gal(L/F(\pi_L)) $ is a normal subgroup of $\Gamma$, which maps isomorphically to $\Gamma/ I$. \item For $\pi_L$ as in (a), there are natural isomorphisms $\Gamma/I \cong \Zb/\frac{d}{e}\Zb$ and $I \cong \mu_e(F)$.
\item The Galois group $\Gamma$ is a finite abelian group, which is a split extension of a cyclic group with $d/e$ elements by a cyclic group of order $e$. 
\[
\xymatrix{
0 \ar[r] & I \ar[r] \ar[d] & \Gamma \ar[r] \ar[d] & \Gamma/I \ar@{-->}@/_/[l] \ar[r] \ar[d] & 0 \\
0 \ar[r] & \Zb/e\Zb \ar[r] & \Zb/e\Zb \oplus \Zb/\frac{d}{e}\Zb \ar[r] & \Zb/\frac{d}{e}\Zb \ar[r] & 0
}
\]
Hence the choice of $F(\pi_L)$ induces an isomorphism $\Gamma \cong \Gamma/I \times I$.
\item Every splitting of $\Gamma \to \Gamma/I$ arises from a uniformiser of $L$ in the above way.
\end{enumerate}
\end{lemma}
\begin{proof}

(a) By \cite[Proposition II.5.11]{MR1282723} the field $F(\pi_L)$ is totally ramified and of degree $e$ over $F$. Hence by \cite[Proposition II.5.12]{MR1282723} there exists a uniformiser $\pi'_L$ of $F(\pi_L)$ for which $(\pi'_L)^e$ is in $F$. Since by assumption $F$ contains all $e$-th roots of unity, the field $F(\pi_L)$ is then the splitting field of the polynomial $X^e-(\pi'_L)^e \in F[X]$. Thus it is Galois over $F$ and consequently $\Gal(L/F(\pi_L))$ is normal in $\Gamma$. 

Let $L' \subset L$ be the unique subfield, which is unramified over $F$ and has residue field $k_L$. Since $F(\pi_L)$ is totally tamely ramified over $F$ we find $F(\pi_L) \cap L'=F$. The fields $F(\pi_L)$ and $L'$ together generate $L$. By Galois theory this implies that restriction gives an isomorphism $\Gal(L/F(\pi_L)) \cong \Gal(L'/F) \cong \Gamma/I$.

(b) The Frobenius generator of $\Gal(k_L/k_F)$ gives an isomorphism $\Gal(L/F(\pi_L)) \cong \Gal(k_L/k_F) \cong \Zb/ \frac{d}{e} \Zb$. By the above, there is a natural isomorphism $\mu_e(F) \cong \Gal(F(\pi_L)/F),\; \xi \mapsto (\pi_L \mapsto \xi \pi_L)$.

(c) follows from the previous claims.

(d) Let $\Gamma/ I \to \Gamma$ be a splitting. Then the field $L^{\Gamma/I}$ has degree $e$ over $F$ and residue field $k_F$. Thus we may apply the above considerations to this field to get (d).
\end{proof}

We now turn to defining the principal object of study of this section. It is reminiscent of a root stack and provides a partial compactification of $\Spec F$ measuring tame ramification.

\begin{definition}\label{quoto}
We define $\Xc_{L/F}$ to be the Deligne-Mumford stack $[\Spec( \Oo_L)/\Gamma]$. 
\end{definition} 

\begin{lemma} \label{GaloisTorsor}
Let $X \to Y$ be a finite \'etale Galois morphism of schemes with Galois group $G$. The morphism $X \to Y$ induces an equivalence
\begin{equation*}
	[X/G] \cong Y
\end{equation*}
between the quotient stack $[X/G]$ and $Y$.
\begin{proof}
First we claim that $X$ is a $G$-torsor over $Y$. For this we need to verify that a certain morphism $X \times G \to X \times_Y X$ is an isomorphism. Choosing a geometric point $y$ of $Y$, we use the equivalence between finite \'etale schemes over $Y$ and finite sets with a continuous $\pi_1^{\text{\'et}}(Y,y)$-action. It follows from the definition of an \'etale Galois covering that the image of $X \times G \to X \times_Y X$ under this equivalence is an isomorphism. Thus $X$ is a $G$-torsor over $Y$.

By descent, it suffices to prove the claim that $[X/G] \to Y$ is an isomorphism \'etale-locally on $X$. Thus it follows from the above.
\end{proof}

\end{lemma}
The following describes the generic fibre of $\Xc_{L/F}$:
\begin{lemma}\label{lemma:open}
There is an open immersion 
\[
\xymatrix@!=2cm{
\Spec (F)\; \ar@{^(->}[r]^{j_F} & \;\Xc_{L/F}.
}
\]
\end{lemma}

\begin{proof}
By the construction of $\Xc_{L/F}$ as a quotient stack, the morphism $\Spec \Oo_L \to \Xc_{L/F}$ is an atlas. We have a $\Gamma$-invariant open immersion $\Spec L \hookrightarrow \Spec \Oo_L$, which descends to an open immersion
$$[\Spec L/\Gamma] \hookrightarrow \Xc_{L/F}.$$
Since $[\Spec L/\Gamma] \cong \Spec F$ by Lemma \ref{GaloisTorsor} this proves the claim.
\end{proof}

In the definition of $\Xc_{L/F}$ we allowed $L/F$ to be an arbitrary finite Galois extension, satisfying the assumptions stated at the beginning of the subsection. For two non-isomorphic field extensions $L_1,L_2/F$ the associated Deligne-Mumford stacks $\Xc_{L_1/F}$ and $\Xc_{L_2/F}$ can be isomorphic. In fact, as shown by the following lemma, it suffices to work with totally ramified field extensions (that is, $I = \Gamma$).
\begin{lemma}\label{lemma:wlog}
For every uniformiser $\pi_L$ of $L$, the inclusion $F(\pi_L) \into L$ induces an equivalence
$$\Xc_{F(\pi_L)/F} \cong \Xc_{L/F}.$$
\end{lemma}
\begin{proof}
Since $L$ is unramified over $F(\pi_L)$ by Lemma \ref{lemma:abelian}, the morphism $\Spec (\Oo_L) \to \Spec (\Oo_{F(\pi_L)})$ is a finite \'etale Galois cover with Galois group $\Gal(L/F(\pi_L)) \cong \Gamma /I$. Hence by Lemma \ref{GaloisTorsor} there is a natural isomorphism $[\Spec(\Oo_L)/(\Gamma/I)] \cong \Spec(\Oo_{F(\pi_L)})$. Using the isomorphism $\Gamma \cong \Gamma/I \times I$ from Lemma \ref{lemma:abelian} we get isomorphisms
\begin{equation}
[\Spec(\Oo_L)/\Gamma] \cong [[\Spec(\Oo_L)/(\Gamma/I)]/I] \cong [\Spec(\Oo_{F(\pi_L)}) /I] \cong \Xc_{F(\pi_L)/F}.
\end{equation}
This proves the lemma.
\end{proof}
The lemma below describes the special fibre of $\Xc_{L/F}$ defined in the following definition.
\begin{definition}
The closed reduced complement of the open immersion $i_F\colon \Spec F \hookrightarrow \Xc_{L/F}$ of Lemma \ref{lemma:open} is denoted by $\Xc_{L/F}^{\bullet}$. 
\end{definition}

\begin{lemma}\label{lemma:BI}
For every splitting $\sigma\colon \Gamma \to I$ of the inclusion $I \hookrightarrow \Gamma$, we obtain an equivalence of stacks $\Xc_{L/F}^{\bullet} =[\Spec k_L / \Gamma]\simeq B_{k_F}I$, and thus a closed immersion 
$$j_{\sigma}\colon B_{k_F}I = \Xc_{L/F}^{\bullet}\hookrightarrow \Xc_{L/F}.$$
\end{lemma}
\begin{proof}
At first we remark that a splitting $\sigma$ exists by virtue of Lemma \ref{lemma:abelian}.
Furthermore, it is shown there that $\sigma$ is induced by the choice of a subfield $F(\pi_L) \subset L$. Such a field has residue field $k_F$ and Galois group $\Gal(F(\pi_L)/F)$ isomorphic to $I$. By virtue of Lemma \ref{lemma:wlog} we have an equivalence of stacks 
$$\Xc_{L/F} \cong \Xc_{F(\pi_L)/F}$$ 
and thus may assume $L=F(\pi_L)$ and $\Gamma = I$ without loss of generality. The open inclusion $i_F$ corresponds to the quotient of 
$$\Spec L \hookrightarrow \Spec \Oo_L$$
with respect to the $I$-action. The reduced closed compliment of this open immersion is given by 
$$\Spec k_L \hookrightarrow \Spec \Oo_L.$$
Since we assume $L=F(\pi_L)$ is a totally ramified extension of $F$ we obtain $k_L = k_F$. After quotienting by $I$ we obtain a closed immersion
$$B_{k_F} I = [\Spec k_F / I] \hookrightarrow \Xc_{L/F}.$$
This concludes the proof.
\end{proof}

\subsection{\'Etale cohomology groups of $\Xc_{L/F}$}\label{sub:et_cohomology}

In this subsection we compute several \'etale cohomology groups of the stacks $\Xc_{L/F}$. These computations play a role in our determination of the Brauer group of $\Br(\Xc_{L/F})$.

\begin{lemma}\label{lemma:PicX}
The Picard group of $\Xc_{L/F}$ is isomorphic to $\Zb/e\Zb$.
\end{lemma}

\begin{proof}
As we saw in Lemma \ref{lemma:wlog} we have an isomorphism $\Xc_{L/F} \cong \Xc_{F(\pi_L)/F}$, where $F(\pi_F)/F$ is a Galois extension with Galois group equal to the inertial group $I$. Henceforth we assume without loss of generality that $L/F$ is a totally ramified field extension, that is, satisfies $\Gamma = I$.

By virtue of the definition of $\Xc_{F(\pi_L)/F}$ as a quotient stack, we have that $\Pic(\Xc_{F(\pi_L)/F})$ is isomorphic to the set of equivalence classes of $I$-equivariant line bundles on $\Spec \Oo_{F(\pi_L)}$. Every line bundle on $\Spec \Oo_{F(\pi_L)}$ is trivial, as $\Oo_{F(\pi_L)}$ is a local ring. A $I$-equivariant structure on the trivial line bundle corresponds to an element of the Galois cohomology group $H^1(I,\Oo_{F(\pi_L)}^{\times})$. This yields a isomorphism of abstract groups 
$$\Pic(\Xc_{F(\pi_L)/F}) \simeq H^1(I,\Oo_{F(\pi_L)}^{\times}).$$
The short exact sequence of $I$-modules
$$1 \to 1 + \mathfrak{m}_{F(\pi_L)} \to \Oo_{F(\pi_L)}^{\times} \to k_F^{\times} \to 1$$
yields the exact sequence
$$H^0(I,k_{F(\pi_L)}^{\times})= k_F^{\times} \to H^1(I,1 + \mathfrak{m}_{F(\pi_L)}) \to H^1(I,\Oo_{F(\pi_L)}^{\times}) \to H^1(I,k_{F(\pi_L)}^{\times}) \to H^2(I,1 + \mathfrak{m}_{F(\pi_L)}).$$
A priori, the groups $H^i(I,1 + \mathfrak{m}_{F(\pi_L)})$ for $i= 1,2$ are pro-$p$-groups as $1+\mathfrak{m}_{F(\pi_L)}$ is a pro-$p$-group. Since $I$ is a cyclic group of order $e$, which is coprime to $p$, we have that the cohomology groups $H^i(I,M)$ for $i > 0$ vanish for every finite $I$-module on which multiplication by $e$ is invertible. This shows that $H^i(I,1 + \mathfrak{m}_{F(\pi_L)}) =0$ vanishes for $i=1,2$.
\begin{claim}
$H^1(I,k_F^{\times}) \simeq \Zb/e\Zb$
\end{claim}
\begin{proof}
We have $H^1(I,k_F^{\times}) = \Pic([\Spec k_F/I])$, that is, the group of isomorphism classes of $I$-equivariant line bundles on $\Spec k_F$. Since line bundles on $\Spec k_F$ are trivial, we obtain 
$$\Pic([\Spec k_F/I]) \simeq \Hom(I,\G_m) = \Hom(\mu_e,\G_m)=\Zb/e\Zb$$
and this concludes the proof of the claim.
\end{proof}
We therefore obtain an exact sequence
$$0 = H^1(I,1 + \mathfrak{m}_{F(\pi_L)}) \to H^1(I,\Oo_{F(\pi_L)}^{\times}) \to H^1(I,k_{F}^{\times}) \to H^2 (I,1 + \mathfrak{m}_{F(\pi_L)}) = 0$$
concluding the proof of the lemma.
\end{proof}

The conclusion of the following lemma holds for arbitrary non-negative integers $i$. We opted for the inclusion of a conceptual argument for degrees $i=0,1,2$ which avoids computations with spectral sequences.
\begin{lemma} \label{LocalTransgression}
For every splitting $\sigma$ of the embedding $I \hookrightarrow \Gamma$, pullback along the map $j_{\sigma}\colon B_{k_F}I \to \Xc_{L/F}$ induces isomorphisms
$$H^i_{\et}(\Xc_{L/F},\F) \simeq H^i_{\et}(B_{k_F}I,\F) \text{ for $i=0,1,2$},$$
where $\F$ denotes a finite \'etale sheaf of groups on $X$ of order coprime to $p$.
\end{lemma}
\begin{proof}
Pullback along the inclusion $\Spec k_L \hookrightarrow \Spec \Oo_L$ induces an equivalence of finite \'etale sites \cite[Proposition 3.4.4]{MR559531}. In particular, $\iota^*$ induces equivalences between the groupoid of $\Gamma$-equivariant $\F$-torsors. We obtain an isomorphism
$$H^i_{\et}(\Xc_{L/F},\F) \simeq H^i_{\et}(B_{k_F}I,\F) \text{ for $i=0,1$}$$
since $\F$ is a finite \'etale group scheme, and $H^1_{\et}$ is isomorphic to the group of isomorphism classes of \'etale $\F$-torsors, and $H^0_{\et}$ to the group of sections of the finite \'etale group space $\F$, that is, isomorphisms of $\F$ with the trivial $\F$-torsor.

The case of degree $2$ cohomology is dealt with similarly using the language of gerbes.
\begin{claim}
Pullback along the closed immersion $\iota\colon \Spec k_L \hookrightarrow \Spec \Oo_L$ induces an equivalence of the $2$-groupoids of $\F$-gerbes 
\begin{equation}\label{above}\iota^*\colon\mathsf{Gerbe}(\Spec k_L,\F) \simeq \mathsf{Gerbe}(\Spec \Oo_L,\F).\end{equation}
\end{claim}
\begin{proof}
For a $2$-groupoid $\Cc$ we denote by $\pi_0$ the set of isomorphism classes, and for $X \in \Cc$ we write $\pi_1(\Cc,X)$ for the set of $1$-automorphisms of $X$ up to $2$-isomorphisms, and $\pi_2(\Cc,X)$ for the set of $2$-automorphisms of $\id_X$. A functor $F\colon \Cc \to \mathcal{D}$ of $2$-groupoids is an equivalence if and only if $\pi_0(F)\colon \pi_0(\Cc) \to \pi_0(\mathcal{D})$ is an isomorphism, and for every $X \in \Cc$ and $i=1,2$ the induced map $\pi_i(\Cc,X) \to \pi_i(\mathcal{D},F(X))$ is an isomorphism.

For the $2$-groupoids above, one has
$$\pi_0(\mathsf{Gerbe}(U,\F))\simeq H^2_{\et}(U,\F),$$
and for every $\Gg \in \mathsf{Gerbe}(U,\F)$ we have
$$\pi_i(\mathsf{Gerbe}(U,\F),\Gg) \simeq H^{2-i}_{\et}(U,\F)$$
(note that the right hand side does not depend on $\Gg$). 

Pullback along $\iota$ induces isomorphisms $H^i_{\et}(\Spec \Oo_L,\F) \simeq H^i_{\et}(\Spec k_F,\F)$ of \'etale cohomology groups (see \cite[Remark III.3.11]{MR559531}). We therefore deduce that 
the functor $\iota^*$ from \eqref{above} is an equivalence as asserted.
\end{proof}
The equivalence of $2$-groupoids of gerbes asserted by the lemma above yields that $\iota^*$ also induces an equivalence of $2$-groupoids of $\Gamma$-equivariant $\F$-gerbes on $\Spec \Oo_L$ and $\Spec k_L$. We deduce that $H^2_{\et}(\Xc_{L/F},\F) \simeq H^2_{\et}(B_{k_F}I,\F)$.
\end{proof}

\begin{lemma} \label{preLocalTransgression}
  There is a split short exact sequence
  \begin{equation*}
    \Zb/e\Zb \hookrightarrow H^2_{\et}(B_{k_F}I,\mu_{e}) \twoheadrightarrow \Zb/e\Zb.
  \end{equation*}
\end{lemma}

\begin{proof}
The set of isomorphism classes of $\mu_e$-gerbes on $[\Spec k_F/I] = [\Spec k_L/\Gamma]$ are in correspondence with central extensions of finite \'etale group schemes
$$1 \to \mu_e \to \widehat{\Gamma} \to \Gamma \to 1$$
up to isomorphism (see Lemma \ref{thatsalemma}). By part (c) of \emph{loc. cit.} we have a split short exact sequence
$$\Hom(\mu_e,\mu_e) \hookrightarrow H^2_{\et}(B_{k_F}I,\mu_{e}) \twoheadrightarrow \Ext(\mu_e,\mu_e),$$
where the Hom and Ext-groups are meant for abstract abelian groups, since $\mu_e$ is a constant group scheme over $k_F$ by virtue of our standing assumptions. Since $\mu_e$ is a cyclic group of order $e$, we may identify $\Hom(\mu_e,\mu_e)$ and $\Ext(\mu_e,\mu_e)$ with $\Zb/e\Zb$.
%
%
%
%
This concludes the proof of the lemma.
\end{proof}

\begin{lemma}\label{lemma:localTransgression}
A choice of a splitting $\sigma\colon \Gamma \to I$ of the inclusion $I \hookrightarrow \Gamma$ induces a split short exact sequence
  \begin{equation}
    \Zb/e\Zb \hookrightarrow H^2_{\et}(\Xc_{L/F},\mu_{e}) \twoheadrightarrow \Zb/e\Zb.
  \end{equation}
\end{lemma}

\begin{proof}
Every choice of a splitting $\sigma\colon \Gamma \to I$ of the inclusion $I \hookrightarrow \Gamma$ yields an isomorphism 
$$j_{\sigma}^*\colon H^2_{\et}(\Xc_{L/F},\mu_{e}) \simeq  H^2_{\et}(B_{k_F}I,\mu_e)$$ (see Lemma \ref{LocalTransgression}, let $\F=\mu_e$). We apply Lemma \ref{preLocalTransgression} to conclude the assertion above.
\end{proof}

\subsection{On the Brauer group of $\Xc_{L/F}$}\label{sub:comp_Br}

In this subsection we relate the Brauer group of $\Xc_{L/F}$ to $\Br(F)$. A Brauer class $\alpha \in \Br(\Xc_{L/F})$ can be pulled back along the inclusion $i_F\colon \Spec F \hookrightarrow \Xc_{L/F}$. We claim that the resulting Brauer class $i_F^*\alpha$ on $F$ has order dividing $d$. Indeed, its pullback to $L$ extends to $\Spec \Oo_L$. This follows from the commutative diagram
\[
\xymatrix{
\Spec L \ar[r] \ar[d] & \Spec \Oo_L \ar[d] \\
\Spec F \ar[r] & \Xc_{L/F}.
}
\]
Since $\Br(\Oo_L) = 0$ we have $(i_F^*\alpha)_L = 0$. The commutative diagram of Theorem \ref{BrauerLoc}
\[
\xymatrix{
\Br(F) \ar[r]^{\inv}_{\simeq} \ar[d] & \Qb/\Zb \ar[d]^{[d]} \\
\Br(L) \ar[r]^{\inv}_{\simeq} & \Qb/\Zb,
}
\]
where $[d]$ denotes multiplication by $d$,
yields that $i_F^*\alpha$ has order dividing $d$.

\begin{lemma}\label{lemma:Brinj}
Restriction along $i_F$ induces an injective map $i_F^*\colon \Br(\Xc_{L/F}) \hookrightarrow \Br(F)[d]$.
\end{lemma}

\begin{proof}
We have already seen above that $i_F^*$ factors through $\Br(F)[d] \subset \Br(F)$. It remains to show that it is injective. We will show this using the Brauer-Severi construction. For this purpose we let $\alpha \in \Br(\Xc_{L/F})$ be a Brauer class on $\Xc_{L/F}$, such that $i_F^*\alpha = 0$. Let $B$ be an Azumaya algebra on $\Xc_{L/F}$ representing the Brauer class $\alpha$.

To an Azumaya algebra $B$ of rank $r^2$ on a Deligne-Mumford stack $\Xc$, one can associate a $\Pb^r$-bundle, called the Brauer-Severi variety $Y_{B} \xrightarrow{p_{B}} \Xc$. The Azumaya algebra $p_B^*B$ splits on $\Xc$ if and only if $Y_{B} \to \Xc$ has a section. This is discussed in \cite{MR559531} for $\Xc$ being a scheme. The generalisation to Azumaya algebras on Deligne-Mumford stacks is an application of faithfully flat descent theory. 

Let $B$ be an Azumaya algebra on $\Xc_{L/F}$, which splits when restricted to $\Spec(F) \hookrightarrow \Xc_{L/F}$. We see that we have a $\Gamma$-equivariant section of $Y_{B} \times_{\Xc_{L/F}} \Spec \Oo_L$ over $\Spec L$. By faithfully flat descent the morphism $p_{B}$ is proper and smooth. 
The valuative criterion of properness therefore yields an extension of this section to $\Spec \Oo_L$, and uniqueness implies that this section is $\Gamma$-equivariant. We conclude that we have a section of $Y_{B} \to \Xc_{L/F}$ and thus a splitting of the Azumaya algebra $B$ on $\Xc_{L/F}$. 
\end{proof}

Our next goal is to prove that $\Br(\Xc_{L/F})$ has order $e$. This implies $\Br(\Xc_{L/F}) \simeq \Br(F)[e] = \Zb/e\Zb$. This is the content of the following proposition. Its proof relies on our computation of \'etale cohomology groups of $\Xc_{L/F}$ given earlier.

\begin{proposition}\label{prop:BrX}
Pullback along $i_F$ induces an isomorphism $i_F^*\colon \Br(\Xc_F) \simeq \Br(F)[e]$. In particular, since $\Br(F) \simeq \Qb/\Zb$, we have that $\Br(\Xc_{L/F})$ has order $e$.
\end{proposition}

\begin{proof}
Let $[e]\colon \G_m \to \G_m$ be the $e$-th power map. The Kummer sequence 
$$1 \to \mu_e \to \G_m \xrightarrow{[e]} \G_m \to 1$$
yields a long exact sequence of \'etale cohomology groups
$$\cdots \to H^1_{\et}(\Xc_{L/F},\G_m) \xrightarrow{[e]} H^1_{\et}(\Xc_{L/F},\G_m) \to H^2_{\et}(\Xc_{L/F},\mu_e) \to H^2_{\et}(\Xc_{L/F},\G_m) \xrightarrow{[e]} H^2_{\et}(\Xc_{L/F},\G_m) \to \cdots.$$
According to Lemma \ref{lemma:PicX} we have that $H^1_{\et}(\Xc_{L/F},\G_m)$ is isomorphic to $\Zb/e\Zb$. Thus, the cokernel of $H^1_{\et}(\Xc_{L/F},\G_m) \xrightarrow{[e]} H^1_{\et}(\Xc_{L/F},\G_m)$ is isomorphic to $\Zb/e\Zb$. In Lemma \ref{lemma:localTransgression} we established a non-canonical isomorphism $H^2_{\et}(\Xc_{L/F},\mu_e) \simeq \Zb/e\Zb \oplus \Zb/e\Zb$. We therefore have the following exact sequence
$$\Zb/e\Zb \xrightarrow{0} \Zb/e\Zb \to \Zb/e\Zb \oplus \Zb/e\Zb \to H^2_{\et}(\Xc_{L/F},\G_m) \xrightarrow{[e]} H^2_{\et}(\Xc_{L/F},\G_m).$$
Thus, we obtain $|H^2_{\et}(\Xc_{L/F},\G_m)[e]| = e$.

Lemma \ref{lemma:gabber_var} below shows that $H^2_{\et}(\Xc_{L/F},\G_m) = \Br(\Xc_{L/F})$, and we have seen in Lemma \ref{lemma:Brinj} that $\Br(\Xc_{L/F}) \hookrightarrow \Br(F)[d]$, and therefore $\Br(\Xc_{L/F})[e] \hookrightarrow \Br(F)[e] = \Zb/e\Zb$. Hence the above shows $$\Br(\Xc_{L/F})[e] \simeq \Zb/e\Zb.$$ 

We have seen in Lemma \ref{lemma:wlog} that we may assume without loss of generality that $L/F$ is totally tamely ramified. We then have $d=[L:F] = e$, and therefore by virtue of \ref{lemma:Brinj} an embedding $\Br(\Xc_{L/F}) \hookrightarrow \Qb/\Zb[e]$. Since the right hand side has order $e$, we see that the order of every element in $\Br(\Xc_{L/F})$ divides $e$. This implies $\Br(\Xc_{L/F})[e] = \Br(\Xc_{L/F})$ and thus finishes the proof.
\end{proof}

\begin{lemma}\label{lemma:gabber_var}
Let $\Gamma$ be a finite abstract group acting on a scheme $U$, which admits an ample invertible sheaf. Then, the quotient stack $[U/\Gamma]$ satisfies
$$\Br([U/\Gamma]) \simeq H^2_{\et}([U/\Gamma],\G_m)^{tors}.$$
\end{lemma}

We explain the proof of the lemma below. It relies on a geometric strategy to decide whether a given cohomology class $\alpha \in H^2_{\et}(\Xc,\G_m)$ on an algebraic stack $\Xc$ lies in the image of $\Br(\Xc) \hookrightarrow H^2_{\et}(\Xc,\G_m)$.

\begin{definition}
Let $(U_i \to \Xc)_{i\in I}$ be a smooth cover of $\Xc$ by schemes and let $(\alpha_{ij}) \in Z^2_{\et}(\Xc,\G_m)$ be a \v Cech cocycle for $\Xc$, representing an element $\alpha \in H^2_{\et}(\Xc,\G_m)$.  An $\alpha$-twisted coherent sheaf on $\Xc$ is given by coherent sheaves $\mathcal{F}_i \in \Coh(U_i)$ and isomorphisms
$$\phi_{ij}\colon \mathcal{F}_j|_{U_{ij}} \to \mathcal{F}_j|_{U_{ij}},$$ such that 
on $U_{ijk}$ the twisted cocycle identity $\phi_{ij}\circ \phi_{jk} = \alpha_{ijk}\cdot{}\phi_{ik}$ holds.
\end{definition}

Let $\alpha$ be a $\G_m$-gerbe on $\Xc$. Then, $\alpha$ comes from an Azumaya algebra $B$  if and only if there exists a locally free $\alpha$-twisted sheaf of finite rank: indeed, an Azumaya algebra $B$ is a $B$-module and thus gives rise to a twisted sheaf, vice versa, if $\F$ is an $\alpha$-twisted sheaf we have an Azumaya algebra $B=\Eend(\F)$, which represents $\alpha$. 

\begin{proof}[Proof of Lemma \ref{lemma:gabber_var}]
We recall that according to a theorem of Gabber \cite{gabber}, the Brauer group $\Br(U)$ of a scheme $U$ which admits an ample invertible sheaf can be identified with the torsion part $H^2_{\text{\'et}}(U,\G_m)^{tors}$ of degree $2$ \'etale cohomology of $\G_m$. This proves the assertion above for $\Gamma$ being the trivial group.

For the proof of the general case we observe that we have a finite and \'etale quotient map $q\colon U \to [U/\Gamma]$. Let $\alpha \in H^2_{\et}([U/\Gamma],\G_m)^{tors}$ be a cohomological Brauer class. The pullback $q^*\alpha \in H^2_{\et}(U,\G_m)^{tors}$ corresponds by virtue of Gabber's theorem to an element of $\Br(U)$. Equivalently, there exists a locally free $q^*\alpha$-twisted sheaf $E$ on $U$. The pushforward $q_*E$ is a locally free $\alpha$-twisted sheaf on $[U/\Gamma]$. This shows that $\alpha \in \Br([X/\Gamma])$.
\end{proof}

In Lemma \ref{lemma:BI} we saw that a splitting $\sigma\colon \Gamma \to I$ of the inclusion $I \hookrightarrow \Gamma$ induces an equivalence of stacks $[\Spec k_F/ I] \cong [\Spec k_L/\Gamma]$.

\begin{proposition} \label{BrLocalTransgression}
For every splitting $\sigma$ of the embedding $I \hookrightarrow \Gamma$, pullback along the map $j_{\sigma}\colon B_{k_F}I \to \Xc_{L/F}$ induces an isomorphism 
\begin{equation}\label{eqn:isomu}
j_{\sigma}^*\colon\Br(\Xc_{L/F})[e] \simeq \Br(B_{k_F}I)[e].
\end{equation}
\end{proposition}
\begin{proof}
The Kummer sequence
$$1 \to \mu_e \to \G_m \xrightarrow{[e]} \G_m \to 1$$
yields an exact sequence
$$H^2_{\et}(U,\mu_e) \to \ker(H^2_{\et}(U,\G_m) \xrightarrow{[e]} H^2_{\et}(U,\G_m)) = \Br(U)[e] \to 0,$$
where $U$ denotes a Deligne-Mumford stack on which $e$ is invertible. In particular, we see that $H^2_{\et}(U,\mu_e) \twoheadrightarrow \Br(U)[e]$ is surjective. Applying this to $U=[\Spec k_L/\Gamma]$ and $U=\Xc_{L/F}$ we obtain the following commutative diagram
\[
\xymatrix{
H^2_{\et}(\Xc_{L/F},\mu_e) \ar[r]^-{\simeq} \ar@{->>}[d] & H^2_{\et}([\Spec k_L/\Gamma],\mu_e) \ar@{->>}[d] \\
\Br(\Xc_{L/F})[e] \ar[r]^-{j_{\sigma}^*} & \Br([\Spec k_F/I])[e]. 
}
\]
We deduce that $j_{\sigma}^*\colon \Br(\Xc_{L/F})[e]  \twoheadrightarrow \Br([\Spec k_F/I])[e]$ is a surjection. Proposition \ref{prop:BrX} guarantees that $\Br(\Xc_L/F)[e]$ is a finite group of order $e$. We therefore have that $j_{\sigma}^*\colon \Br(\Xc_{L/F})[e]  \twoheadrightarrow \Br([\Spec k_F/I])[e]$ is an isomorphism.
\end{proof}

\begin{lemma}\label{encore}
The isomorphism \eqref{eqn:isomu} is independent of the choice of $\sigma$.
\end{lemma}

\begin{proof} 
Let $\sigma_1,\sigma_2\colon \Gamma \to I$ be two splittings, and let $\alpha_i\colon B_{k_F}I \xrightarrow{\simeq} [k_L/\Gamma]$ be the corresponding equivalences. Then, we obtain an autoequivalence $\alpha_2^{-1}\alpha_1\colon B_{k_F}I \to B_{k_F}I$. It remains to show that the diagram
\[
\xymatrix{
\Br(B_{k_F}I) \ar[rd] \ar[r]^{(\alpha_2^{-1}\alpha_1)^*} & \Br(B_{k_F}I) \ar[d] \\
& H^1(k_F,\mu_e) \oplus H^0(k_F,\Zb/e\Zb)
}
\]
commutes.

To see this we argue as follows: maps from $[\Spec k_F/I]$ to itself, are classified by $I$-equivariant $I$-torsors on the point $\Spec k_F$. That is, the set of such self-maps up to isomorphism is in bijection with $H^1(k_F,I) \oplus \Hom(I,I)$. Let $[\Spec k_F/I] \to [\Spec k_F/I]$ be a map inducing the identity $\id_I\colon I \to I$. Then the induced map $H^1(k_F,I^{\vee}) \simeq \Br([k_F/I]) \to \Br([k_F/I]) \to H^1(k_F,I^{\vee})$ is the identity. Since $\alpha_2^{-2}\alpha_1$ satisfies this property, this concludes the proof.
\end{proof}

\subsection{Transgression and Hasse's invariant}\label{sub:comp_Hasse}

In the following we assume that $\mu_{e^2}$ is contained in $k$. We denote by $\tau\colon \Br(B_{k_F}I)[e] \to H^1_{\et}(\I B_{k_F}I,\mu_e)$ the transgression morphism of Lemma \ref{lemma:muGm}. Recall that the inertia stack $\I B_{k_F} I$ is equivalent to $I\times B_{k_F} I$ (see Remark \ref{inerst}). Let 
\begin{equation}\label{eqn:IB}\Spec k_F \times I \to \I B_{k_F} I\end{equation} be the map induced from the equivalence above and the natural map $\Spec k_F \to B_{k_F} I$. Furthermore, we have an isomorphism $I \simeq \mu_e$.  
Pullback along \eqref{eqn:IB} induces a map
\begin{equation}
\Tr\colon H^1_{\et}(\I B_{k_F}\mu_e,\mu_e) \to \Hom(\mu_e,H^1_{\et}(k_F,\mu_e)) \simeq \Zb/e\Zb,
\end{equation}
where the isomorphism is given by evaluating an element of $H^1_{\et}(k_F,\mu_e) = \Hom(\Gal(\overline{k}_F/k_F))$ at the Frobenius element $\Fr$, and using $\Hom(\mu_e,\mu_e) \simeq \Zb/e\Zb$.
The composition 
$$\tilde{\tau} = \Tr(\tau\circ j_{\sigma}^*)\colon \Br(\Xc_{L/F}) \to Hom(\mu_e,\mu) = \Zb/e\Zb$$
is independent of $\sigma$ (according to Lemma \ref{encore}). The result below asserts that for elements of $\Br(F)[e]$, the value of the transgression map $\tilde{\tau}$ agrees with the Hasse invariant. A more general result can be found in the authors' recent \cite{GWZ18}. In \emph{loc. cit.} we extend this comparison to arbitrary Brauer classes of order coprime to $p$. The proof given below is similar to the one in \emph{loc. cit.} and is included for the sake of keeping the paper self-contained.

\begin{theorem}\label{thm:comp_Hasse}
Let $(i_F^*)^{-1}\colon \Br(F)[e] \to \Br(\Xc_{L/F})$ be the inverse to the isomorphism of Proposition \ref{prop:BrX}, and $\tilde{\tau}\colon \Br(\Xc_{L/F}) \to \Zb/e\Zb$ the morphism defined above. Then the diagram
\[
\xymatrix{
\Br(F)[e] \ar[r]^{(i_F^*)^{-1}} \ar[rd]_{\inv} & \Br(\Xc_{L/F})[e] \ar[d]^{\tilde{\tau}} \\
& \Qb/\Zb[e]
}
\]
commutes.
\end{theorem}

\begin{proof}
According to Lemma \ref{lemma:wlog} we may assume that $L/F$ is totally ramified. We begin the proof by recalling the definition of the Hasse invariant. The Brauer group of $F$ is isomorphic to $H^2(\Gal(\bar{F}/F),\bar{F}^{\times})$. However, a gerbe on $F$ splits on an unramified cover, which means that it belongs to the image of the cohomology group $H^2(\Gal(F^{un}/F),(F^{un})^{\times})=H^2(\widehat{Z},(F^{un})^{\times})$.

Recall that the Hasse invariant is defined by composition 
$$H^2(\Gal(F^{un}/F) ,(F^{un})^{\times}) \simeq H^2(\Gal(\overline{k}/k),\Zb)) \simeq H^1_{\text{\'et}}(\widehat{\Zb},\Qb/\Zb) \simeq \Qb/\Zb.$$
The first map is induced by $\nu\colon (F^{un})^{\times} \to \Zb$, sending the uniformiser ${\pi}$ to $1$. This map has a section $\Zb \to (F^{un})^{\times}$, given by $1\mapsto \pi$.

For a fixed Brauer class $\alpha \in \Br(F)[e]$, we denote by $\phi=(\phi_{ijk})$ the $\mathbb{Z}$-valued $2$-cocycle on the site of unramified coverings of $F$ (the site equivalent to the small \'etale site of $\Oo_F$) which corresponds to $\alpha$ under the isomorphism of the paragraph above. That is, the gerbe $\alpha$ is represented by $(\pi^{\phi_{ijk}})$.

By assumption, the order $\ord(\alpha)$ of $\alpha$ is a divisor of $e$, and is coprime to $p$. This implies that $(e\cdot{}\phi_{ijk})$ is a coboundary. We denote by $\psi=(\psi_{ij})$ the $1$-cochain, such that $d(\psi) = e \cdot \phi$.

Recall that we assume that $L/F$ is a totally ramified field extension. We denote by $\pi^{1/e}$ a uniformiser for $L$. We define an $L^{\times}$-valued $1$-cochain $(\pi^{\frac{\psi_{ij}}{e}})$, such that

\begin{equation}\label{eqn:3}d(\pi^{\frac{\psi_{ij}}{e}}) = \pi^{\phi_{ij}}.\end{equation}

Next we define a $2$-cochain taking values in $\Zb/e\Zb = \Hhom(\mu_n,\G_m)$. It is given by 
$$c_{ij}\colon \xi \mapsto \xi^{\psi_{ij}}$$
for $\xi \in \mu_e$. We claim that this $1$-cochain is a coycle, that is, satisfies $d(c_{ij}) = 1$. Indeed, $\xi \in \mu_e$ gives rise to a field automorphism $\sigma_{\xi}$ of $L/F$ (and all unramified base changes thereof) sending $\pi^{\frac{1}{e}} \mapsto \xi \cdot{} \pi^{\frac{1}{e}}$. We apply $\sigma_{\xi}$ to \eqref{eqn:3} and obtain
$$d(\xi^{\psi_{ij}} \pi^{\frac{\psi_{ij}}{e}}) = \pi^{\phi_{ij}}.$$
This shows $d(\xi^{\psi_{ij}}) = 1$ as asserted above.  
\begin{claim}
The cohomology class induced by the $1$-cocycle $(c_{ij})$ agrees with $\tilde{\tau}(\alpha) \in \Zb/e\Zb\simeq H^1(k,\Zb/e\Zb)$.
\end{claim}
\begin{proof}
By definition, the $2$-cocycle $(\pi^{\phi_{ijk}})$ represents $\alpha$ on the small site of unramified \'etale schemes over $F$ (equivalently, the small \'etale site of $\Oo_F$). We pull back $\alpha$ to the $\Gamma$-torsor $\Spec L \to \Spec F$  where we have chosen an $e$-th root $\pi$, denoted by $\pi^{\frac{1}{e}}$. The gerbe $\alpha_L$ obtained thereby is also represented by the cocycle $(\pi^{\phi_{ijk}})$, now seen as an $L^{\times}$-valued cocycle. It is a coboundary, as shown by the following computation: 
$$(\pi^{\phi_{ijk}}) = \left((\pi^{\frac{1}{e}})^{e\cdot{}\phi_{ijk}}\right) = \left( (\pi^{\frac{1}{e}})^{d(\psi_{ij})} \right).$$
The $1$-cochain $\left( (\pi^{\frac{1}{e}})^{(\psi_{ij})} \right)$ represents a splitting of the pulled back gerbe $\alpha_L$.
There is an action of $\xi \in \mu_e$ on $L/F$, given by $\pi^{\frac{1}{e}} \mapsto \xi \mapsto \xi \cdot{} \pi^{\frac{1}{e}}$. The splitting of $\alpha_L$, given by $\left( (\pi^{\frac{1}{e}})^{(\psi_{ij})} \right)$, is sent to 
$$\left(\xi \mapsto  (\xi^{\psi_{ij}}\pi^{\frac{1}{e}})^{(\psi_{ij})} \right) = \left( c_{ij}(\xi)\cdot{} \pi^{\psi_{ij}} \right) $$
by the action of $\xi \in \mu_e$. 
Therefore, $(c_{ij})$ is a $2$-cocycle representing $\tilde{\tau}(\alpha)$.
\end{proof}

One can identify $\Zb/e\Zb$ with $\mu_e^{\vee} = \Hom(\mu_e,\G_m)$, and thereby obtains a $1$-cocycle $(\psi_{ij})$ in $\Zb/e\Zb$. The boundary map of the Bockstein sequence
$$\Zb \hookrightarrow \Zb \twoheadrightarrow \Zb/e\Zb$$
sends $(\phi_{ijk})$ to the cocycle $(\psi_{ij})$ in $\Zb/e\Zb$. The commutative diagram (with exact rows)
\[
\xymatrix{
\Zb \ar@{^(->}[r] \ar[d] & \Zb \ar@{->>}[r] \ar[d] & \Zb/e\Zb  \ar[d] \\
\Zb \ar@{^(->}[r] & \Qb \ar@{->>}[r] & \Qb/\Zb 
}
\]
together with the definition of the Hasse invariant as image of $(\psi_{ijk}) \in H^2(k,\Zb) \simeq H^1(k,\Qb/\Zb)$, shows that $\tilde{\tau}(\alpha) = \inv(\alpha)$. This concludes the proof.
%
\end{proof}

%

\subsection{Hasse invariants for gerbes on Deligne-Mumford stacks}\label{sub:application}

We now turn to a description of the typical set-up to which we will apply the results of this section. For the applications we have in mind, $Y$ will be a Hitchin fibre and $\Mm$ the moduli stack of stable Higgs bundles for the groups $\SL_n$ or $\PGL_n$ of coprime degree (see Section \ref{tms}).
\begin{situation}\label{situationY}
\begin{enumerate}
\item\label{M} Let $\Mm = [U/\Gamma]$ be an admissible finite abelian quotient stack over $\Oo_F$. We denote by $M$ the coarse moduli space of $\Mm$, it is an algebraic $\Oo_F$-space. 
\item\label{Y} We suppose that we have a proper algebraic $\Oo_F$-scheme $Y/\Oo_F$ with a map $Y \to M$, such that there is a factorisation 
\[
\xymatrix
{
Y \times_{\Spec \Oo_F} F \ar[r] \ar[d] & \Mm \ar[d] \\
Y \ar[r] & M.
}
\]
\item\label{neutral} Furthermore, we fix a $\G_m$-gerbe $\alpha$ on $\Mm$ of order $r$ coprime to $p$, such that $F$ contains all $r\cdot{}|\Gamma|$-th roots of unity, and denote by $\inv_{\alpha}\colon M(\Oo_F)^{\natural} \to \Qb/\Zb$ the map sending $y\colon \Spec F \to \Mm$ to $\inv(y^*\alpha)$. Recall from Subsection \ref{sectorbm} that $M(\Oo_F)^{\natural}$ consists of certain $\Oo_F$-points of $M$, for which the induced $F$-rational point lifts to $\Mm$. 
\end{enumerate}
\end{situation}

\begin{definition}\label{defi:taualpha}
Let $\tau\colon \Br(\Mm_{k_F})[r] \to H^1_{\et}(\I \Mm_{k_F},\mu)$ be the transgression morphism of Lemma \ref{lemma:muGm}. For $\alpha \in \Br(\Mm_{k_F})[r]$ we denote by 
$$\tau'_{\alpha}\colon \I\Mm(k_F)_\iso \to H^1(k_F,\mu_r) = \mu_r.$$
the induced function sending $y \in \I\Mm(k_F)_\iso$ to the restriction of the $\mu_r$-torsor $\tau(\alpha)$ on $\I\Mm_{k_F}$ along $y\colon \Spec k_F \to \I\Mm_{k_F}$. We let $\tau_{\alpha}(y) \in \Zb/r\Zb$ be the unique element, such that $\zeta^{\tau_{\alpha}} = \tau'_{\alpha}$, where we recall that $\zeta \in \mu_r$ denotes the fixed primitive root of unity of order $r$.
\end{definition}

Under the assumptions of Situation \ref{situationY} we will show in Corollary \ref{cor:YF} that the function $\inv_{\alpha}$ on $Y(F)$ factors through $e_{\Mm}\colon Y(F) \to \I\Mm(k_F)_\iso$, and the resulting function on $\I\Mm(k_F)$ agrees with the function $\tau_{\alpha}$. We will deduce this from the following more general assertion.

\begin{proposition}\label{prop:YIM}
Assume that \eqref{M} and \eqref{neutral} of Situation \ref{situationY} are satisfied. Then, there is a commutative diagram of sets
\[
\xymatrix{
M(\Oo_F)^{\natural} \ar[rd]^{\inv_{\alpha}} \ar[d]_{e_{\Mm}} & \\
\I\Mm(k_F)_\iso \ar[r]^-{\tau_{\alpha}} & \Qb/\Zb.
}
\]
\end{proposition}

\begin{proof}
In Construction \ref{SpecMap} it is shown that the restriction of $y \in M(\Oo_F)^{\natural}$ to $\Spec F$ extends to a map
$$y'\colon [\Spec \Oo_L/\Gal(L/F)] \to \Mm$$
for an appropriately chosen finite tame field extension $L/F$ with abelian Galois group. By definition, the quotient stack $[\Spec \Oo_L/\Gal(L/F)]$ is $\Xc_{L/F}$. By virtue of Theorem \ref{thm:comp_Hasse} we have
$$\inv(y^*\alpha) = \tilde{\tau}(y'^*\alpha).$$
By definition, the right hand side agrees with $\tau_{\alpha}(y)$.
\end{proof}

\begin{corollary}\label{cor:YF}
Assume that the assumptions of Situation \ref{situationY} are satisfied. Then, there is a commutative diagram of sets
\[
\xymatrix{
Y(F) \ar[rd]^{\inv_{\alpha}} \ar[d]_{e_{\Mm}} & \\
\I\Mm(k_F)_\iso \ar[r]^{\tau_{\alpha}} & \Qb/\Zb.
}
\]
\end{corollary}

\begin{proof}
Let $y \in Y(F)$. According to assumption \eqref{Y} of Situation \ref{situationY} the scheme $Y$ is proper over $\Oo_F$. Therefore, $y$ extends uniquely to an $\Oo_F$-point $\tilde{y} \in Y(\Oo_F)$. Furthermore, the restriction of $\tilde{y}|_{\Spec F}$ factors through the quotient stack $\Mm$, again by assumption \eqref{Y}. This shows that $y \in Y(F)$ gives rise to an element of $\tilde{y} \in M(\Oo_F)^{\natural}$. 
Thus, the value of the specialisation map $e_{\Mm}\colon M(\Oo_F)^{\natural} \to \I\Mm(k_F)_\iso$ at $\tilde{y}$ is well-defined. We can now apply Proposition \ref{prop:YIM} to conclude the proof.
\end{proof}

\subsection{$p$-adic integrals of the Hasse invariant} Let $r\geq 1$ be an integer prime to $p$ such that $F$ contains all $r$-th roots of unity. Recall that we fix a generator $\zeta \in \mu_r(F)$. This choice gives rise to an embedding $\mu_r(F) \hookrightarrow \Cb^{\times}$. We also fix an isomorphism $\overline{\Qb}_{\ell} \simeq \Cb$. By combining these choices we obtain an embedding $\mu_r(F) \hookrightarrow \overline{\Qb}_{\ell}^{\times}$.
\begin{definition} \label{DefL}
  Let $\Mm$ be an admissible finite abelian quotient stack over $\Oo_F$ and $\alpha$ a $\mu_r$-gerbe on $\Mm$. By Subsection \ref{GerbesSection} the gerbe $\alpha$ induces a $\mu_r$-torsor $P_\alpha$ on $\I\Mm$. We denote by $L_{\alpha}$ the $\ell$-adic local system on $\I(\Mm \times_{\Oo_F} k)$ induced from $P_\alpha$ via the embedding $\iota$.
\end{definition}

\begin{construction}\label{fi}
Let $\Mm$ be an admissible finite abelian quotient stack over $\Oo_F$ and $\alpha$ a $\mu_r$-gerbe on $\Mm$. We obtain a function $$\I\Mm(k)_\iso \to \mu_r(F) \subset \Cb,\; x \mapsto \Tr_{\Fr_x}(L_{\alpha}|_x).$$ By composing this with the specialisation map $e_{\Mm}:M(\Oo_F)^{\natural} \to \I\Mm(k)_\iso$ we obtain a function
\begin{equation}\label{deffa}
 f_{\alpha}\colon M(\Oo_F)^\natural \to \Cb,\; x \mapsto \Tr_{\Fr_{e(x)}}(L_{\alpha}|_{e(x)}).
\end{equation}
Note that by definition  $f_\alpha$ agrees with $\tau_\alpha \circ e_{\Mm}$ under the exponential.
\end{construction}

On the other hand there is a natural function on $\Mm(\Oo_F)^\natural$ associated with $\alpha$, namely

\begin{align*} \Mm(\Oo_F)^\natural &\rightarrow  \Cb \\
															x &\mapsto \exp(2\pi i \cdot \inv(x^*\alpha)),
\end{align*}
where $\inv:\Br(F) \rightarrow \Qb/\Zb$ denotes the Hasse invariant. It follows from Proposition \ref{prop:YIM}, that this functions agrees with $f_\alpha$.

\begin{corollary} \label{StringyF}
 Let $\Mm$ be an admissible finite abelian quotient stack over $\Oo_F$ and $\alpha$ a $\mu_r$-gerbe on $\Mm$. The associated function $f_\alpha\colon M(\Oo_F)^\natural \to \Cb$ satisfies
 \begin{equation*}
   \frac{\#_{\st}^{\alpha}(\Mm_{k_F})}{q^{\dim \Mm}}=\int_{M(\Oo_F)^\natural} \bar{f}_{\alpha}\mu_\text{orb}= \int_{M(\Oo_F)^\natural} \exp(-2\pi i \cdot {\inv}_{\alpha})\mu_\text{orb},
 \end{equation*}
where $\#_{\st}^{\alpha}(\Mm_{k_F})$ is defined in \ref{defi:l-torsor} and $\bar{f}_\alpha$ denotes the complex conjugate of $f_\alpha$.
\end{corollary}
\begin{proof}
We choose a presentation $\Mm = [Y/\Gamma]$ as in Definition \ref{AdmStack}. By Theorem \ref{thm:appendix} we have

$$\int_{M(\Oo_F)^\natural} f_{\alpha}\mu_\text{orb} =  q^{-\dim \Mm} \sum_{\gamma \in \Gamma} \sum_{\Z \in \pi_0([Y^\gamma/\Gamma])} q^{F(\gamma^{-1},\Z)} \sum_{x \in \Z(k_F)_{\text{iso}}} \frac{\Tr_{\Fr_x}(L_\alpha|_x)}{|\Aut(x)|},$$
where we also used the relation $-w(\gamma,\Z) = F(\gamma^{-1},\Z)-\dim \Mm$. Now under the identification $[Y^\gamma/\Gamma] = [Y^{\gamma^{-1}}/\Gamma]$ we have $L_{\alpha}|_{[Y^\gamma/\Gamma]} = L^{-1}_{\alpha}|_{[Y^{\gamma^{-1}}/\Gamma]}$, see \eqref{autgerb}, and thus 
\[\sum_{\Z \in \pi_0([Y^\gamma/\Gamma])} q^{F(\gamma^{-1},\Z)} \sum_{x \in \Z(k_F)_\iso} \frac{\Tr_{\Fr_x}(L_\alpha|_x)}{|\Aut(x)|}= \sum_{\Z \in \pi_0([Y^{\gamma^{-1}}/\Gamma])} q^{F(\gamma^{-1},\Z)} \sum_{x \in \Z(k_F)_\iso} \frac{\overline{\Tr_{\Fr_x}(L_\alpha|_x)}}{|\Aut(x)|},  \]
which implies the first equality. The second equality follows directly from Proposition \ref{prop:YIM}.
\end{proof}

\section{Mirror Symmetry}\label{sec:mirror}


In this section we formulate various comparison theorems for dual abstract Hitchin systems. The definition of dual abstract Hitchin systems is motivated by the theory of $G$-Higgs bundles, notably the $\SL_n$ and $\PGL_n$ cases, which appear in Hausel--Thaddeus's conjecture. 

\subsection{Relative splittings of gerbes}


Let $U \to V$ be a proper morphism of algebraic spaces with geometrically connected fibres and $r\geq 1$ a positive integer which is invertible on $V$.
According to \cite{MR0260746} there exists an algebraic stack of line bundles $\widetilde{\mathsf{Pic}}(U/V)$. Its $\G_m$-rigidification (see \cite[Section 5]{abramovich2003twisted}) will be denoted by $\mathsf{Pic}(U/V)$. We denote by $\mathsf{Pic}^{\tau}(U/V)$ the open substack of line bundles on $U$ of torsion degree, that is, those line bundles inducing a torsion element in the geometric fibres of $\pi_0(\mathsf{Pic}(U/V))$. 

\begin{definition}
Let $A$ be a smooth group scheme over $V$. We call an $A$-gerbe $\alpha \in H^2_{\text{\'et}}(U,A)$ \emph{$V$-arithmetic} if there exists an \'etale covering family $\{V_i \to V\}_{i \in I}$ such that $\alpha$ splits when pulled back to $U \times_V V_i$ for all $i \in I$. 
\end{definition}

These definitions are already interesting to us when $V$ is the spectrum of a field. 
Let $K$ be a field and $X$ a projective and geometrically connected $K$-scheme. A $\G_m$-gerbe $\beta$ on $X$ is $K$-arithmetic, if and only if $\beta_{K^{\mathrm{sep}}} \in \Br(X_{K^{\mathrm{sep}}})$ is trivial. We denote the corresponding subgroup of $\Br(X)$ by $\Br(X)_K$. There is a short exact sequence 
\begin{equation}\label{BrK}
0 \to \Br(K) \to \Br(X)_K \to H^1(K,\Pic(X_{K^{\mathrm{sep}}})) \to 0,
\end{equation}
which expresses the difference between $\Br(K)$ and $\Br(X)_K$.

If $\alpha$ is a $\G_m$-gerbe on $U$, then one defines the $V$-stack of relative splittings $\widetilde{\Split}(U/V,\alpha)$ as follows: for every affine scheme $W$ with a morphism $f\colon W \to V$ one associates the groupoid of splittings of $f^*\alpha$ on $U \times_V W$. Since we can tensor splittings of $f^*\alpha$ with line bundles on $U \times_V W$, we see that $\widetilde{\Pic}(U/V)$ acts on $\widetilde{\Split}(U/V,\alpha)$. Furthermore, any pair of $f^*\alpha$-splittings differs by a line bundle on $U \times_V W$ unique up to isomorphism. In other words, $\widetilde{\Split}(U/V,\alpha)$ is a $\widetilde{\Pic}(U/V)$-quasitorsor.

\begin{lemma}
Let $\alpha$ be $V$-arithmetic, then $\widetilde{\Split}(U/V,\alpha)$ is a $\widetilde{\Pic}(U/V)$-torsor.
\end{lemma}

\begin{proof}
We already observed that $\widetilde{\Split}(U/V,\alpha)$ is a $\widetilde{\Pic}(U/V)$-quasitorsor. The $V$-arithmeticity assumption on $\alpha$ implies the existence of an \'etale covering $\{V_i \to V\}_{i \in I}$, such that $\alpha$ splits when pulled back to $U \times_V V_i$. In particular, $\widetilde{\Split}(U/V,\alpha)$ has a $V_i$-rational point. This implies the torsor property.
\end{proof}

Since the stack $\widetilde{\Split}(U/V,\alpha)$ is a torsor under $\widetilde{\Pic}(U/V)$, its fibres contain infinitely many connected components. For instance, for $U \to V$ a relative family of smooth proper curves, this sheaf of sets of connected components would be a $\Zb$-torsor. The following two definitions will be used to single out a much smaller subset $\widetilde{\Split}(U/V,\alpha)$.

\begin{definition} For $\alpha \in H^2_{\text{\'et}}(U,\mu_r)$ a $\mu_r$-gerbe on $U$ we denote by $\widetilde{\mathsf{Split}}_{\mu_r}(U/V,\alpha)$ the $V$-stack sending a test scheme $S \to V$ to the groupoid of splittings of the pullback of $\alpha$ to $U \times_V S$. The $\mu_r$-rigidification of this stack in the sense of \cite[Section 5]{abramovich2003twisted} will be denoted by $\mathsf{Split}_{\mu_r}(U/V,\alpha)$.
\end{definition}

We now make use of the embedding of smooth group schemes $\mu_r \hookrightarrow \G_m$. It gives rise to a natural map from the stack $\widetilde{\mathsf{Tors}}_{\mu_r}(U/V)$ of $\mu_r$-torsors, to $\widetilde{\Pic}(U/V)$. The Kummer sequence implies that after rigidification of the stacks, one obtains an isomorphism
$\mathsf{Tors}_{\mu_r} \simeq \widetilde{\Pic}(U/V)[r].$

\begin{definition}\label{principalcomponent}
For a $V$-arithmetic gerbe $\alpha \in H^2_{\text{\'et}}(U,\mu_r)$ we denote the induced $\mathsf{Pic}^{\tau}(U/V)$-torsor $${\mathsf{Split}_{\mu_r}}(U/V,\alpha) \times^{\mathsf{Pic}(U/V)[r]} \mathsf{Pic}^{\tau}(U/V)$$ by $\mathsf{Split}'(U/V,\alpha)$. This torsor will be referred to as the \emph{principal component} of the space of relative splittings of $\alpha$.
\end{definition}

%
%
%

It is clear that the existence of a splitting of $\alpha$ on $X/K$ as above, implies that $\mathsf{Split}'(X)$ has an $K$-rational point. The converse is not true, as shown by the case where $\alpha$ is the pullback of a non-split gerbe $\beta$ on $\Spec K$ and we assume $X$ to have a $K$-rational point $x$. In this situation, $\alpha$ cannot be split, as $x^*\alpha = \beta$. On the other hand, $\mathsf{Split}'(X,\alpha) \simeq \Pic^{\tau}(X)$ has a $K$-rational point. This is essentially the only thing that can happen as the next lemma proves.

\begin{lemma}\label{lemma:splitK}
Let $K$ be a field and $X$ a geometrically connected proper $K$-scheme and $\alpha$ on $X$ be an $K$-arithmetic $\mu_r$-gerbe. We denote by $\beta$ the induced $\G_m$-gerbe on $X$. If $\mathsf{\Split}'(X,\alpha)$ has a $K$-rational point, then $\beta$ agrees with the pullback of a $\G_m$-gerbe on $\Spec K$.
\end{lemma}

\begin{proof}
It follows from the assumption that $\beta$ is sent to zero in $H^1(K,\Pic^{\tau}(X_{K^{\mathrm{sep}}}))$, and hence also induces zero in $H^1(K,\Pic(X_{K^{\mathrm{sep}}}))$. We infer from exactness of \eqref{BrK} that $\beta$ lies in the image of the map $\Br(K) \to \Br(X_K)$ and thus, $\beta$ is isomorphic to the pullback of an element of $\Br(K)$.
\end{proof}

Our main motivation to work with arithmetic gerbes comes from geometry. The principal component of splittings provides an algebraic analogue of the manifold of unitary splittings of a torsion gerbe. We record this observation in the following remark. It will not be needed in any of the arguments, but it provides the key to translate our algebraic viewpoints on gerbes into the analytic language of flat unitary gerbes used in Section 3 of \cite{MR1990670}.

Recall that we denote by $\Pic^{\tau}(U/V)$ the torsion component of the relative Picard variety (if it exists), that is, $\Pic^{\tau}(U/V)$ is the preimage of the maximal torsion subgroup of $\pi_0(\Pic(U/V))$.

\begin{rmk}\label{rmk:unitary}
Let $U \to V$ be a smooth projective morphism of complex varieties, and $\alpha$ a $V$-arithmetic $\mu_r$-gerbe on $U$. If $\Pic^{\tau}(U/V) = \Pic^0(U/V)$, then the complex space $\mathsf{Split}'(U/V,\alpha)$ is isomorphic to the manifold of flat unitary splittings of the unitary gerbe induced by $\alpha$.
\end{rmk}

\begin{lemma}\label{lemmaTate} Assume we are in Situation \ref{situationY} and that the $F$-fibre $Y_F$  a torsor under an abelian variety $A$. Let $\alpha$ be an arithmetic $\mu_r$-gerbe. Consider the space $T=\Split'(Y/ F,\alpha)$ whose generic fibre is canonically an $A^\vee$-torsor over $\Spec(F)$. If $Y_F(F)\not= \emptyset$ then for any trivialisation $h\colon Y_F\cong A$ there exists a root of unity $\xi$, such that the function 
\begin{equation*}
  f_\alpha\circ h^{-1}\colon A(F) \to M(\Oo_F)^\natural \to \Cb
\end{equation*}
is equal to 
\[a \mapsto \xi\cdot \exp(2\pi i(a,[T])),\]
 where $f_\alpha$ is as in \eqref{deffa} and $(a,[T])$ is the Tate duality pairing.
\end{lemma}

\begin{proof}
Let us assume first that $Y_F$ is the trivial $A$-torsor $A$ together with the  trivialisation given by the identity map. The claim then follows from Corollary \ref{cor:YF} and the definition of the Tate Duality pairing for an abelian variety $A/F$: given a rational point $x \in A(F)$ and the torsor $T \in H^1(F,A^\vee)$ we use the isomorphism $H^1(F,A^\vee) \simeq \Ext^2(A,\G_m)$ to produce a gerbe $\alpha_T \in \Br(A)$. By virtue of Remark \ref{TDCons} we have that $(x,T)$ equals the Hasse invariant of the gerbe $x^*\alpha_T \in \Br(F)$ obtained by pulling back $\alpha_T$ along $x \colon \Spec F \to A$. We have shown in Corollary \ref{cor:YF} that $f_\alpha$ is equal to $\inv(x^*\alpha)$ under the embedding $\Qb/\Zb \to \Cb$ given by the exponential, thus we obtain $f_\alpha = \exp(2\pi i(-,T))$.

In the general case, where we choose an arbitrary trivialisation of $Y$, we apply Corollary \ref{cor:group_torsor}. It follows from the statement of \emph{loc. cit} that the choice of trivialisation leads to $f_\alpha = \xi \cdot{}\exp(2\pi i(-,T))$, and the scalar $\xi$ is a root of unity.
\end{proof}

\subsection{The setting for mirror symmetry of abstract Hitchin systems}

In this subsection we denote by $R$ a Noetherian commutative ring. The reader is invited to think of it as a subring of the field of complex numbers $\mathbb{C}$ or the valuation ring $\Oo_F$ of a non-archimedean local field $F$.  
\begin{definition}\label{presituation}
Let $\Aa/R$ be a smooth $R$-variety, $\Mm$ a smooth admissible finite abelian quotient stack over $R$ in the sense of Definition \ref{AdmStack} together with a morphism $\pi\colon \Mm \to \Aa$ and $\Pc$ an $\Aa$-group scheme acting on $\Mm$ relative to $\Aa$. We say that $(\Mm,\Pc,\Aa)$ is an \emph{abstract Hitchin system} over $R$, if there exists an open dense subscheme $\Aa^{\lozenge} \subset \Aa$ with respect to which the following conditions are satisfied:
\begin{enumerate}
\item[(a)] We denote the coarse moduli space of $\Mm$ by $M$. We assume that the map $M \to \Aa$ is proper. 

\item[(b)] The base change $\Pc^{\lozenge} = \Pc \times_{\Aa} \Aa^{\lozenge}$ is an abelian $\Aa^{\lozenge}$-scheme. 

\item[(c)] There exists an open dense subset $\Mm' \subset \Mm$, which is a $\Pc$-torsor relative to $\Aa$. Furthermore we assume that $\Mm^{\lozenge} \defeq \Mm \times_{\Aa} \Aa^{\lozenge}$ is contained in  $\Mm'$, and that $\codim (\Mm \setminus \Mm') \geq 2$.
\end{enumerate}
\end{definition}

The definition above is directly modelled on the properties of the $G$-Hitchin system studied in \cite[Section 4]{MR2653248}.

Condition (c) above implies that the stack $\Mm$ is generically a scheme, since a torsor over $\Pc/\Aa$ is at least an algebraic space, and algebraic spaces of finite type are generically schematic (\cite[Tag 06NH]{stacks-project}). The generic fibre of $\Pc/\Aa$ is connected (as is implied by (b)), but special fibres may have several connected components.

\begin{definition}\label{situation} A \emph{dual pair of abstract Hitchin systems} over $R$ consists of two abstract Hitchin systems $(\Mm_i,\Pc_i,\Aa)$ over $R$ for $i=1,2$ together with $\Aa$-arithmetic $\mu_r$-gerbes $\alpha_i$ on $\Mm_i$ for $i=1,2$ for some integer $r$ such that there exists an open dense subset $\Aa^{\lozenge} \subset \Aa$ satisfying the conditions of Definition \ref{presituation} for $i=1,2$ with respect to which the following conditions hold: 
\begin{enumerate}
\item[(a)] 
We require there to be an \'etale  isogeny $\phi\colon \Pc_1 \to \Pc_2$ of group $\Aa$-schemes. In positive or mixed characteristic we assume that over $\Aa^{\lozenge}$ the orders of the geometric fibres of the group scheme $\ker \phi$ are invertible in $R$.
\item[(b)] Over $\Aa^{\lozenge}$ there exists an isomorphism $\psi\colon (\Pc_1^{\lozenge})^{\vee} \xrightarrow{\simeq} \Pc_2^{\lozenge}$, with respect to which the isogeny $\phi$ is self-dual. That is, the diagram
	\[
	\xymatrix{
	(\Pc_2^{\lozenge})^{\vee} \ar[r]^{\phi^{\vee}} \ar[d]_{\psi^{\vee}} & (\Pc_1^{\lozenge})^{\vee} \ar[d]^{\psi} \\
	\Pc_1^{\lozenge} \ar[r]^{\phi} & \Pc_2^{\lozenge}
	}
	\]
commutes.

\item[(c)] For $i=1,2$ we denote by $i'$ the unique element in the set $\{1,2\}\setminus \{i\}$. By definition we have that ${\mathsf{Split}'(\Mm_i^{\lozenge}/\Aa^\lozenge,\alpha_i}) $ is an $(\Pc_{i}^{\lozenge})^{\vee}$-torsor. The isomorphism of (b) defines a $\Pc_{i'}^{\lozenge}$-torsor structure on the same space. We stipulate that for $i=1,2$ we have isomorphisms of $\Pc^{\lozenge}_i$-torsors
$${\mathsf{Split}'(\Mm_i^{\lozenge} / \Aa^{\lozenge},\alpha_i)} \simeq \Mm^{\lozenge}_{i'}.$$ 
\item[(d)]  For every local field $F$, every homomorphism $R \to \Oo_F$ and every element $a \in \Aa(\Oo_F) \cap \Aa^{\lozenge}(F)$ with image $a_F \in \Aa(F)$, both fibres $\pi_1^{-1}(a_F)$ and $\pi_2^{-1}(a_F)$ have an $F$-rational point if and only if both gerbes $\alpha_{1}|_{\pi_{1}^{-1}(a_F)}$ and $\alpha_{2}|_{\pi_{2}^{-1}(a_F)}$ split.
\item[(e)] The integer $r$ is invertible in $R$ and $R$ contains all $r$-th roots of unity.
\end{enumerate}
\end{definition}

Note that condition (c) implies the ``if'' direction of condition (d), but not the ``only if'' direction by Lemma \ref{lemma:splitK}.

\begin{rmk}
  For a ring homomorphism $R \to R'$ by base change change every dual pair of abstract Hitchin systems over $R$ induces such a pair over $R'$.
\end{rmk}
We can now state our main result, an abstract version of the Topological Mirror Symmetry Conjecture by Hausel--Thaddeus.

\begin{theorem}[Topological Mirror Symmetry]\label{Mirror}
Let $R$ be a subalgebra of $\Cb$ of finite type over $\Zb$. Let $(\Mm_i,\Pc_i,\Aa,\alpha_i )$ be a pair of dual abstract Hitchin systems over $R$. Then we have an identity of stringy $E$-polynomials
$$E_{\mathsf{st}}(\Mm_1\times_R \Cb,\alpha_1) = E_{\mathsf{st}}(\Mm_2 \times_R \Cb,\alpha_2).$$
\end{theorem}

As stated in the introduction we will deduce this result in complex geometry from an analogue over non-archimedean local fields. 

For a dual pair of abstract Hitchin system we denote by $\pi_i^I \colon \I\Mm_i \to \Aa$ the inertia stacks of $\Mm_i$ together with the induced morphisms to the base $\Aa$. 

\begin{theorem}[Arithmetic Mirror Symmetry]\label{p-Mirror}
 Let $k$ be a finite field. Let $(\Mm_i,\Pc_i,\Aa,\alpha_i )$ a pair of dual abstract Hitchin systems over $k$. Then
$$\#_{\st}^{\alpha_1}(\Mm_1) = \#_{\st}^{\alpha_2}(\Mm_2).$$
In fact the identity holds fibrewise i.e. for every $a \in \Aa(k)$ we have
$$\mathsf{Tr}(\Fr,(R\pi^I_{1,*}N_{\alpha_1}(F_1))_a) = \mathsf{Tr}(\Fr,(R\pi^I_{2,*}N_{\alpha_2}(F_2))_a),$$
where $N_{\alpha_i}$ denotes the $\ell$-adic local system on $\I\Mm_i$ induced by the $\mu_r$-gerbe $\alpha_i$ and some embedding $\mu_r(k)\into \bar \Qb_\ell^\times$ as in Definition \ref{DefL}, $F_i\colon \I\Mm_i \to \mathbb{Q}$ denotes the locally constant functions given by the fermionic shift (see \ref{fermionicshift}), and $N_{\alpha_i}(F_i)$ indicates Tate twist by the fermionic shift.
\end{theorem}

\begin{proof}[Reduction of Theorem \ref{Mirror} to \ref{p-Mirror}]
This is an application of Theorem \ref{thm:stringy-katz}.
\end{proof}

\subsection{Proof of Arithmetic Mirror Symmetry}\label{arithmeticduality}

In this subsection we prove Theorem \ref{p-Mirror} by means of $p$-adic integration. Let $F$ be the local field $k((t))$. By pulling back the pair of dual abstract Hitchin systems $(\Mm_i,\Pc_i,\Aa,\alpha_i)_{i\in \{1,2\}}$ along $\Spec(\Oo_F) \rightarrow \Spec(k)$ one obtains a pair of dual abstract Hitchin systems over $\Oo_F$, which we denote by the same letters.

We start by describing the orbifold measures on $\Mm_i(\Oo_F)^\natural$ in terms of volume forms. As it is enough to prove the fibrewise assertion of Theorem \ref{p-Mirror} we fix an $a\in \Aa(k)$. By replacing $\Aa$ with a neighbourhood of $a$ if necessary we may assume that $\Omega^{\text{top}}_{\Aa/\Oo_F}$ is trivial and we fix a global volume form $\omega_{\Aa}$ on $\Aa$. 

We also fix a global, translation-invariant, trivialising section $\tilde{\omega}_2$ of $\Omega^{\text{top}}_{\Pc_2/\Aa}$, which exists since $\Pc_2 \rightarrow \Aa$ is an $\Aa$-group scheme. As the kernel of the isogeny $\phi$ form Definition \ref{situation} is prime to the characteristic of $k$, the pull-back $\tilde{\omega}_1 = \phi^*\tilde{\omega}_2$ is a global, translation-invariant, trivialising section of $\Omega^{\text{top}}_{\Pc_1/\Aa}$. Through the isomorphism
\[ \Omega^{\text{top}}_{\Pc_i/\Oo_F} \cong \pi_i^*\Omega^{\text{top}}_{\Aa/\Oo_F} \otimes \Omega^{\text{top}}_{\Pc_i/\Aa},\]
we thus obtain a global volume form
\[ \omega_i = \pi_i^*\omega_{\Aa} \wedge \tilde{\omega}_i \]
on $\Pc_i$. A similar definition will be given for $\Mm'_i$ below, but at first we need the following lemma.



\begin{lemma}\label{relativeforms} 
Let $V$ be a scheme, and $A \xrightarrow p V$ a smooth group $V$-scheme. We denote its zero section by $s \colon V \to A$. 
For every $A$-torsor $T \xrightarrow q V$ there exists a canonical isomorphism of sheaves
\[q^*s^*\Omega^m_{A/V} \to \Omega^m_{T/V}.\]
\end{lemma}
\begin{proof}
  For $(T \to V)=(A \to V)$, one obtains a canonical isomorphism $q^*s^*\Omega^m_{A/V} \to \Omega^m_{A/V}$ by extending sections of $s^*\Omega^m_{A/V}$ to sections of $\Omega^m_{A/V}$ invariant under left translation (c.f. \cite[Proposition 4.2.2]{NeronModels}).

For an arbitrary torsor $T$ and a trivialisation $T_{V'} \cong A_{V'}$ over some covering $V'$ of $V$ one obtains an isomorphism 
\begin{equation} \label{BlubIso}
  (q^*s^*\Omega^m_{A/V})_{V'} \cong (\Omega^m_{A/V})_{V'} \cong (\Omega^m_{T/V})_{V'}
\end{equation}
 by composition. The left invariance noted above implies that the isomorphism \eqref{BlubIso} remains unchanged if one changes the trivialisation by the action of an element of $A(V')$. Hence \eqref{BlubIso} descends to a canonical isomorphism $q^*s^*\Omega^m_{A/V} \cong \Omega^m_{T/V}$.
\end{proof}

By Lemma \ref{relativeforms}, the relative forms $\tilde{\omega}_i$ induce a section of $\Omega^\top_{\Mm'_i/\Aa}$ on the $\Pc_i$-torsor $\Mm_i' \subset \Mm_i$ denoted by the same symbol. As before, we define a top degree form
$$\omega_i = \pi^*\omega_{\Aa} \wedge \tilde{\omega}_i.$$

From Remark \ref{cod2} we conclude that the orbifold measure $\mu_{\text{orb},i}$ on $\Mm_i^{\text{coarse}}(\Oo_F)^\natural$ is given by integrating $|\omega_i|$, since the intersection $\Delta \cap \Mm_i' = \emptyset$ is empty, as $\Mm_i'$ cannot contain orbifolds points as a relative $\Pc_i$-torsor.


\begin{definition}\label{notation}
Let $\Aa(\Oo_F)^{\flat}$ be the set
\[\Aa(\Oo_F)^{\flat} =\Aa(\Oo_F) \cap \Aa^{\lozenge}(F).\]
For every $b\in \Aa(\Oo_F)^\flat$ we write $\tilde{\omega}_{i,b}$ for the volume form on the fibres of $\Mm_i \rightarrow \Aa$ and $\Pc_i \rightarrow \Aa$ over $b$. 
\end{definition}

The following lemma allows us to compare the volumes of fibres of dual Hitchin systems. An alternative argument, based on the behaviour of \emph{N\'eron models} with respect to duality of abelian varieties, can be found in the authors' \cite{GWZ18}.

\begin{lemma}[Key Lemma]\label{keylemma}
Let $b \in \Aa(\Oo_F)^{\flat}$ be a rational point such that $\pi_i^{-1}(b)(F)$ is non-empty for $i \in \{1,2\}$. Then we have 
$$\int_{\pi_1^{-1}(b)(F)} |\tilde{\omega}_{1,b}|= \int_{\pi_2^{-1}(b)(F)} |\tilde{\omega}_{2,b}|.$$
\end{lemma}

\begin{proof} Let $\Pc_{i,b}$ denote the fibre of $\Pc_i \rightarrow \Aa$ over $b$. As $\pi_i^{-1}(b)(F)$ is non-empty, there is an isomorphism $\pi_i^{-1}(b) \cong \Pc_{i,b}$ over $F$ and hence
\[ \int_{\pi_i^{-1}(b)(F)} |\tilde{\omega}_{i,b}| = \int_{\Pc_{i,b}(F)} |\tilde{\omega}_{i,b}|. \]
By Proposition \ref{intools} we have
\[ \frac{1}{|\ker \phi(F)|} \int_{\Pc_{1,b}(F)} |\tilde{\omega}_{1,b}| =  \int_{\phi(\Pc_{1,b}(F))} |\tilde{\omega}_{2,b}|. \]
The right hand side can be equated to
$$ \frac{1}{|\Pc_{2,b}(F)/\phi(\Pc_{1,b}(F))|} \int_{\Pc_{2,b}(F)} |\tilde{\omega}_{2,b}|,$$
since by translation invariance of $\tilde{\omega}_{2}$ we have
$$\int_{\Pc_{2,b}(F)} |\tilde{\omega}_{2,b}| = \sum_{[y] \in \Pc_{2,b}(F)/\phi(\Pc_{1,b}(F))} \int_{y\phi(\Pc_{1,b}(F))} |\tilde{\omega}_{2,b}| = |\Pc_{2,b}(F)/\phi(\Pc_{1,b}(F))|\cdot{}\int_{\phi(\Pc_{1,b}(F))} |\tilde{\omega}_{2,b}|$$


By condition (b) of Definition \ref{situation}, the isogeny $\phi$ is a self-dual isogeny and therefore we can apply Proposition \ref{prop:self-dual} to deduce $|(\Pc_2)_b(F)/\phi(\Pc_1)_b(F)| = |\ker \phi(F)|$. This concludes the proof.
\end{proof}




Recall from Construction \ref{fi}, that the $\mu_r$-gerbe $\alpha_i$ on $\Mm_i$ induces a function 
\[ f_{\alpha_i}:M_i(\Oo_F)^\natural \rightarrow \Cb.     \]
The presence of these functions allows us to generalise the previous Lemma \ref{keylemma} to arbitrary fibres.

\begin{theorem}\label{p-adic-Mirror}
For every $b \in \Aa(\Oo_F)^{\flat}$ we have an equality of integrals
$$\int_{\pi_1^{-1}(b)(F)} f_{\alpha_1}|\tilde{\omega}_{1,b}| = \int_{\pi_2^{-1}(b)(F)}f_{\alpha_2}|\tilde{\omega}_{2,b}|.$$
\end{theorem}

\begin{proof}
There are four cases to consider.
\begin{itemize}
\item[(1)] If $\pi_1^{-1}(b)(F) = \pi_2^{-1}(b)(F) = \emptyset$, then both integrals are $0$.

\item[(2)] If $\pi_1^{-1}(b)(F) = \emptyset$, but $\pi_2^{-1}(b)(F)\neq \emptyset$ the left hand side will be $0$, so we have to prove that the integral on the right vanishes. By Definition \ref{situation} (b) the fibres $\Pc_{1,b}$ and $\Pc_{2,b}$ are dual abelian varieties and we denote the Tate duality pairing by
$$(-,-)\colon \Pc_{2,b}(F) \times H^1(F,\Pc_{1,b}) \to \mathbb{Q}/\mathbb{Z} \subset \Cb.$$

By virtue of Lemma \ref{lemmaTate} for an isomorphism of $\Pc_{2,b}$-torsors $\pi_2^{-1}(b) \xrightarrow h \Pc_{2,b}$, the function $f_{\alpha_2} \circ h^{-1}$ is equal to $\xi_b \exp(2\pi i(-,t_1))$ on $\Pc_{2,b}(F)$. Here $\xi_b \in \mu(\Cb)$ is a constant and $t_1 \in H^1(F,\Pc_{1,b})$ denotes the isomorphism class of $\pi_1^{-1}(b) \cong \Split'(\pi^{-1}_2(b),\alpha_2)$. Since the Tate duality pairing is non-degenerate we deduce that up to the constant $\xi_h$, the function
\[f_{\alpha_2} \circ h^{-1}: \Pc_{2,b}(F) \to \Cb, \]
is a non-trivial character on $\Pc_{2,b}(F)$, as by assumption $t_1 \neq 0$. Translation-invariance of $\tilde{\omega}_{2,b}$ therefore implies 
\[\int_{\pi_2^{-1}(b)(F)}f_{\alpha_2}|\tilde{\omega}_{2,b}| = \int_{\Pc_{2,b}(F)} f_{\alpha_2} \circ h^{-1} |\tilde{\omega}_{2,b}| = 0. \]


\item[(3)] The case that $\pi_2^{-1}(b)(F) = \emptyset$, but $\pi_1^{-1}(b)(F) \neq \emptyset$ is treated analogously to case (2).

\item[(4)] If $\pi_i^{-1}(b)(F) \neq \emptyset$ for both $i= 1,2$, then by Definition \ref{situation} (d) the gerbe $\alpha_{i|\pi^{-1}_i(b)}$ splits for $i=1,2$. From Proposition \ref{prop:YIM} we thus see
\[ f_{\alpha_i|\pi_i^{-1}(b)(F)} \equiv 1. \] 
The proposition now follows from our Key Lemma \ref{keylemma}, which shows that $\Pc_{1,b}(F)$ and $\Pc_{2,b}$ have the same volume.
\end{itemize}
These four cases cover all possibilities and therefore establish the formula claimed in the proposition.
\end{proof}

Finally we are ready for the

\begin{proof}[Proof of Theorem \ref{p-Mirror}]
It is enough to prove for every $a \in \Aa(k)$ the fibrewise assertion
$$\mathsf{Tr}(\Fr,(R\pi^I_{1,*}L_1(F_1))_a) = \mathsf{Tr}(\Fr,(R\pi^I_{2,*}L_2(F_2))_a),$$
as the global equality 
$$\#_{\st}^{\alpha_1}(\Mm_1) = \#_{\st}^{\alpha_2}(\Mm_2)$$
follows from summing over all $a \in \Aa(k)$.

Let $\Mm_{i,a}$ be the fibre of $\Mm_i \to \Aa$ over $a$. Using the Grothendieck-Lefschetz trace formula (see \cite[Theorem 12.1(iv)]{kiehl2013weil}) we can write 
\[\mathsf{Tr}(\Fr,(R\pi^I_{i,*}L_i(F_i))_a) = \sum_{x \in \I\Mm_{i,a}(k)_\iso} q^{F_i(x)} \frac{\mathsf{Tr}(\Fr,(L_{\alpha_i})_x)}{|\Aut(x)|}.\] 

Now let $\iota:\I\Mm_{i,a}(k) \rightarrow \I\Mm_{i,a}(k)$ denote the involution sending a pair $(m,\phi) \in \I\Mm_{i,a}(k)$ to $(m,\phi^{-1})$. Then one has for every $x\in \I\Mm_{i,a}(k)$ the relations $F_i(x) = -w_i(\iota(x)) + \dim \Mm_i$ and 
\[ \mathsf{Tr}(\Fr,(L_{\alpha_i})_x) = \overline{\mathsf{Tr}(\Fr,(L_{\alpha_i})_{\iota(x)})},  \]
see \eqref{autgerb} and the proof of Corollary \ref{StringyF} for a similar argument. Thus we get 
\[\mathsf{Tr}(\Fr,(R\pi^I_{i,*}L_i(F_i))_a) = q^{\dim \Mm_i} \sum_{x \in \I\Mm_{i,a}(k)_\iso} q^{-w_i(x)} \frac{\overline{\mathsf{Tr}(\Fr,(L_{\alpha_i})_x)}}{|\Aut(x)|}.\] 
By Theorem \ref{thm:appendix} the right hand side can be written as a $p$-adic integral, namely
\[\sum_{x \in \I\Mm_{i,a}(k)_\iso} q^{-w_i(x)} \frac{\overline{\mathsf{Tr}(\Fr,(L_{\alpha_i})_x)}}{|\Aut(x)|} = \int_{e^{-1}(\I\Mm_{i,a}(k)_\iso)} \bar{f}_{\alpha_i}\mu_{\text{orb},i} =  \int_{e^{-1}(\I\Mm_{i,a}(k)_\iso)} \bar{f}_{\alpha_i} |\omega_i|.\]
Next let $\Aa(\Oo_F)^\flat_a = \{b \in \Aa(\Oo_F)^\flat \ |\ b|_{\Spec(k_F)} = a \}$. Since 
\[e^{-1}(\I\Mm_{i,a}(k)_\iso) = \{b \in \Mm(\Oo_F)^\natural \ |\ \pi_i(b)|_{\Spec(k_F)}= a\},\]
and the complement of $\Aa(\Oo_F)^\flat_a$ inside $\Aa(\Oo_F)_a = \{b \in \Aa(\Oo_F) \ |\ b|_{\Spec(k_F)} = a\}$ has measure zero with respect to $\omega_{\Aa}$ by Proposition \ref{intools}. Furthermore, the resulting map of analytic $F$-manifolds
$$e^{-1}(\I\Mm_{i,a}(k)_\iso \to \Aa(\Oo_F)^\flat_a$$ is a submersion, since it is induced by the smooth morphism $\Mm_i^{\lozenge} \to \Aa^{\lozenge}$ (and taking $F$-rational points). 
Thus we can apply Proposition \ref{Igusa} and get
\[ \int_{e^{-1}(\I\Mm_{i,a}(k)_\iso)} f_{\alpha_i} |\omega_i| = \int_{b\in \Aa(\Oo_F)^\flat_a} |\omega_{\Aa}| \int_{\pi^{-1}(b)(F)} f_{\alpha_i} |\tilde{\omega}_{i,b}|. \]
Theorem \ref{p-Mirror} now follows from Theorem \ref{p-adic-Mirror}.
\end{proof}


\section{The Topological Mirror Symmetry Conjecture by Hausel--Thaddeus}\label{tms}


In this section we explain how our main theorem implies the Topological Mirror Symmetry conjecture by Hausel--Thaddeus \cite{MR1990670}. This is mostly a matter of recalling that the assumptions in \ref{situation} are satisfied. The reader is referred to the original sources by Hitchin \cite{MR887284} and Simpson \cite{MR1320603} for an introduction to Higgs bundles.

\subsection{Moduli Spaces of Higgs Bundles}
\label{sec:HiggsModuli}

Our strategy requires us to consider moduli spaces of Higgs bundles over various base schemes. For the sake of avoiding awkward language we fix a Noetherian scheme $S$, and consider a smooth and proper morphism $X \to S$ whose geometric fibres are connected curves of a fixed genus $g$. Below we will recall the definition and basic properties of moduli spaces of Higgs bundles over $X/S$. We assure the minimalists amongst the readers that only the following down-to-earth cases are relevant to us: $S = \Spec \Cb$, $S = \Spec \Fb_q$, $S = \Spec \Oo_F$, $S = \Spec F$, $S = \Spec R$, where $F$ is a local field and $R \subset \Cb$ is a finite type subalgebra of $\Cb$.

\begin{definition}
\begin{itemize}
\item[(a)] Let $D$ be a line bundle on $X$. A $D$-Higgs bundle is a pair $(E,\theta)$, where $E$ is a vector bundle on $S$, and $\theta\colon E \to E \otimes D$ an $\Oo_X$-linear morphism.
\item[(b)] For integers $n$ and $d$ we denote by $\Mm_{\GL_n}^d(X)$ the moduli space of \emph{stable} $D$-Higgs bundles of rank $n$ and degree $d$.
\item[(c)]For a line bundle $L$ of degree $d$ on $X$, a line bundle $D$ of arbitrary degree and an integer $n$ we denote by $\Mm_{\SL_n}^L(X)$ the moduli space of \emph{stable} $D$-Higgs bundles $(E,\theta)$ together with an isomorphism $\det(E) \simeq L$ satisfying $\mathsf{Tr}\,\theta = 0$. 
\end{itemize}
\end{definition}

For the rest of this section we fix a line bundle $L$ on $X$ of degree $d$ as well as a line bundle $D$ and an integer $n$. Traditionally one chooses $D$ to be equal to the canonical line bundle $\Omega_{X/S}^1$. However we do not need this restriction, and the general case is of independent interest. We denote by $J_X$ the Jacobian of $X$ over $S$ and by $\Gamma=J_X[n]$ the associated finite flat group scheme of $n$-torsion points.

\begin{rmk}
Existence of the moduli space in this generality can be deduced easily from algebraicity of the stack of vector bundles $\Bun_n(X/S)$ on $X$ (Olsson's algebraicity result for mapping stacks \cite[Theorem 1.1]{olsson2006hom} implies algebraicity of $\Bun_n(X/S)$). For this d\'evissage argument one considers the stack of all $D$-Higgs bundles and the forgetful map $\mathsf{Higgs}_D(X/S) \to \Bun_n(X/S)$. It follows from \cite[Th\'eor\`eme 7.7.6]{MR0163911} that this map is representable and affine. Therefore the stack $\mathsf{Higgs}_D(X/S)$ is itself algebraic. Since stable $D$-Higgs bundle form an open substack, we obtain algebraicity of the stack of stable $D$-Higgs bundles. The corresponding moduli space can be obtained by rigidifying this stack with respect to the group $\G_m$. Rigidification preserves algebraicity (see \cite{abramovich2003twisted}), and hence we deduce that $\Mm_{\SL_n}^L(X)$ is representable by an algebraic space.
\end{rmk}

Henceforth we will leave $D$ implicit, and simply refer to $D$-Higgs bundles as Higgs bundles.  However, we emphasise that according to our conventions, $\Mm_{\SL_n}^L(X)$ is the space of stable $\SL_n$-Higgs bundles. Nonetheless, the case of principal interest is when $n$ and $d$ are coprime integers (see Theorem \ref{thm:proper}). It is well-known that in this case $\Mm_{\SL_n}^L(X)$ is a smooth variety, which is acted on by the finite group scheme of $n$-torsion points $\Gamma$. As we work in a more general setting than usual we provide a proof of smoothness:

\begin{lemma}\label{Lsmooth}
Assume that $D \otimes (\Omega_{X/S}^1)^{-1}$ is a line bundle, which is either of strictly positive degree or equal to $\Oo_{X/S}$. Then the moduli space $\Mm_{SL_n}^L(X/S)$ is smooth over $S$.
\end{lemma}

\begin{proof}
Without loss of generality we may assume that $S$ is affine. The deformation theory of (twisted) $\SL_n$-Higgs bundles $(E,\theta)$ (over an arbitrary base) is governed by the (relative) hypercohomology of the complex (here $\Eend_0(E)$ denotes the sheaf of trace-free endomorphisms of $E$)
$$C^{\bullet}(E,\phi)= [\Eend_0(E) \to \Eend_0(E) \otimes D]$$
sitting in degrees $-1$ and $0$. We refer the reader to \cite[4.14]{MR2653248} for derivation of this fact in a general context.
We have natural isomorphisms $\Hb^0(X,C^{\bullet}(E,\phi)) \simeq \End(E,\phi)$, $\Hb^1(X,C^{\bullet}(E,\phi)) \simeq T_{(E,\phi)}\Mm_{\SL_n}^L$, and $\Hb^2(X,C^{\bullet}(E,\phi))$ equals the space of obstructions. In order to show that $\Mm_{\SL_n}^L(X/S)$ is smooth, we have to show vanishing of $\Hb^2(X,C^{\bullet}(E,\phi))$ for a stable Higgs bundle $(E,\phi)$. Since stable $\SL_n$-Higgs bundles have a discrete group of automorphisms the group $\End(E,\phi) \simeq \Hb^0(X,C^{\bullet}(E,\phi))$ vanishes. Serre duality applied to the family of curves $X/S$ implies 
$$\Hb^2(X,C^{\bullet}(E,\phi))^{\vee} \simeq \Hb^0(X,C^{\bullet}(E,\phi)^{\vee} \otimes \Omega_{X/S}^1).$$
The complex $C^{\bullet}(E,\phi)^{\vee} \otimes \Omega_{X/S}^1$ is given by 
$$[\Eend_0(E) \otimes D^{-1} \otimes \Omega_X^1 \to \Eend_0(E) \otimes \Omega_X^1] \simeq [\Hhom(E,E\otimes D^{-1} \otimes \Omega_X^1) \to \Hhom(E,E\otimes D^{-1} \otimes \Omega_X^1) \otimes D].$$
Therefore, $\Hb^0$ of this complex describes the space of trace-free homomorphisms of Higgs bundles $$\Hom_0((E,\phi),(E\otimes D^{-1} \otimes \Omega_X^1,\phi).$$ 
Let us assume that the degree of $D \otimes (\Omega_{X/S}^1)^{-1}$ is strictly positive. Since $(E,\phi)$ and $(E\otimes D^{-1} \otimes \Omega_X^1,\phi)$ are stable and the second Higgs bundle is of strictly smaller degree than the first, we have that this space of homomorphisms is $0$. Similarly, if $D \otimes (\Omega_{X/S}^1)^{-1}$ is equal to $\Oo_{X/S}$, then we have $\End_0((E,\phi),(E,\phi)) = 0$ as noted above.
\end{proof}

\begin{definition}\label{defi:SL_n_L}
We denote by $\Mm_{\PGL_n}(X/S)$ the moduli stack of families of $\PGL_n$-Higgs bundles, which admit a presentation as a stable Higgs bundle over each geometric point. The notation $\Mm_{\PGL_n}^d(X/S)$ refers to the moduli stack of stable $\PGL_n$-Higgs bundles, which admit a presentation by a vector bundle of degree congruent to $d$ modulo $n$ over each geometric point. 
\end{definition}

We will sometimes denote $\Mm_{\SL_n}^L(X/S)$ and $\Mm_{\PGL_n}^d(X/S)$ simply by $\Mm_{\SL_n}^L$ and $\Mm_{\PGL_n}^d$.
\begin{rmk}
\begin{enumerate}[(a)]
\item The connected components of the moduli stack of stable $\PGL_n$-bundles are parametrised by congruence classes of integers modulo $n$. That is, we have $\Mm_{\PGL_n} = \bigsqcup_{\bar{d} \in \Zb/n\Zb} \Mm_{\PGL_n}^{d}$.
\item The stack $\Mm_{\PGL_n}^d$ is equivalent to the quotient stack $[\Mm_{\SL_n}^L(X)/\Gamma]$. In particular, we see that the resulting quotient stack only depends on the degree $d$ (modulo $n$) of the line bundle $L$.
\end{enumerate}
\end{rmk}

The Hitchin base $\Aa$ is defined to be the affine $S$-space corresponding to the locally free sheaf $$\oplus_{i=2}^nH^0(X,D^{\otimes i}).$$ It receives a morphisms $\chi_{\SL_n}$ and $\chi_{\PGL_n}$ (called the Hitchin map) from the moduli spaces $\Mm_{\SL_n}^L(X/S)$ and $\Mm_{\PGL_n}^d(X/S)$. These maps are given by the familiar construction of characteristic polynomials, applied to the Higgs field $\theta$ itself. 

\begin{theorem}[Hitchin, Nitsure, Faltings]\label{thm:proper}
If $d$ and $n$ are coprime, then the morphisms $\chi_{\SL_n},\chi_{\PGL_n} $ are proper.
\end{theorem}

See Nitsure's \cite{MR1085642} for a proof of properness in the case of $\GL_n$-Higgs bundles, which implies the assertion for $\SL_n$, or Faltings's \cite{MR1211997} for a proof of the case of $G$-Higgs bundles for reductive $G$.

We conclude this subsection by mentioning a connection between the notion of twisting (see Subsection \ref{twisting}) and the moduli spaces $\Mm_{\SL_n}^L(X)$ for varying $L$. The proof of Theorem \ref{thm:independent} is based on this property. We let $F$ be a local field and take $S=\Spec(\Oo_F)$. We let $M$ be a line bundle of degree $0$ on $X$ and assume that $n$ is an integer prime to the residue characteristic of $F$. We associate to $M$ a torsor under $\Gamma = J_X[n]$ as follows: The multiplication by $n$ map is an isogeny
$$0 \to \Gamma \to J_X \xrightarrow{[n]} J_X \to 0.$$
This sequence induces a long exact sequence of (unramified) Galois cohomology groups as part of which we get a boundary map
$$\delta\colon J_X(\Oo_F) \to H^1_{\text{ur}}(F,\Gamma).$$
The torsor associated to $M$ is $\delta(M)$. It can also be understood as the fibre $[n]^{-1}(M)$ with its natural $\Gamma$-action. 

\begin{lemma}\label{lemma:twisty}
There is an isomorphism 
$$\Mm_{\SL_n}^{LM}(X) \simeq (\Mm_{\SL_n}^L(X))_{\delta(M)},$$
where $()_{\delta(M)}$ denotes the twist by the torsor $\delta(M)$ (as in Definition \ref{defi:twisting}).
\end{lemma}

\begin{proof}
We denote by $F^{\text{ur}}$ the unramified closure of $F$, and by $\Oo_{F^{\text{ur}}}$ its ring of integers. The latter is a discrete valuation ring with algebraically closed residue field $\overline{k}_F$. Hence there exists a line bundle $M^{\frac{1}{n}} \in \Pic(X_{\Oo_{F^{\text{ur}}}})$ for which $(M^{\frac{1}{n}})^n \simeq M$. There is a morphism 
$$\Mm_{\SL_n}^{LM}(X)_{\Oo_{F^{\text{ur}}}} \to \Mm_{\SL_n}^L(X)_{\Oo_{F^{\text{ur}}}}$$
given by tensoring a family of Higgs bundles with $M^{\frac{1}{n}}$. The obstruction for this isomorphism to descend to an isomorphism defined over $\Oo_F$ is precisely given by $\delta(M) \in H^1_{\text{ur}}(F,\Gamma)$. This shows that we have an isomorphism $\Mm_{\SL_n}^{LM}(X) \simeq (\Mm_{\SL_n}^L(X))_{\delta(M)}$.
\end{proof}

\subsection{The Prym variety and its properties}

We denote by $\pi\colon Y \to X \times \Aa$ the universal family of spectral curves. The moduli stacks $\Mm_{\SL_n}^L(X)$ and $\Mm_{\PGL_n}^d(X)$ are acted on by smooth commutative group schemes $\Pc_{\SL_n}/\Aa$ and $\Pc_{\PGL_n}/\Aa$: First we consider the $\SL_n$ side, where $\Pc_{\SL_n}$ is given by the relative \emph{Prym variety} of the universal family of spectral curves $Y \to X \times \Aa$.

Let $\Aa^{\lozenge} \subset \Aa$ be the open dense subscheme corresponding to smooth spectral curves. Over this open subset, the fibres of the Hitchin map $\chi\colon \Mm_{\SL_n}^L \rightarrow \Aa$ admit the following well-known description. For a finite morphism $C \to C'$ of relative curves over $S$ we refer the reader to Hausel--Pauly's \cite[Section 3]{MR2967059} for the definition of the norm map $\Nm_{C/C'} \colon \Pic(C/S) \to \Pic(C'/S)$. Although the authors of \emph{loc. cit.} work under more restrictive assumptions, their treatment of the norm map is easily generalised to our more general situation. For a line bundle $L$ on $C'$ we denote by $\Nm_{C/C'}^{-1}(L)$ the $\G_m$-rigidification of the stack obtained by taking the preimage of $L$ under the morphism $\Nm_{C/C'}$.
We will sometimes abbreviate $\Pc_{\SL_n}$ as $\Pc$.
\begin{lemma}\label{Norm_determinant}
Let $a \in \Aa^{\lozenge}(S)$ be an $S$-valued point. 
\begin{itemize}
\item[(a)] The fibre $\chi^{-1}(a) = \Mm_{\SL_n}^L \times_{\Aa} S$ is naturally equivalent to the stack $$\Nm_{Y\times_{\Aa}S/X}^{-1}(L \otimes \det(\pi_*\Oo_{Y \times_{\Aa} S})).$$ 
\item[(b)] The equivalence of (a) is an equivalence of $\Pc_a$-torsors.
\item[(c)] For line bundles $M_1$ and $M_2$ on $X$ we have the following identity in $H^1_{\text{\'et}}(S,\Pc)$:
$$[{\Nm^{-1}_{Y\times_{\Aa}S/X}(M_1)}][\Nm^{-1}_{Y\times_{\Aa}S/X}(M_2)] = [\Nm^{-1}_{Y\times_{\Aa}S/X}(M_1M_2)],$$
where $[{\;}]$ denotes the class of the $\Pc$-torsor in $H^1_{\text{\'et}}(S,\Pc)$.
\item[(d)] For a line bundle $M \in \Pic(X)$ we have an identity of torsors $[{\Nm^{-1}_{Y\times_{\Aa}S/X}}(M^n)] = 0$ in $H^1_{\text{\'et}}(S,\Pc)$. 
\end{itemize}
\end{lemma}

\begin{proof}
The first part is a consequence of the formula $$\Nm_{Y\times_{\Aa}S/X}({L'}) = \det(\pi_*{L'}) \cdot{} \det(\pi_*\Oo_{{Y\times_{\Aa} S}})^{-1}$$ for a line bundle $L'$ on $Y \times_{\Aa} S$ (see \cite[Corollary 3.12]{MR2967059}). 
The second and third parts follow from the multiplicativity of the norm map: Indeed, the tensor product of line bundles on $Y$ induces a map
$${\Nm^{-1}_{Y\times_{\Aa}S/X}(M_1)} \times {\Nm^{-1}_{Y\times_{\Aa}S/X}(M_2)} \to {\Nm^{-1}_{Y\times_{\Aa}S/X}(M_1M_2)}.$$
This map has the property of being a bilinear map of $\Pc_a$-torsors, and therefore induces a morphism of $\Pc_a$-torsors 
$${\Nm^{-1}_{Y\times_{\Aa}S/X}(M_1)} \otimes {\Nm^{-1}_{Y\times_{\Aa}S/X}(M_2)} \to {\Nm^{-1}_{Y\times_{\Aa}S/X}(M_1M_2)}.$$
As a morphism of $\Pc_a$-torsors this is automatically an isomorphism.

The fourth assertion is a consequence of the identity $\Nm_{Y\times_{\Aa}S/X}(\pi^*M) = M^{\otimes n}$ and multiplicativity: Multiplication with $\pi^*M$ induces a morphism of torsors ${\Nm^{-1}_{Y\times_{\Aa}S/X}(\Oo_X)} \to {\Nm^{-1}_{Y\times_{\Aa}S/X}(\pi^*M)}$, which is automatically an isomorphism.
\end{proof}

We will now turn to a description of the $\PGL_n$-counterpart of the Prym-variety $\Pc_{\SL_n}/\Aa$. At first we recall its definition, which renders the description of the action on $\Mm_{\PGL_n}^e$ relative to $\Aa$ tautological. We then turn to the verification of the properties demanded in Definition \ref{situation}. Recall that we denote $\Pic(X)[n]$ by $\Gamma$.

\begin{definition}
We define $\Pc_{\PGL_n} = \Pc_{\SL_n}/\Gamma$.
\end{definition}

We include the proof of the assertion below for the convenience of the reader since the original reference does not comment on the self-duality property of the isogeny. In the following, for an $\Aa$-scheme $Z$ we denote by $Z^\lozenge$ the base change $Z \times_{\Aa} \Aa^\lozenge$.

\begin{proposition}[c.f. {\cite[Lemma 2.2 and 2.3]{MR2967059}}]  \label{dual}
There is an isomorphism of abelian $\Aa^{\lozenge}$-schemes $(\Pc_{\SL_n}^{\lozenge})^{\vee}$ and $\Pc_{\PGL_n}^{\lozenge}.$ With respect to this isomorphism the quotient map $\Pc_{\SL_n}^{\lozenge} \to \Pc_{\PGL_n}^{\lozenge}$ is a self-dual isogeny. Furthermore, the identification $\Gamma \cong \Gamma^\vee$ associated to this self-dual isogeny via \eqref{SelfDualKernel} is the same as the one associated to the self-dual isogeny $J_X \xrightarrow{[n]} J_X$.
\end{proposition}

\begin{proof}
We begin the proof by fixing notation. The relative Jacobian of the trivial family of curves $X \times_S \Aa/\Aa$ will be denoted by $J$. The relative Jacobian of the universal spectral curve $Y / \Aa$ will be denoted by $\widetilde{J}$. Similarly we denote by $J^1$ and $\widetilde{J}^1$ the relative moduli spaces of degree $1$ line bundles.

Henceforth we restrict every $\Aa$-scheme to the open subset $\Aa^{\lozenge}$. To avoid awkward notation we will omit the corresponding superscript.

The relative norm map induces a morphism of abelian $\Aa$-schemes $\widetilde{J} \xrightarrow{\Nm} J$. Similarly, pullback of line bundles yields $\pi^*\colon J \to \widetilde{J}$. We claim that these two morphisms are dual to each other with respect to the canonical isomorphism $J^{\vee} \simeq J$ induced by the Poincar\'e bundle (and similarly for $\widetilde{J}$). To see this we observe that we have a commutative diagram (the horizontal arrows represent the Abel-Jacobi map)
\[
\xymatrix{
Y \ar[r]^{\mathsf{AJ}_{Y}} \ar[d]_{\pi} & \widetilde{J}^1 \ar[d]^{\Nm} \\
X \ar[r]^{\mathsf{AJ}_{X}} & J 
}
\]
to which we can apply the contravariant $\Pic^0(?/\Aa^\lozenge)$ functor to obtain the commutative diagram of abelian schemes
\[
\xymatrix{
\widetilde{J} & \widetilde{J} \ar[l]_{\id} \\
J \ar[u]^{\pi^*} & J \ar[l]_{\id} \ar[u]_{\Nm^{\vee}}.
}
\]
We will now show that the dual of the isogeny 
\begin{equation}\label{eqn:isogeny}
\Gamma \to \Pc_{\SL_n} \to \Pc_{\PGL_n} 
\end{equation}
is equivalent to itself. A convenient framework for the argument is provided by the theory of abelian group stacks, as explained in \cite{AriApp}. Equivalently, one could employ a derived category of commutative group sheaves. For (nice) abelian group stacks there exists a duality functor given by $\Hhom(-,B\G_m)$. It sends an abelian scheme to its dual, and a finite \'etale group scheme $\Gamma$ to $B\Gamma^{\vee}$, the classifying stack of its Cartier dual. The sequence of maps in \eqref{eqn:isogeny} is sent to the fibre sequence
\begin{equation}\label{eqn:isogeny2}
\Pc_{\PGL_n}^{\vee} \to \Pc_{\SL_n}^{\vee} \to B\Gamma^{\vee},
\end{equation}
where the first map is the sought-for dual isogeny. It is sufficient to show that $\Pc_{\SL_n}^{\vee} \to B\Gamma^{\vee}$ is equivalent to the map $\Pc_{\PGL_n} \to B\Gamma$. This amounts to the following three claims: That the dual of $\Pc_{\SL_n}$ is $\Pc_{\PGL_n}$, that $\Gamma^{\vee} \simeq \Gamma$, and that the isogeny \eqref{eqn:isogeny} is self-dual. We will obtain these assertions by analysing the two commutative diagrams below, which are related by the duality functor $\Hhom(-,B\G_m)$. 
\[
\xymatrix{
\Gamma \ar[r] \ar[d] & \Pc_{\SL_n} \ar[d] & & & B\Gamma^{\vee} & \Pc^{\vee}_{\SL_n} \ar[l] \\
J \ar[r]^{\pi^*} \ar[d]_{n} & \widetilde{J} \ar[d]^{\Nm} & & \ar@{<~>}[l] & J \ar[u] & \widetilde{J} \ar[u] \ar[l]_{\Nm} \\
J \ar[r]^{\id} & J & & & J \ar[u]^n & J . \ar[u]_{\pi^*} \ar[l] 
}
\]
Furthermore, the top row of the first diagram is the fibre of the vertical arrows and hence the top row of the second diagram is the corresponding cofibre. By explicitly computing the fibres in the second diagram we obtain a commuting square
\[
\xymatrix{
B\Gamma^{\vee} & \Pc_{\SL_n}^{\vee} \ar[l] \\
B\Gamma \ar[u]^{\simeq} & \Pc_{\PGL_n}, \ar[l] \ar[u]_{\simeq}
}
\]
where we have observed that $B\Gamma$ is the cofibre of $J \xrightarrow{n} J$ and $\Pc_{\PGL_n}$ is by definition the cofibre of $J \xrightarrow{\pi^*} \widetilde{J}$. This concludes the proof of the first assertion. The last part of the claim follows from the construction above.
\end{proof}

The following verifies condition (c) of Definition \ref{presituation}.
\begin{proposition}[{Ng\^o, \cite[Proposition 4.3.3]{MR2653248}}] \label{MPrime}
There exist natural open and dense open subspaces $\Mm_{\SL_n}^L(X)'$ of $\Mm_{\SL_n}^L(X)$ and $\Mm_{\PGL_n}^d(X)'$ of $\Mm_{\PGL_n}^d(X)$ on which the Prym varieties $\Pc_{\SL_n}$ and $\Pc_{\PGL_n}$ act faithfully and transitively. Their complements are of codimension $\geq 2$. 
\end{proposition}

Altogether we have shown:
\begin{proposition} \label{AbsHSysSummary}
Assume that $\Gamma$ is constant over $S$. Then, $\Mm_{\SL_n}^L(X)\to \Aa$ and $\Mm_{\PGL_n}^d(X) \to \Aa$ are abstract Hitchin systems over $R$.
\end{proposition}
\begin{proof}
It follows from the assumption that $\Mm_{\PGL_n}^d(X)$ is the quotient of $\Mm_{\SL_n}^L$ with respect to the abstract finite abelian group $\Gamma$. Since $\Mm_{\SL_n}^L$ is smooth (Lemma \ref{Lsmooth}), we deduce that is a finite abelian quotient stack. The Hitchin system $\chi$ and the action of the Prym varieties $\Pc_{\SL_n}$ and $\Pc_{\PGL_n}$ was shown above (Proposition \ref{MPrime}) to satisfy the requirements of the definition of abstract Hitchin systems.
\end{proof}
We conclude this subsection with two technical lemmas that will be needed later.

\begin{lemma}\label{order2}
For $a \in \Aa^{\lozenge}(S)$ we denote by $P^0_a$ the fibre $(\Pc_{\SL_n})_a$, and by $P^L_a$ the fibre $(\Mm_{\SL_n}^L)_a$. Moreover, for the finite flat morphism $\pi\colon Y \to X \times \Aa$, we write $M = \det(\pi_*\Oo_Y)$.
\begin{enumerate}[(a)]
\item We have an abstract isomorphism of $P_a^0$-torsors $P_a^L  \cdot{} P_a^{L'} \simeq P_a^{LL'M^{-1}}$.
\item For an integer $d$ we have an abstract isomorphism of $P_a^0$-torsors between $(P^{L}_a)^d$ and $P_a^{L^dM^{-d + 1}}$.
\item A line bundle $N \in \Pic(X)$ induces an abstract isomorphism of $P_a^0$-torsors $P^L_a \simeq P_a^{L N^n}$.
\end{enumerate}
\end{lemma}

\begin{proof}
According to Lemma \ref{Norm_determinant} we identify the torsor $P_a^L$ with (rigidification of) the fibre $\Nm^{-1}_{Y\times_{\Aa}S/X}(LM)$. Multiplication of line bundles $\Pic(Y) \times \Pic(Y) \to \Pic(Y)$ yields a $P_a^0$-bilinear map $P_a^{L} \times P_a^{L'} \to P_a^{LL'M}$. This induces a morphism of torsors $P_a^L  \cdot{} P_a^{L'} \to P_a^{LL'M}$. Since a morphism of torsors is automatically an isomorphism, we conclude that assertion (a) must hold. Statement (b) follows by induction.

Assertion (c) is based on the fact that multiplication by $\pi^*N$ induces an isomorphism of torsors $P_a^L \to N \cdot{} P_a^L$. Using Lemma \ref{Norm_determinant} again, in combination with the formula $\Nm_{Y\times_{\Aa}S/X}(L \cdot \pi^*N) = \Nm_{Y\times_{\Aa}S/X}(L) \cdot N^n$ (see \cite[Proposition 3.10]{MR2967059}), we obtain an isomorphism of $P_a^0$-torsors $P^L_a \simeq P_a^{L N^n}$.
\end{proof}

Over a local field every torsor over an abelian variety is of finite order. In our particular situation, we can show that the order of the $P_a^0$-torsor $P^L_a$ divides $n$.

\begin{lemma}\label{order}
Let $F$ be a local field whose residue characteristic satisfies $p > n$, and $X$ a curve over $\Spec \Oo_F$. We assume that $a \in \Aa^{\lozenge}(F)$ is an $F$-point, which extends to an $\Oo_F$-point of $\Aa$. We denote by $d_{L,a}$ the order of the $P_a$-torsor $P^L_a$. Then we have $d_{L,a}|n$.
\end{lemma}

\begin{proof}
This follows from Lemma \ref{Norm_determinant}. To see this, observe that there is a line bundle $M'$ on $X$, such that $P^L_a \simeq {\Nm^{-1}_{Y\times_{\Aa}S/X}(M')}$. By virtue of \ref{Norm_determinant}(b-d) we have 
$n\cdot{} [P^L_a] = [{\Nm^{-1}_{Y\times_{\Aa}S/X}((M')^n)}] = 0$ in $H^1_{\text{\'et}}(S,\Pc)$.
\end{proof}

\subsection{On a conjecture by Mozgovoy--Schiffmann}

In this subsection we use $p$-adic integration to show independence of the point counts of moduli spaces of Higgs bundles from the degree $d$ as long as it is coprime to $n$, the independence of $d$ part of \cite[Conjecture 1.1]{mozgovoy2014counting} for $d$ and $n$ coprime. These arguments are independent from Section \ref{sec:mirror}.

There is an alternative proof of this fact by Yu \cite{YH18} using automorphic methods.  A full proof of the conjecture, without the coprimality assumption on $(d,n)$, has been obtained by Mellit \cite{mellit2017poincar} by combinatorial means.

\begin{theorem}\label{mozschiff}
Let $n$ and $d$ be positive coprime integers, let $k$ be a finite field of characteristic $p$ prime to $n$ and $d$, let $X/k$ be smooth proper curve of genus $g$, and let $D$ be a line bundle on $X$ such that $D \otimes (\Omega_X^1)^{-1}$ is of strictly positive degree or equal to $\Oo_X$. We assume that $k$ contains a primitive $n^{2g}$-th root of unity. Then, for any integer $e$ prime to $n$ and $p$ we have $\#\Mm_{\GL_n}^d(X)(k) = \#\Mm_{\GL_n}^{e}(X)(k).$
\end{theorem}

\begin{proof}
We choose a local field $F$ with $k_F = k$ and lift $X$ and $D$ to $\Oo_F$. Note that there are no obstructions to lifting, since $X$ is a curve.

We will show that the $p$-adic volume of the moduli space $\Mm_{\GL_n}^d(X)(\Oo_F)$ is independent of $d$. Since for $d$ and $n$ coprime, the moduli space $\Mm_{\GL_n}^d$ is smooth, we obtain the asserted comparison of point-counts over $k_F$, by evoking Weil's equation \eqref{eqn:weil} 
$$\vol\left(\Mm_{\GL_n}^d(X)(\Oo_F)\right) = \frac{\#\Mm_{\GL_n}^d(X)(k_F)}{q^{\dim \Mm_{\GL_n}^d(X)}}.$$
The comparison of $p$-adic volumes however remains true for arbitrary $d$. Henceforth $d$ can be an arbitrary integer.

The morphism $\pi\colon \Mm_{\GL_n}^d(X) \to \Aa_{\GL_n} = \bigoplus_{i=1}^{n}H^0(X,D^{\otimes i})$ is also an abstract Hitchin system in the sense of Definition \ref{presituation}. In particular, there exists a regular part $\Mm_{\GL_n}^d(X)'$ which is a torsor under a group $\Aa_{\GL_n}$-scheme $\Pc_{\GL_n}$. The disjoint union of these open subschemes is isomorphic to a relative Picard space
$$\bigsqcup_{d \in \mathbb{Z}}\Mm_{\GL_n}^d(X)' \simeq {\Pic}(Y/\Aa_{\GL_n}),$$
where $Y\to \Aa_{\GL_n}$ denotes the universal family of spectral curves. This observation has the following consequence: for $a \in \Aa_{\GL_n}(\Oo_F)$ we have that the corresponding $(\Pc_{\GL_n})_a$-torsor $\pi^{-1}(a)'$ is trivial: Indeed, the group scheme $\Pc_{\GL_n,a}$ decomposes over $k_F$ as 
$$1 \to \Pc_{\GL_n,a,k_F}^{\circ} \to \Pc_{\GL_n,a,k_F} \to \pi_0(\Pc_{\GL_n,a,k_F}) \to 1.$$
By Lang's theorem \cite[Theorem 2]{10.2307/2372673} the group $H^1_{\text{\'et}}(k_F,\Pc^\circ_{\GL_n,a,k_F})$ is zero. Hence every connected component of $\pi^{-1}(a)'_{k_F}$ has a $k_F$-rational point, which by smoothness extends to $\Oo_F$. This shows that every connected component of $\bigsqcup_{d \in \mathbb{Z}}\pi^{-1}(a)'$ has an $\Oo_F$-rational point.

Arguing as in Section \ref{arithmeticduality}, we choose a gauge form $\eta$ on $\Aa_{\GL_n}$ and a relative gauge form $\omega$ over $\Aa$ (that is, a translation-invariant generators of the sheaf $\Omega^{\mathsf{top}}_{\Pc^{\circ}/\Aa}$) and obtain
$$\int_{\Mm_{\GL_n}^d(\Oo_F)}  d\mu_{\orb}  = \int_{a \in \Aa_{\GL_n}(\Oo_F)^{\flat}} \left(\int_{\pi^{-1}(a)(F)} |\omega_a|\right) |\eta| = \int_{a \in \Aa(\Oo_F)^{\flat}} \left(\int_{\Pc_a(F)} |\omega_a| \right)|\eta|.$$
The last equality holds because the torsors $\pi^{-1}(a)$ are trivial by the paragraph above. We conclude the proof by observing the independence of $d$ of the right hand side.
\end{proof}

There exists a variant of the above results for $\SL_n$-Higgs bundles:
\begin{theorem}\label{thm:independent}
Let $n$ and $d$ be positive coprime integers and $k$ be a finite field of characteristic $p>n$. We consider a smooth proper curve $X/k$ of genus $g$ endowed with a line bundle $D$, such that $D \otimes (\Omega_X^1)^{-1}$ is of strictly positive degree or equal to $\Oo_X$. Let $L$ be a line bundle on $X$ of degree $d$. We assume that $k$ contains a primitive $n^{2g}$-th root of unity. Then, there exists a degree $0$ line bundle $N$ on $X$, such that for any $e$ and $n$ coprime we have $\#\Mm_{\SL_n}^L(X)(k) = \#\Mm_{\SL_n}^{L^eN^{e-1}}(X)(k)$.
\end{theorem}
\begin{proof}
We choose a local field $F$ with $k_F = k$, and lift $X$, $D$ and $L$ to $\Oo_F$. This is possible since $X$ is a curve.

Let $a \in \Aa^{\lozenge}(F)$. We write $P^0_a$ to denote the Prym variety acting faithfully and transitively on the Hitchin fibre $\chi^{-1}(a) \subset \Mm_{\SL_n}^L(X)$. Recall that this action is induced by base change of the relative action of $\Pc$ on $\Mm_{\SL_n}^L(X)$, which endows the latter with a torsor structure over $\Aa^{\lozenge}$. 

 We will now compute the $p$-adic volume of $\Mm_{\SL_n}^L$. According to \ref{intools} we have
$$\vol \left(\Mm_{\SL_n}^L(X)(\Oo_F)\right) = \vol \left(\Mm_{\SL_n}^L(X)(\Oo_F) \cap \Mm_{\SL_n}^{L,\lozenge}(X)(F) \right).$$
The right hand side can be computed as a double integral by applying \ref{Igusa} (and a second time \ref{intools}):
\begin{align*}\vol \left(\Mm_{\SL_n}^L(X)(\Oo_F) \cap \Mm_{\SL_n}^{L,\lozenge}(X)(F) \right) &= \int_{\Aa(\Oo_F) \cap \Aa^{\lozenge}(F)} \left(\int_{P^{L}_a(F)} |\omega_a|\right) |\eta|\\ &= \int_{\Aa(\Oo_F) \cap \Aa^{\lozenge}(F)} \left(\int_{P^0_a(F)} \delta_{P^L_a} |\omega_a|\right) |\eta|,\end{align*}
where $\delta_{P^{L}_a}$ denotes the indicator function of the subset $\{a \in \Aa(\Oo_F) \cap \Aa^{\lozenge}(F)|P^{L}_a \text{ is the trivial torsor}\}$, $\eta$ a gauge form on $\Aa$ and $\omega$ a global translation-invariant generator of $\Omega^{\mathsf{top}}_{\Pc^{\circ}/\Aa}$.

Since $e$ is chosen to be prime to $n$, and the order of $P^L_a$ in $H^1(F,P^0_a)$ divides $n$ (Lemma \ref{order}), we see $\delta_{(P^{L}_a)^e} = \delta_{P^L_a}$. It suffices therefore show the existence of a line bundle $N$ of degree $0$, such that the torsor $(P^{L}_a)^e$ is isomorphic to $P_a^{L^eN^{e-1}}$.

According to Lemma \ref{order2}(b) we have an equivalence $(P^{L}_a)^e \simeq P_a^{L^eM^{e-1}}$ of $P_a^0$-torsors. The degree of $M$ is equal to 
$$\sum_{i=0}^{n-1} (-\deg D)i = (-\deg D)n(n-1),$$
and hence is divisible by $n$. We let $Q$ be a line bundle on $X$, such that $\deg Q^n = \deg M$. We define $N = M \cdot{} Q^{-n}$. According to Lemma \ref{order2} we have equivalences of $P_a^0$-torsors.
$$P_a^{L^eM^{e-1}} \simeq P_a^{L^eM^{e-1}Q^{-n(e-1)}}\simeq P^{L^e N^{e-1}}.$$
This implies $\delta_{P^L_a} = \delta_{P^{L^e N^{e-1}}_a}$, and therefore we have
$$\vol \left(\Mm_{\SL_n}^L(\Oo_F)\right) = \vol \left(\Mm_{\SL_n}^{L^e N^{e-1}}(\Oo_F)\right).$$
The connection between $p$-adic volumes and point counts yields $\#\Mm_{\SL_n}^L(k) = \#\Mm_{\SL_n}^{L^eN^{e-1}}(k)$.
\end{proof}

Applying the usual reduction argument one deduces from this equality of point counts an agreement of Betti and Hodge numbers. 

\begin{corollary}
Let $n$ and $d$ be positive coprime integers. We consider a smooth proper curve $X/\Cb$ endowed with a line bundle $D$, such that $D \otimes (\Omega_X^1)^{-1}$ is of strictly positive degree or equal to $\Oo_X$. Let $L$ be a line bundle on $X$ of degree $d$. The Betti and Hodge numbers of the complex manifolds $\Mm_{\SL_n}^L(X)$ and $\Mm_{\GL_n}^d(X)$ are independent of $d$.
\end{corollary}

\subsection{Topological mirror symmetry for moduli spaces of Higgs bundles}\label{conclusion}
We continue to use the setup from the previous subsections for $S=\Spec(R)$, where $R$ is a Noetherian ring in which $n$ is invertible. We also assume that $n$ and $d$ are coprime, and that the \'etale group scheme $\Gamma=\mathsf{Pic}^0(X)[n]$ is constant over $R$. Furthermore, we suppose that $\mu_{n\cdot{}|\Gamma|} = \mu_{n^{2g+1}}$ is constant over $R$.


 As explained in \cite{MR1990670}, the moduli space $\Mm_{\SL_n}^L(X)$ is endowed with a natural $\mu_n$-gerbe $\alpha_{\SL_n,L}$. The definition of \emph{loc. cit.} is stated for the case $R=\mathbb{C}$ but their arguments can be applied to this more general situation with minor modifications. Indeed, one defines $\alpha_{\SL_n,L}$ as the obstruction to the existence of a universal family of Higgs bundles on $\Mm_{\SL_n}^L(X) \times_R X$. That is, the gerbe $\alpha_{\SL_n,L}$ is represented by the morphism of stacks
$$\mathbb{M}_{\SL_n}^L(X) \to \Mm^L_{\SL_n}(X),$$
where the left hand side denotes the stack of stable ($L$-twisted) $\SL_n$-Higgs bundles, and the right hand side is the associated coarse moduli space (or $\mu_n$-rigidification).

Recall that $\Gamma$ denotes the group $R$-scheme $J_X[n]=\Pic^0(X)[n]$ of $n$-torsion points in the Jacobian. By our assumptions this group scheme is constant, and by abuse of notation we will denote again by $\Gamma$ is group of global sections. The group $R$-scheme $\Gamma$ acts on $\Mm_{\SL_n}^L(X)$ by tensoring families of Higgs bundles, and Hausel--Thaddeus observe in \emph{loc. cit.} that the gerbe $\alpha_{\SL_n,L}$ is endowed with a natural $\Gamma$-equivariant structure. We therefore obtain a $\mu_n$-gerbe $\alpha_{\PGL_n,L}$ on $\Mm_{\PGL_n}^d(X)$ by descending the gerbe $\alpha_{\SL_n,L}$ on $\Mm_{\SL_n}^L(X)$ to the quotient $\Mm_{\PGL_n}^d(X) = [\Mm_{\SL_n}^L(X)/\Gamma]$. 

Let now $e\geq 1$ be another integer coprime to $n$. Let $e'=ad$ be a multiple of $d$, which is congruent to $e$ modulo $n$ and $L'=L^{a}$. Thus $L'$ has degree $e'$ and the previous construction yields gerbes $\alpha_{\SL_n,L'}$ on $\Mm_{SL_n}^{L'}(X)$ and $\alpha_{\PGL_n,L'}$ on $\Mm_{\PGL_n}^e$.

\begin{theorem}[Hausel--Thaddeus]\label{thm:HT}
The gerbes $\alpha_{\SL_n,L}$ and $\alpha_{\PGL_n,L'}$ are arithmetic gerbes and there are canonical isomorphisms
$$\mathsf{Split}'((\Mm_{\SL_n}^L)^{\lozenge}/\Aa^{\lozenge},\alpha_{\SL_n,L}^e) \simeq (\Mm_{\PGL_n}^e)^{\lozenge}/\Aa^{\lozenge},$$
$$\mathsf{Split}'((\Mm_{\PGL_n}^e)^{\lozenge}/\Aa^{\lozenge},\alpha_{\PGL_n,L'}^d) \simeq (\Mm_{\SL_n}^L)^{\lozenge}/\Aa^{\lozenge}$$
of $\Pc^\lozenge_{\PGL_n}$- (resp. $\Pc^\lozenge_{\SL_n}$)-torsors.
\end{theorem}

The proof of this result can be found in \cite[Proposition 3.2 \& 3.6]{MR1990670}. It applies mutatis mutandis to the slightly more general context we are working in. We remark that they denote $\mathsf{Split}'$ by $\mathsf{Triv}$. Furthermore their definition of $\mathsf{Triv}$ is a priori via unitary splittings of $\mu_r$-gerbes. However, in the proof they actually argue with the torsor $(\mathsf{Split}_{\mu_r} \times \Pic^{\tau})/{\Pic[r]}$, which corresponds exactly to our definition of $\mathsf{Split}'$ (see Remark \ref{rmk:unitary}).

The following verifies condition (d) of Definition \ref{situation}:
\begin{lemma} \label{Assd}
  Assume that $R=\Oo_F$ for a local field $F$ (and $R$ satisfies the assumptions stated at the beginning of this subsection). For every element $a \in \Aa(\Oo_F) \cap \Aa^{\lozenge}(F)$ with image $a_F \in \Aa(F)$, the fibres $\chi_{\SL_n}^{-1}(a_F)$ and $\chi_{\PGL_n}^{-1}(a_F)$ both have an $F$-rational point if and only if both gerbes $\alpha_{\SL_n,L}^e|_{\chi_{\SL_n}^{-1}(a_F)}$ and $\alpha_{\PGL_n,L'}^d|_{\chi_{\PGL_n}^{-1}(a_F)}$ split.
\end{lemma}
\begin{proof}
The ``if'' direction follows from Theorem \ref{thm:HT}.

Now to the ``only if'' direction. We first notice that the existence of $F$-rational points implies the triviality of $\mathsf{Split}'$ by Theorem \ref{thm:HT}. It follows from Lemma \ref{lemma:splitK}, that the gerbes $\alpha^e_{\SL_n,L}$ and $\alpha^d_{\PGL_n,L'}$ are constant along the Hitchin fibres over $F$. Thus it is enough to check that they are trivial at an arbitrary $F$-rational point.

We start with the $\SL_n$-side: The moduli space $\Mm_{\SL_n}^L(X)$ is smooth over $\Oo_F$, and proper over $\Aa$. Every $x \in \Mm_{\SL_n}^L(X)(F)$, which lies over $a \in \Aa(\Oo_F)$, and extends therefore to an $\Oo_F$-rational point. This shows that $x^*\alpha^e_{\SL_n,L} = 0$, since $\Br(\Oo_F) = 0$.

On the $\PGL_n$-side we also show that for each $a \in \Aa(\Oo_F)$ there exists an $\Oo_F$-rational point in $\chi_{\PGL_n^e}^{-1}(a)$. By Proposition \ref{MPrime} there exists an open substack $\Mm^e_{\PGL_n}(X)'$, which maps to the Hitchin base $\Aa$ by means of a surjective and smooth morphism. Moreover, we can describe it as
$$\Mm^e_{\PGL_n}(X)' \simeq \Pic^{e'}(Y/\Aa)/\Pic^0(X),$$
where $e'$ is an integer depending on our initial choice of $e$, as well as the rank $n$ and genus $g$ of $X$. For a proper (and possibly singular) curve $C$ over a finite field, Picard varieties (of any degree $e'$) always have a rational point. Indeed, $\Pic^{e'}(C)$ is a torsor under the geometrically connected group $\Pic^0(C)$, and the assertion follows from Lang's theorem \cite[Theorem 2]{10.2307/2372673}. This shows the same assertion for curves over $\Oo_F$ (since Picard varieties are smooth, and therefore $k_F$-rational points can be lifted).
\end{proof}

The following summarises the properties of the two moduli spaces together with their natural gerbes established above. Recall that we fixed several assumptions on $R$ at the beginning of this subsection.
\begin{proposition} \label{DualSysVerif}
 The two abstract Hitchin systems $(\Mm_{\SL_n}^L, \Pc_{\SL_n}, \Aa)$ and $(\Mm_{\PGL_n}^e,\Pc_{\PGL_n},\Aa)$ (c.f. Proposition \ref{AbsHSysSummary}) together with the $\mu_n$-gerbes $\alpha_{\SL_n,L}^e$ and $\alpha_{\PGL_n,L'}^d$ form a dual pair of abstract Hitchin systems over $R$.
\end{proposition}
\begin{proof}
  Assumption (a) of Definition \ref{situation} follows directly from the definition of $\Pc_{\PGL_n}$ as $\Pc_{\SL_n}/\Gamma$. We have verified in Proposition \ref{dual} that the natural isogeny $\Pc_{\SL_n} \to \Pc_{\PGL_n}$ is self-dual. This shows (b). Assumption (c) is Hausel--Thaddeus's Theorem \ref{thm:HT}. Assumption (d) is verified by Lemma \ref{Assd} after base change along $R \to \Oo_F$. Assumption (e) holds for $r=n$ by the assumptions on $R$. 
\end{proof}
As a consequence of our main result \ref{Mirror} we now obtain the following theorem, conjectured by Hausel--Thaddeus. For this we fix a prime $\ell$, which is invertible in $k$ and an embedding $\mu_n(k) \into \bar \Qb_\ell$.

\begin{theorem}\label{tmsthm}
Let $X$ be a smooth projective curve of genus $g$ over a base field $k$ endowed with a line bundle $D$, such that $D \otimes (\Omega_X^1)^{-1}$ is of strictly positive degree or equal to $\Oo_X$. Let $n$ be a positive integer, and let $d$ and $e$ be two integers prime to $n$. Let $L \in \mathsf{Pic}^d(X)$ be a line bundle of degree $d$. We assume that $k$ contains a primitive $n^{2g}+1$-th root of unity.
\begin{itemize}
\item[(a)] In case $k = \Cb$ we have the equality of stringy $E$-polynomials
\[E(\Mm_{\SL_n}^L(X);x,y)=E_{\st}(\Mm_{\PGL_n}^e(X),\alpha_{\PGL_n,L'}^d;x,y).\]

\item[(b)] In case $k = \Fb_q$ is a finite field of characteristic $p > n$ we have an equality of stringy point counts
\[\#(\Mm_{\SL_n}^L(X)(k)) = \#_{\mathsf{st}}^{\alpha_{\PGL_n,L'}^d}(\Mm^e_{\PGL_n}(X)).\]
\end{itemize}
\end{theorem}

\begin{proof}
To prove (a) we spread out: We choose a subring $R \subset \Cb$, which is finitely generated over $\Zb$, contains $1/n$, and all $n^{2g}$-th roots of unity, such that $X$, $D$ and $L$ extend to $\Spec(R)$. After replacing $R$ by an \'etale extension we may also assume that $\Gamma$ is constant over $R$. Then assertion (a) is a consequence of Theorem \ref{Mirror} applied to the dual pair of abstract Hitchin systems given by Proposition \ref{DualSysVerif}. Here we use that since $\Mm_{\SL_n}^L(X)$ is a smooth scheme the stringy Hodge number $h^{p,q}_{c,\mathsf{st}}(\Mm_{\SL_n}^L(X),\alpha_{\SL_n,L}^e)$ is equal to $h^{p,q}_c(\Mm_{\SL_n}^L(X))$.  

Similarly, assertion (b) follows from Theorem \ref{p-Mirror} using Proposition \ref{DualSysVerif}. 
\end{proof}
\begin{corollary}
  In the situation of Theorem \ref{tmsthm} (a), in case $D$ is the canonical line bundle, we have an equality of Hodge numbers
$h^{p,q}_c(\Mm_{\SL_n}^L(X)) = h^{p,q}_{c,\mathsf{st}}(\Mm^e_{\PGL_n}(X),\alpha_{\PGL_n,L'}^d).$
\end{corollary}
\begin{proof}
The Hodge structures on the cohomology groups $H^*_\text{c}(\Mm_{\SL_n}^L(X))$ and $H^*_{\text{c},\st}(\Mm_{\PGL_n}^e(X),\alpha_{\PGL_n,L'}^d)$ are pure in this case \cite[Section 6]{MR1990670} and hence the equality of Hodge numbers follows from the equality of $E$-polynomials.
\end{proof}

There is a second cohomological result, related to the work of Ng\^o \cite{MR2653248}, which we can deduce from our work. Let $R$ be a finite field $k$ of characteristic $p > n$. 

We denote by $\I\chi_{\PGL_n}$ the morphism $\I\Mm_{\PGL_n}^e \to \Aa$ induced by $\chi_{\PGL_n}\colon \Mm_{\PGL_n}^e \to \Aa$ and by $N_{L'}$ the $\overline{\Qb}_{\ell}$-local system on $I \Mm_{\PGL_n}^e$ induced by the $\mu_n$-torsor on $\I\Mm_{\PGL_n}^e$ associated to the $\mu_n$-gerbe $\alpha_{\PGL_n,L'}^d$ (see Subsection \ref{gerbes}) by means of our chosen embedding $\mu_n(k) \hookrightarrow \overline{\Qb}_{\ell}$. Finally $N_{L'}(F_{\PGL_n})$ denotes the Tate twist of $N_{L'}$ by the fermionic shift.

Let $X$ be a smooth projective curve over a finite base field $k$ endowed with a line bundle $D$, such that $D \otimes (\Omega_X^1)^{-1}$ is of strictly positive degree or equal to $\Oo_X$. Let $n$ be a positive integer, and let $d$ and $e$ be two integers prime to $n$. We denote by $L \in \mathsf{Pic}^e(X)$ a line bundle of degree $e$. We assume that $R=k$ satisfies the assumptions stated at the beginning of this subsection.
\begin{theorem} 
For every $a \in \Aa(k)$ the $\Gal(k)$-representations $(R(\chi_{\SL_n})_*\overline{\Qb}_{\ell})_a$ and $(R(\I\chi_{\PGL_n})_*N_{L'}(F_{\PGL_n}))_a$ are abstractly isomorphic.
\end{theorem}

\begin{proof}
The two complexes of $\ell$-adic sheaves $R(\chi_{\SL_n})_*\overline{\Qb}_{\ell}$ and $R(\I\chi_{\PGL_n})_*N_{L'}$ are pure. In the first case this follows from Verdier duality, the fact that the map $\chi_{\SL_n}$ is proper, and Deligne's purity theorem \cite{PMIHES_1980__52__137_0}. In the second case one applies the same argument to every stratum of the inertia stack (which are moduli stacks of Higgs bundles).

It is hence sufficient to establish an equality of point counts
$$\#(\chi_{\SL_n})^{-1}(a)(k) = \mathsf{Tr}(\Fr,(R(\I\chi_{\PGL_n})_*N_{L'})_a).$$
This is a special case of the second assertion of Theorem \ref{p-Mirror}.
\end{proof}

\subsection{The action of $\Gamma$ on the cohomology of $\Mm_{\SL_n}^L(X)$}\label{cohoact}
 We can also prove the following refined version of Theorem \ref{tmsthm}, which describes the action of the group $\Gamma$ on the cohomology of the $\SL_n$-moduli space. 

We continue to use the setup of the previous subsection in the case where the base ring $R$ is a finite field $k$, which contains all $n^{2g+1}$-th roots of unity. We choose a primitive $n$-th root of unity $\zeta \in k^{\times}$. Furthermore, we will assume that $\Gamma$ is a constant \'etale group scheme over $k$.


The Jacobian $J_X$ of $X$ comes with a canonical isomorphism $J_X \cong J_X^\vee$ under which the isogeny $[n]\colon J_X \to J_X$ is self-dual. Hence as a special case of \eqref{SelfDualKernel} we obtain an isomorphism
\begin{equation}
  \label{GammaSelfDual}
  \Gamma \cong \Gamma^\vee.
\end{equation}

By our assumptions $\Gamma$ is constant, and hence by the assumptions on $k$ the group scheme $\Gamma^\vee$ is constant as well with value $\Hom(\Gamma,\mu_n(k))$. (Recall that we write $\Gamma$ also for the group of global sections $\Gamma(k)$.) Using \eqref{GammaSelfDual} and our chosen root of unity $\zeta$ we obtain a non-degenerate pairing of abstract groups
\begin{equation} \label{GammaPairing}
  (\;,\; )\colon \Gamma \times \Gamma \cong \Gamma \times \Hom(\Gamma,\mu_n(k)) \to \mu_n(k) \cong \Zb/n \Zb \subset \Qb/\Zb
\end{equation}
where we identify $\Gamma$ with $\Gamma^*$.

A result similar to the one below has been conjectured by Hausel in \cite{HauselMiraflores}. The strategy is to apply Theorem \ref{tmsthm}(b), and compute the arithmetic Fourier transforms with respect to the finite group $\Gamma$ of both sides.

\begin{theorem}\label{miraflores} Let $X$ be a smooth projective curve over a finite field $k$ endowed with a line bundle $D$, such that $D \otimes (\Omega_X^1)^{-1}$ is of strictly positive degree or equal to $\Oo_X$. Let $d$ and $e$ be two integers prime to $n$. We denote by $L \in \mathsf{Pic}^d(X)$ a line bundle of degree $d$ and let $L'$ be a power of $L$ of degree congruent to $e$ modulo $n$. We assume that $k$ satisfies the properties described at the beginning of this subsection.
For any $\gamma \in \Gamma$ and $\chi \in \Gamma^*$, which correspond to each other under the identification $\Gamma \cong \Gamma^*$ given by the pairing \eqref{GammaPairing} we have  
\begin{equation}\label{refinedtms} 
  \Tr(\Fr,H^*_c(\Mm^{L}_{\SL_n}(X),\overline{\Qb}_{\ell})_{\chi}) =  \Tr(\Fr,H^*_c([(\Mm_{\SL_n}^{L'}(X))^\gamma/\Gamma],N_{L'}(F_{\PGL_n}))).
\end{equation}
\end{theorem}

\begin{proof} 
As in the proof of Lemma \ref{lemma:twisty}, given $L_1,L_2 \in \mathsf{Pic}^d(X)(k)$, which differ by an element of the form $L_0^n$ for some $L_0 \in  \mathsf{Pic}^0(X)(k)$, we have an isomorphism $\Mm^{L_1}_{\SL_n} \cong \Mm^{L_2}_{\SL_n}$ given by tensoring with $L_0$. Hence up to isomorphism $\Mm^{L_1}_{\SL_n}$ only depends on the class of $[L_1]$ in $\mathsf{Pic}^d(X)(k)/(\mathsf{Pic}^0(X)(k)^n)$. We denote the corresponding moduli space for $L_1=L$ by $\Mm^{[0]}_{\SL_n}$. 

More precisely, let $\delta$ be the isomorphism
\begin{equation*}
  \mathsf{Pic}^0(X)(k)/(\mathsf{Pic}^0(X)(k))^n \cong H^1(k,\Gamma) \cong \Gamma.
\end{equation*}
Then the above shows that every other moduli space $\Mm^{L_2}_{\SL_n}$ can be obtained by twisting $\Mm^{[0]}_{\SL_n}$ by the element $\delta(L L_2^{-1})$ as in Lemma \ref{lemma:twisty}. We write the twist by $\nu \in \Gamma$ as $\Mm^{[\nu]}_{\SL_n}$. 

By Proposition \ref{prop:twists} there is for every $\nu \in \Gamma$ an isomorphism $\Mm^{[\nu]}_{\SL_n} \cong \Mm^{[0]}_{\SL_n}$ over an algebraic closure of $k$ under which Frobenius gets sent to $\nu^{-1} \circ \Fr$. Using this the $\ell$-adic cohomology of twists can be understood completely: For $\chi \in \Gamma^*$ we compute
\begin{align} \label{somesum}
  \sum_{\nu \in \Gamma} \Tr(\Fr,H^*_c(\Mm^{[\nu]}_{\SL_n},\overline{\Qb}_{\ell}))\chi(\nu) &= \sum_{\nu \in \Gamma} \Tr(\nu^{-1}\circ \Fr,H^*_c(\Mm^{[0]}_{\SL_n},\overline{\Qb}_{\ell}))\chi(\nu)\\
 & = \sum_{\chi'\in\Gamma^*} \sum_{\nu\in\Gamma} \Tr(\nu^{-1}\circ\Fr,H^*_c(\Mm^{[0]}_{\SL_n},\overline{\Qb}_{\ell})_{\chi'})\chi(\nu). \nonumber
\end{align}

By definition $\nu^{-1}$ acts on $H^*_c(\Mm^{[0]}_{\SL_n},\overline{\Qb}_{\ell})_{\chi'}$ by $\chi'(\nu^{-1})$. This together with a character sum argument shows 
\begin{align*} \sum_{\chi'\in\Gamma^*} \sum_{\nu\in\Gamma} \Tr(\nu^{-1}\circ\Fr,H^*_c(\Mm^{[0]}_{\SL_n},\overline{\Qb}_{\ell})_{\chi'})\chi(\nu)& = \sum_{\chi'\in\Gamma^*} \Tr(\Fr,H^*_c(\Mm^{[0]}_{\SL_n},\overline{\Qb}_{\ell})_{\chi'})\sum_{\nu\in\Gamma} \chi'(\nu^{-1})\chi(\nu)\\
& = |\Gamma| \Tr(\Fr,H^*_c(\Mm^{[0]}_{\SL_n},\overline{\Qb}_{\ell})_{\chi}),\end{align*}
which is exactly the left hand side of \eqref{refinedtms} up to the factor $|\Gamma|$.

For the right hand side of \eqref{refinedtms} we will show in Lemma \ref{lemma:twistandshout2} below that for elements $L_1$ and $L_2$ of $\mathsf{Pic}^e(X)(k)$ inducing $\nu=\delta(L_1 L_2^{-1})\in \Gamma$ and for any $\gamma'\in\Gamma$ we have
\begin{equation}\label{twistandshout}\Tr(\Fr,H^*_c([(\Mm_{\SL_n}^{L_1})^{\gamma'}/\Gamma],N_{L'_1}(F_{\PGL_n}))) = (\nu,\gamma')^{-1} \Tr(\Fr,H^*_c([(\Mm_{\SL_n}^{L_2})^{\gamma'}/\Gamma],N_{L'_2}(F_{\PGL_n}))).\end{equation}
Assuming this, we can conclude as follows: For every $\nu \in \Gamma$ let $L_\nu$ be an element of $\mathsf{Pic}^d(X)(k)$ such that $\delta(L L_\nu^{-1})=\nu$ and let $L'_\nu$ be a power of $L_\nu$ of degree congruent to $e$ modulo $n$ as in Subsection \ref{conclusion}. Then using Theorem \ref{tmsthm} (b) and \eqref{twistandshout} we can expand the left hand side of \eqref{somesum} for $\chi \in \Gamma^*$ corresponding to $\gamma \in \Gamma$ as follows:
\begin{align*}
\sum_{\nu \in \Gamma} \Tr(\Fr,H^*_c(\Mm^{[\nu]}_{\SL_n},\overline{\Qb}_{\ell}))\chi(\nu) 
&= \sum_{\nu \in \Gamma} \Tr(\Fr, H^*_c(\Mm^e_{\PGL_n},\alpha^d_{\PGL_n,L'_{\nu}})) \chi(\nu) \\
&=\sum_{\nu \in \Gamma} \sum_{\gamma' \in \Gamma} \Tr(\Fr, H^*_c([(\Mm_{\SL_n}^{L'_\nu})^{\gamma'}/\Gamma],N_{L'_\nu}(F_{\PGL_n}))) \chi(\nu) \\
&= \sum_{\nu \in \Gamma }  \sum_{\gamma' \in \Gamma}  \Tr(\Fr,H^*_c([(\Mm_{\SL_n}^{L'})^{\gamma'}/\Gamma],N_{L'}(F_{\PGL_n}))\chi(\nu)(\nu,\gamma')^{-1} \\
&= |\Gamma| \Tr(\Fr,H^*_c([(\Mm_{\SL_n}^{L'})^{\gamma}/\Gamma],N_{L'}(F_{\PGL_n})))
\end{align*}
This is the right hand side of \eqref{refinedtms} and thus proves the assertion.

It remains to establish \eqref{twistandshout}. To do this we translate the problem into one of $p$-adic integrals. We choose a local field $F$ with residue field $k$ and a lift to $\Oo_F$ our curve $X$ together with all the appearing line bundles on $X$. 

First let $a \in \Aa^{\flat}(\Oo_F)$ be a point with image $a_F$ in $\Aa^\lozenge(F)$. We denote the abelian varieties $\Pc_{\SL_n,a_F}$, respectively $\Pc_{\SL_N,a_F}/\Gamma=\Pc_{\PGL_n,a_F}$ over $F$ by $A$, respectively $B$ and consider the associated long exact sequence of locally compact abelian groups
$$0 \to \Gamma \to A(F) \to B(F) \xrightarrow{\beta} H^1(F,\Gamma) \xrightarrow{\alpha} H^1(F,A) \to H^1(F,B) \to H^2(F,\Gamma) \to 0$$
from Construction \ref{longseq}. We consider the Tate duality pairings
\begin{equation} \label{longseqagain}
  \langle \;,\; \rangle \colon H^1(F,A) \times B(F) \to \Qb/\Zb
\end{equation}
and
\begin{equation*}
  \langle \;,\; \rangle_\Gamma \colon H^1(F,\Gamma) \times H^1(F,\Gamma) \to \Qb/\Zb.
\end{equation*}

Using the self-duality of the isogeny $A \to B$ given by Proposition \ref{dual}, Lemma \ref{longseqduality} gives a canonical isomorphism of the long exact sequence \eqref{longseqagain} with its own Pontryagin dual. As part of this isomorphism there is an identity
\begin{equation}\label{weilpairing}
\langle \alpha(T),b \rangle = \langle T, \beta(b)  \rangle_\Gamma.
\end{equation}
for elements $b \in B(F)$ and $T \in H^1(F,A)$. As explained above, the moduli space $\Mm^{L_2}_{\SL_n}(X)$ is the twist of $\Mm^{L_1}_{\SL_n}(X)$ by $\nu \in H^1(\Oo_F,\Gamma)$, so that $T_{L_2}=\alpha(\nu) \cdot T_{L_1}$ in $H^1(F,A)$. Using this and \eqref{weilpairing} we find
\begin{equation}
  \label{pairingeq}
  \langle T_{L_1},b \rangle = \langle \nu, \beta(b) \rangle_\Gamma^{-1} \cdot \langle T_{L_2},b \rangle
\end{equation}
for $b \in B(F)$.

\begin{lemma}\label{lemma:twistandshout2}
We have $$\Tr(\Fr,H^*_c([(\Mm_{\SL_n}^{L_1})^{\gamma'}/\Gamma],N_{L'_1}(F_{\PGL_n}))) = (\nu,\gamma')^{-1} \Tr(\Fr,H^*_c([(\Mm_{\SL_n}^{L_2})^{\gamma'}/\Gamma],N_{L'_2}(F_{\PGL_n}))).$$
\end{lemma}

\begin{proof}
The left hand side can be computed by means of the Grothendieck-Lefschetz trace formula as 
$$\sum_{x \in [(\Mm_{\SL_n}^{L_1})^{\gamma'}/\Gamma](k_F)} \frac{\Tr(\Fr_x,N_{L'_1}(F_{\PGL_n})))}{|\Aut(x)|},$$
which in turn can be understood as the $p$-adic integral (see Corollary \ref{StringyF}) of the function $f_{\alpha_{L_1}}$ given by Construction \ref{fi} on the subset $e^{-1}([(\Mm_{\SL_n}^{L_1})^{\gamma'}/\Gamma](k_F))$, where $e$ is the specialisation map of Construction \ref{SpecMap}. The same description exists for the right hand side.

Lemma \ref{lemmaTate} describes the functions $f_{\alpha_{L_i}}$ in terms of the Tate duality pairing between $B(F)$ and $H^1(F,A)$. We can therefore apply \eqref{pairingeq} to compare them.

The torsor $\beta(b)$ appearing corresponds to an element $(t,\gamma'') \in \Gamma \oplus \Gamma$. By definition of the specialisation map $e\colon \Mm_{\PGL_n}^e(\Oo_F)^{\flat} \to \I\Mm_{\PGL_n}^e(k_F)$ we have $e(b) \in [(\Mm_{\SL_n}^{L'_i})^{\gamma''}/\Gamma]$ so that $\gamma'=\gamma''$. Using this, the fact that $\nu \in H^1(\Oo_F,\Gamma)$, Lemma \ref{ConstantTateDuality} and the last claim of Proposition \ref{dual}, it follows that $\langle \nu, \beta(b) \rangle_\Gamma = (\nu,\gamma')$. This allows us to compare the left and right hand sides.
\end{proof}
The lemma above concludes the proof of Theorem \ref{miraflores}.
\end{proof}

In the sequel \cite{GWZ18} to this article we will revisit the strategy employed above to give a new proof of Ng\^o's Geometric Stabilisation Theorem \cite[Th\'eor\`eme 6.4.2]{MR2653248}. 


\bibliographystyle{amsalpha}
\bibliography{master}
\end{document}